\documentclass{article}
\usepackage{graphicx}

\newif\ifarxiv

\usepackage{graphicx}
\usepackage{amsfonts,amsmath,amssymb,amsthm}

\usepackage{algorithm,algorithmic}
\usepackage{comment}
\usepackage[numbers,sort]{natbib}
\usepackage{pifont}
\usepackage{booktabs}
\usepackage[dvipsnames,table]{xcolor}
\usepackage{makecell}
\usepackage{multirow}
\usepackage{cancel}
\usepackage{tabularx}
\usepackage{subcaption}
\usepackage{wrapfig}
\usepackage{fontawesome}
\usepackage[bottom]{footmisc}

\usepackage{tcolorbox}
\usepackage{xparse}

\newcommand{\cmark}{{\color{Green}\ding{51}}}%
\newcommand{\xmark}{{\color{Red}\ding{55}}}
\newcommand{\qmark}{{\color{orange}\textbf{?}}}

\newcommand{\Rhat}[0]{\widehat{\cR}}

\newcommand{\lambdahat}[0]{\widehat\lambda}
\newcommand{\eps}[0]{\varepsilon}

\DeclareMathOperator*{\EE}{\mathbb{E}}
\DeclareMathOperator*{\argmin}{\arg\min}

\DeclareMathOperator*{\PP}{\mathbb{P}}

\newcommand{\cFhat}[0]{\widehat{\cF}}
\newcommand{\prunedERM}[0]{\widehat{f}_{\operatorname{PC}}}
\newcommand{\cFstar}[0]{\cF^\star}

\usepackage{color}
\definecolor{color1}{HTML}{105e8a}
\definecolor{color2}{HTML}{e99926}
\definecolor{color3}{HTML}{b82a0c}
\definecolor{color4}{HTML}{3e8a10}
\definecolor{color5}{HTML}{80037e}
\definecolor{color6}{HTML}{070707}

\definecolor{c1}{HTML}{a6611a}
\definecolor{c2}{HTML}{dfc27d}
\definecolor{c3}{HTML}{f5f5f5}
\definecolor{c4}{HTML}{80cdc1}
\definecolor{c5}{HTML}{018571}

\definecolor{b1}{HTML}{7fc97f}
\definecolor{b2}{HTML}{beaed4}
\definecolor{b3}{HTML}{fdc086}
\definecolor{b4}{HTML}{ffff99}
\definecolor{b5}{HTML}{386cb0}

\usepackage{graphicx} 
\usepackage{amsfonts,amsmath,amssymb,amsthm}
\usepackage{microtype}
\usepackage{nicefrac}
\usepackage[pagebackref]{hyperref}
\hypersetup{
    colorlinks,
    linkcolor={black},
    citecolor={color3},
    urlcolor={color5},
}
\usepackage{enumitem}

\usepackage{algorithm}

\usepackage[bottom]{footmisc}

\usepackage{tikz}
\usetikzlibrary{arrows.meta,positioning}
\usetikzlibrary{calc}
\usepackage{pgfplots}
\pgfplotsset{compat=1.18} 

\usepackage[noabbrev, capitalize, nameinlink]{cleveref}
\DeclareRobustCommand{\abbrevcrefs}{%
\crefname{theorem}{Thm.}{Thms.}%
\crefname{proposition}{Prop.}{Props.}%
\crefname{corollary}{Cor.}{Cors.}%
}

\DeclareRobustCommand{\cshref}[1]{{\abbrevcrefs\cref{#1}}}

\newtheoremstyle{thmstyle}
  {6pt} 
  {2pt} 
  {\itshape} 
  {} 
  {\bfseries} 
  {.} 
  {.5em} 
  {} 

\newtheoremstyle{defstyle}
  {6pt} 
  {2pt} 
  {} 
  {} 
  {\bfseries} 
  {.} 
  {.5em} 
  {} 

\newtheoremstyle{remarkstyle}
  {6pt}       
  {2pt}       
  {}          
  {}          
  {\itshape}  
  {.}         
  {.5em}      
  {}          

\theoremstyle{thmstyle}
\newtheorem{theorem}{Theorem}
\newtheorem*{theorem*}{Theorem}
\crefname{theorem}{Theorem}{Theorems}
\newtheorem{lemma}{Lemma}
\newtheorem{proposition}{Proposition}
\crefname{proposition}{Proposition}{Propositions}
\newtheorem{corollary}{Corollary}
\crefname{corollary}{Corollary}{Corollaries}

\theoremstyle{defstyle}
\newtheorem{example}{Example}
\newtheorem{definition}{Definition}

\theoremstyle{remarkstyle}
\newtheorem{remark}{Remark}

\newcommand{\nrm}[1]{\left\|#1 \right\|}
\newcommand{\prn}[1]{\left(#1 \right)}
\newcommand{\brk}[1]{\left[#1 \right]}
\newcommand{\crl}[1]{\left\{#1 \right\}}
\newcommand{\abs}[1]{\left|#1 \right|}
\newcommand{\floor}[1]{\left\lfloor #1 \right\rfloor}
\newcommand{\ceil}[1]{\left\lceil #1 \right\rceil}

\newcommand{\RR}[0]{\mathbb{R}}
\newcommand{\NN}[0]{\mathbb{N}}
\newcommand{\ZZ}[0]{\mathbb{Z}}

\newcommand{\cA}[0]{\mathcal{A}}

\newcommand{\cH}[0]{\mathcal{H}}
\newcommand{\cE}[0]{\mathcal{E}}
\newcommand{\cF}[0]{\mathcal{F}}

\newcommand{\cP}[0]{\mathcal{P}}
\newcommand{\cB}[0]{\mathcal{B}}

\newcommand{\cR}[0]{\mathcal{R}}
\newcommand{\cX}[0]{\mathcal{X}}
\newcommand{\cY}[0]{\mathcal{Y}}

\newcommand{\one}{\mathbf 1}

\DeclareMathOperator{\uniformOp}{Uniform}
\newcommand{\uniform}[1]{\uniformOp\prn{#1}}
\DeclareMathOperator{\conv}{conv}

\DeclareMathOperator*{\Var}{Var}

\newcommand{\erm}[0]{\widehat{f}_{\mathrm{ERM}}}
\newcommand{\starestimator}[0]{\widehat{f}_{\star}}
\newcommand{\Qestimator}[0]{\widehat{f}_{Q}}
\newcommand{\AEWestimator}[0]{\widehat{f}_{\mathrm{ew}}}

\newcommand{\Sequential}[0]{\widehat{f}_{\operatorname{SQ}}}

\newcommand{\starhull}[0]{\mathrm{star}}

\newcommand{\fhat}[0]{\widehat{f}}
\newcommand{\fstar}[0]{f^{\star}}

\DeclareMathOperator{\KL}{KL}

\newcommand{\excessRisk}[0]{\cE_{(P,\cF)}}

\newcommand{\excessRiskPar}[2]{\cE_{(#1,#2)}}

\newcommand{\rhohat}[0]{\widehat{\rho}}

\newcommand{\Deltahat}[0]{\widehat{\Delta}}

\newcommand{\Bernoulli}[0]{\operatorname{Bernoulli}}
\newcommand{\Binomial}[0]{\operatorname{Binomial}}

\newcommand{\PM}[0]{\fhat_{\operatorname{PM}}}
\newcommand{\EW}[0]{\fhat_{\operatorname{EW}}}
\newcommand{\BOA}[0]{\fhat_{\operatorname{BOA}}}
\newcommand{\midpoint}[0]{\fhat_{\circ}}

\usepackage{amsmath,amsthm,amssymb}
\usepackage{bbm}
\DeclareMathOperator*{\ee}{\mathbb{E}}
\DeclareMathOperator*{\p}{\mathbb{P}}

\newcommand{\ind}{\mathbf{1}}
\renewcommand{\d}{\mathop{}\!\mathrm{d}}

\makeatletter
\newcommand{\alphacmd@factory}[1]{}
\newcounter{alphacmdcounter}
\newcommand{\GenerateAlphabetCmds}[2]{%
    \renewcommand{\alphacmd@factory}[1]{%
        \expandafter\providecommand\csname #1##1\endcsname{{#2{##1}}}%
    }
    \setcounter{alphacmdcounter}{0}
    \loop
        \stepcounter{alphacmdcounter}
        \edef\alphacmd@ID{\@Alph\c@alphacmdcounter}
        \expandafter\alphacmd@factory\alphacmd@ID
    \ifnum\thealphacmdcounter<26
    \repeat
}
\newcommand{\GenerateAlphabetCmdsLower}[2]{%
    \renewcommand{\alphacmd@factory}[1]{%
        \expandafter\providecommand\csname #1##1\endcsname{{#2{##1}}}%
    }
    \setcounter{alphacmdcounter}{0}
    \loop
        \stepcounter{alphacmdcounter}
        \edef\alphacmd@ID{\@alph\c@alphacmdcounter}
        \expandafter\alphacmd@factory\alphacmd@ID
    \ifnum\thealphacmdcounter<26
    \repeat
}
\makeatother

\GenerateAlphabetCmds{c}{\mathcal}
\GenerateAlphabetCmdsLower{c}{\mathcal}

\GenerateAlphabetCmds{r}{\mathbf}
\GenerateAlphabetCmdsLower{r}{\mathbf}

\newcommand{\Amini}[0]{\cA_{\operatorname{mini}}}
\newcommand{\Aexp}[0]{\cA_{\exp}}
\newcommand{\Aerm}[0]{\cA_{\operatorname{ERM}}}

\newcommand{\Ybar}[0]{\overline{Y}}

\newcommand{\Deltamin}[0]{\Delta_{\min}}

\newcommand{\threshold}[0]{\tau_{m,M}(\delta)}

\newcommand{\ahat}[0]{\widehat{a}}
\newcommand{\dhat}[0]{\widehat{d}}
\newcommand{\Psihat}[0]{\widehat{\Psi}}

\newcommand{\Esep}[0]{E_{\operatorname{sep}}}
\newcommand{\fM}[0]{\mathfrak{M}}

\newcommand{\kl}[0]{\operatorname{kl}}

\newcommand{\fmini}[0]{f_{\operatorname{mini}}}
\newcommand{\fexp}[0]{f_{\operatorname{exp}}}
\newcommand{\pimini}[0]{\pi_{\operatorname{mini}}}
\newcommand{\piexp}[0]{\pi_{\operatorname{exp}}}
\newcommand{\Qbob}[0]{\fhat_{\operatorname{QBOB}}}

\newcommand{\rhohatQ}[0]{\rhohat_{Q}}
\newcommand{\rhohatQun}[0]{\rhohat_{Q,\operatorname{un}}}

\newcommand{\Thetainf}[0]{\Theta_{\operatorname{inf}}}

\newcommand{\Pall}[0]{\cP_{\operatorname{all}}}

\makeatletter
\newcommand\blfootnote[1]{%
  \begingroup
  \def\@thefnmark{}
  \renewcommand\@makefnmark{}
  \@footnotetext{#1}%
  \endgroup
}
\makeatother

\newcommand{\varphitilde}[0]{\widetilde{\varphi}}

\arxivtrue
\usepackage{authblk}

\usepackage{geometry}
\usepackage{titletoc}
\titlecontents{section}
  [2em]
  {\vspace{-0.2em}}
  {\contentslabel{2em}}
  {}
  {\titlerule*[0.4pc]{.}\contentspage}

\titlecontents{subsection}
  [4.5em]
  {\vspace{-0.5em}\small}
  {\contentslabel{2.5em}}
  {}
  {\titlerule*[0.4pc]{.}\contentspage}

\usepackage{parskip}

\title{Reconciling Universal and Uniform Learning with $Q$-Aggregation}

\author[1]{Mikael M{\o}ller H{\o}gsgaard}
\author[1]{Patrick Rebeschini}
\author[2]{Tobias Wegel}
\affil[1]{Department of Statistics, University of Oxford}
\affil[2]{Department of Computer Science, ETH Zurich}

\date{}

\begin{document}

\maketitle
\blfootnote{Authors are listed alphabetically.}

\begin{abstract}
  \noindent We study regression under bounded responses in terms of excess mean squared error. When the comparator class is finite, this setting is known as model selection aggregation, and achieving minimax excess risk requires improper learning algorithms. Contrary to this, in the universal learning framework no improperness is needed, as simple empirical risk minimization achieves the best-possible exponential learning rate. Hence, the two frameworks suggest different optimal algorithmic principles. This poses the question of best-of-both-worlds guarantees: Are minimax and universal exponential rates achievable by the same algorithm?
For finite hypothesis classes, we answer this question in the affirmative by showing that the $Q$-aggregation estimator\textemdash which is known to achieve minimax optimal tails\textemdash achieves exponential universal rates. A wide range of other estimators and algorithmic principles (ERM, sequential averaging, pruning, and star estimation) do not achieve both.
For countably infinite hypothesis classes, we answer the question in the negative by showing that there is an inherent trade-off between achieving exponential universal and minimax uniform rates. This trade-off is exactly traced by combining optimal algorithms from each world using $Q$-aggregation.
Besides these results, we prove several additional structural results about universal rates in learning with squared loss.

\end{abstract}

\startcontents[main]
\printcontents[main]{}{1}{\section*{Contents}}

\section{Introduction}
The classical \emph{model selection aggregation} or \emph{dictionary learning} problem is the following: Given a dictionary $\cF=\crl{f_1,f_2,\ldots}$ of functions $\cX\to [0,1]$ and $n$ i.i.d.\ samples $\rS=(X_i,Y_i)_{i=1}^n\sim P^{n}$ from a distribution $P$ on $\cX\times [0,1]$, find a function $\fhat$ that achieves low excess risk under the squared loss with respect to the dictionary $\cF$,
\begin{equation*}
    \excessRisk(\fhat) = \EE_{(X,Y)\sim P} \brk{(\fhat(X)-Y)^2} -\inf_{f\in \cF}\EE_{(X,Y)\sim P} \brk{(f(X)-Y)^2}.
\end{equation*}
Here, the first term is the risk of $\fhat$, denoted $\cR_P(\fhat)$, and the second term is the smallest achievable risk in $\cF$. The predictor $\fhat$ is obtained from the sample $\rS$ via an algorithm $\cA$, so that $\fhat=\cA(\rS)$.

This problem encompasses many fundamental learning problems, such as hyperparameter sweeping, weight averaging, ensembling, and, more recently, the combination of foundation models for prediction problems. It also yields a rich geometric theory and calls for a range of algorithmic primitives.

In the case of a finite hypothesis class $\cF=\crl{f_1,\ldots,f_M}$, it is well known \cite{tsybakov2003optimal} that the \emph{minimax} rate in this problem, also called the \emph{optimal rate of aggregation}, is given by
\begin{equation}
\label{eqn:finite-minimax-rate}
    \forall n\in\NN,\, \forall\delta\in(0,1):\qquad \inf_{\cA}\ \sup_{\cF, P}\ \PP_{\rS\sim P^{ n}}\prn{\excessRisk(\cA(\rS)) \gtrsim \min\crl{1,\frac{\log (M/\delta)}{n}}} \gtrsim \delta.
\end{equation}
Achieving the rate $\log(M)/n$ in expectation does not imply achieving the guarantee \eqref{eqn:finite-minimax-rate} in probability \cite{audibert2007progressive}, but \cref{eqn:finite-minimax-rate} implies expected excess risk bounded by $\lesssim\log(M)/n$.
It is known that any (randomized) proper algorithm, that is, an algorithm that (randomly) selects a function from the dictionary $\cA(\rS)\equiv\fhat\in \cF$, has expected excess risk lower-bounded by $\gtrsim \sqrt{\log (M)/n}$ in the worst case, making it suboptimal by a square-root. This result includes any empirical risk minimizer over $\cF$.
The key reason for the lower bound is the \emph{uniform} nature of the results: the worst-case distribution and dictionary are allowed to depend on the sample size $n$, and as the sample size grows, the best hypothesis can become more indistinguishable from an information-theoretic viewpoint, meaning that any algorithm trying to select a single function will fail occasionally.

Achieving the minimax-optimal rate in deviation requires moving beyond (randomized) proper learning and employing some form of \emph{hedging}, i.e., combining the predictions of multiple functions (see the discussion of existing methods in \cref{subsec:related-works}). It is not immediately clear why hedging should help, as it might assign mass to suboptimal hypotheses, the same issue faced by the proper learner.
The crucial distinction is that, under the squared loss, averaging benefits from the strong convexity of the loss. To see this, let $\rho$ be a distribution on $\{1,\ldots,M\}$. The improper aggregate predicts $f_\rho(\cdot):=\EE_{k\sim\rho}\brk{f_k(\cdot)}$, whereas a randomized proper learner draws an index $k\sim\rho$ and predicts $f_k(\cdot)$. In expectation, their losses are related through the following Jensen's gap (the difference between the two terms in Jensen's inequality):
\begin{equation}
\label{eqn:Jensen-gap}
    \underbrace{\cR_P(f_\rho)}_{\text{aggregated}} = \underbrace{\EE_{k\sim \rho} \brk{\cR_P(f_k)}}_{\text{randomized}}- \underbrace{\EE_{X\sim P_X} \brk{\Var_{k\sim \rho} \brk{f_k(X)}}}_{\text{Jensen's gap}}.
\end{equation}
The aggregated estimator, in contrast to the randomized proper learner, benefits from the non-positive Jensen's gap term: aggregation turns the uncertainty that may cause a randomized proper learner to select a suboptimal hypothesis into an offset that compensates for the mass assigned to such hypotheses.
The equality above is a special property of the squared loss\footnote{For strongly convex losses, an analogous relation holds with a $ \leq $ inequality and a constant factor in front of the variance.}, and taking the variance term into account when designing aggregation mechanisms is fundamental to achieving optimal rates in the model selection aggregation problem.
From now on, in this work, we only consider \emph{aggregated} estimators, and all our algorithms are deterministic.

Recent studies of \emph{universal learning} \cite{attias2024universal,bousquet2021theory,hanneke2024erm} depart from the uniform minimax benchmark in \cref{eqn:finite-minimax-rate}, viewing its worst-case guarantee as overly pessimistic and advocating a more realistic \emph{universal} perspective, which asks for the common (universal) rate of the excess risk, \emph{pointwise} in the distribution across all distributions, rather than a common rate bound across all distributions. These works show\footnote{Technically, the literature has not considered agnostic (that is, without assumption on the relationship between hypothesis class and distribution) regression with squared loss. However, a straightforward argument (\cref{prop:ERM-exponential}) demonstrates that the same is true in this setting.} that for finite dictionaries, empirical risk minimization (ERM) on the dictionary achieves, \emph{for every fixed distribution}, zero excess risk with probability exponential in the sample size, that is,
\begin{equation}
\label{eqn:universal-rate}
    \forall P, \cF,\quad  \exists c,C>0, \quad \forall n\in \NN: \qquad \PP_{\rS\sim P^{ n}}\prn{\excessRisk(\cA(\rS))> 0} \leq Ce^{-cn}.
\end{equation}
\cref{eqn:universal-rate} implies that the expected excess risk is bounded by an exponentially vanishing rate (\cref{lem:exponential-implications}).
As established in \cref{lem:exponential-is-optimal}, the fastest universal rate attainable for any nontrivial hypothesis class is, as in \cref{eqn:universal-rate}, exponential. This makes \cref{eqn:universal-rate} the natural benchmark for algorithms in the universal rates setting for finite dictionaries.

In contrast to the uniform benchmark in \cref{eqn:finite-minimax-rate}, the universal rate benchmark \eqref{eqn:universal-rate} describes learning on each fixed problem $(P,\cF)$ as the sample size grows, with problem-dependent constants, rather than guarding against a potentially different worst-case problem at every $n$. The universal rates benchmark produces a learning curve for each fixed problem $ (P,\cF) $ and asks for the common rate behavior of these curves, whereas the uniform benchmark combines all learning curves into a single pointwise worst-case learning curve.
\cref{fig:universal-vs-minimax} illustrates these two benchmarks in their expectation formulations.

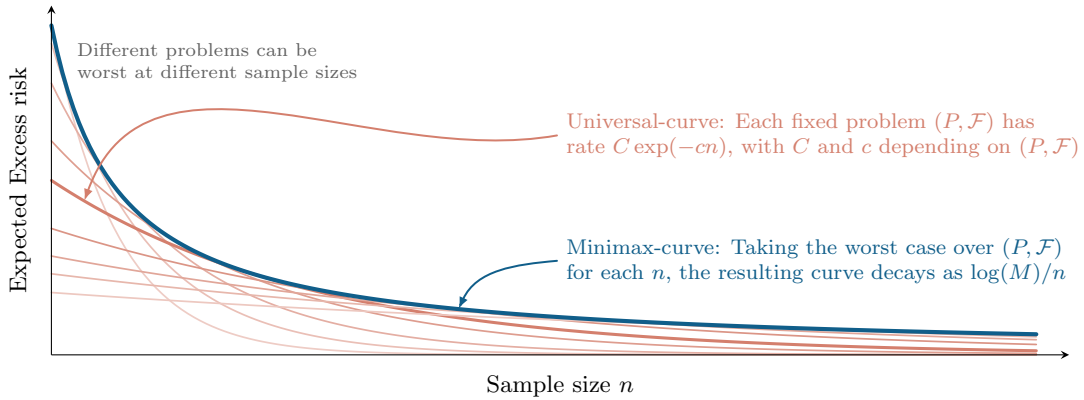
\begin{figure}[H]
    \vspace{-0.1cm}
    \centering
    \begin{tikzpicture}
\begin{axis}[
    width=0.94\linewidth,
    height=6.2cm,
    axis lines=left,
    xmin=1,
    xmax=16.5,
    ymin=0,
    ymax=0.39,
    xlabel={Sample size $n$},
    ylabel={Expected Excess risk},
    xtick=\empty,
    ytick=\empty,
    tick align=outside,
    tick label style={font=\footnotesize},
    label style={font=\small},
    line cap=round,
    line join=round,
    clip=false,
    set layers=axis on top
]

\addplot[color3!24, line width=0.65pt, domain=1:16, samples=160] {exp(-x)};
\addplot[color3!30, line width=0.65pt, domain=1:16, samples=160] {0.7*exp(-0.7*x)};
\addplot[color3!38, line width=0.65pt, domain=1:16, samples=160] {0.5*exp(-0.5*x)};
\addplot[color3!46, line width=0.65pt, domain=1:16, samples=160] {(1/3)*exp(-x/3)};
\addplot[color3!60, line width=1.15pt, domain=1:16, samples=160] {0.25*exp(-0.25*x)};
\addplot[color3!50, line width=0.65pt, domain=1:16, samples=160] {(1/6)*exp(-x/6)};
\addplot[color3!42, line width=0.65pt, domain=1:16, samples=160] {0.125*exp(-0.125*x)};
\addplot[color3!34, line width=0.65pt, domain=1:16, samples=160] {0.1*exp(-0.1*x)};
\addplot[color3!26, line width=0.65pt, domain=1:16, samples=160] {0.075*exp(-0.075*x)};

\addplot[color1, line width=1.55pt, domain=1:16, samples=180] {1/(e*x)};

\node[
    anchor=west,
    align=left,
    font=\footnotesize,
    text=color1,
] (minimax-label) at (axis cs:8.7,0.105)
    {Minimax-curve: Taking the worst case over $(P,\cF)$ \\ for each $ n $, the resulting curve decays as $\log(M)/n$
    }
    ;
\draw[-{Latex[length=1.6mm]}, color=color1, line width=0.75pt]
    (minimax-label.west) to[out=185,in=65] (axis cs:7.2,{1/(7.2*e)});

\node[
    anchor=west,
    align=left,
    font=\footnotesize,
    text=color3!60,
] (universal-label) at (axis cs:8.7,0.245)
    {Universal-curve: Each fixed problem $ (P,\cF) $ has \\ rate $C\exp(-cn)$, with $C$ and $c$ depending on $(P,\cF)$
    }
    ;
\draw[-{Latex[length=1.6mm]}, color=color3!60, line width=0.75pt]
    (universal-label.west) to[out=190,in=70] (axis cs:1.5,{0.25*exp(-0.25*1.5)});

\node[
    anchor=south west,
    align=left,
    font=\scriptsize,
    text=black!62,
] at (axis cs:1.25,0.295)
    {Different problems can be\\worst at different sample sizes};

\end{axis}
\end{tikzpicture}
    \caption{Minimax and universal rates view finite dictionary learning differently. Each curve tracks the expected excess risk for one fixed problem $(P,\cF)$. The minimax benchmark takes the worst problem separately at each $n$, producing the blue curve, whereas a universal guarantee follows each fixed curve and allows problem-dependent constants. In this schematic family, every fixed curve decays exponentially even though the blue minimax curve decays only polynomially. This figure is inspired by the opening figures in \cite{bousquet2021theory,schuurmans1997characterizing}, wherein the collection of universal curves and the minimax curve may not necessarily be for the same algorithm (unlike in this plot). The minimax curve is the envelope of the universal learning curves only when a best-of-both-worlds algorithm exists.}
    \vspace{-0.2cm}
    \label{fig:universal-vs-minimax}
\end{figure}
To summarize, in the minimax viewpoint of \eqref{eqn:finite-minimax-rate}, some form of hedging is necessary for optimality (ERM does not achieve optimal uniform rates). In contrast, no hedging is required to attain the universal benchmark in \eqref{eqn:universal-rate}, as demonstrated by ERM. This seeming tension yields the following question:
\begin{quote}
\centering
\normalsize\textbf{Question:} \emph{Do universal exponential rates come at the cost of uniform guarantees?}
\end{quote}

In the finite dictionary setting, we show that $Q$-aggregation (introduced in \cref{subsec:Q-aggregation}) achieves universal exponential rates without compromising uniform performance. In the following section, we introduce relaxations of the uniform and universal benchmarks in \cref{eqn:finite-minimax-rate,eqn:universal-rate}, allowing us to quantify the extent to which an algorithm performs well in both regimes. Among the many well-known algorithmic principles considered in \cref{tab:minimax-vs-exponential}, $Q$-aggregation is the \emph{only} one that simultaneously achieves the strongest uniform \eqref{eqn:finite-minimax-rate} and universal \eqref{eqn:universal-rate} guarantees.
In contrast, for countably infinite dictionaries, we establish an inherent trade-off between uniform and universal guarantees and show that $Q$-aggregation attains the Pareto-optimal frontier.
See \cref{subsec:summary-main-results} for a full summary of the main results.

\subsection{Definitions and Preliminaries}
\label{subsec:definitions-and-preliminaries}

In this section, we formally introduce the uniform and universal learning frameworks. An overview of notation can be found in \cref{tab:notation} of the appendix. For any set $A$, we write $A^*:= \bigcup_{n=1}^\infty A^n$. We write $a\lesssim b$ if there exists a universal constant $C>0$ such that $a\leq Cb$, and $a\asymp b$ if $a\lesssim b \lesssim a$. We denote $[M]=\crl{1,\ldots, M}$. The simplex is denoted $\triangle_{M}=\{\rho\in[0,1]^M:\sum_{k=1}^M \rho_k =1\}$.

Let $\cX$ denote an arbitrary fixed covariate space, assumed to be large enough, with $\sigma$-algebra $\Sigma$, and let $\Pall$ denote the space of all probability measures on $(\cX\times [0,1],\Sigma \otimes \cB([0,1]))$, where $\cB([0,1])$ is the Borel $\sigma$-algebra. Let $\cM$ denote the space of $\Sigma$-$\cB([0,1])$-measurable functions $\cX\to[0,1]$.
Formally, we study learning over problem instances $(P,\cF)$ from the problem space
\begin{equation*}
    \Theta\subseteq \crl{(P,\cF): P\in \Pall, \cF\subset \cM}.
\end{equation*}
If the dictionary $ \cF $ is constrained to have size at most $M\in\NN$, we denote the corresponding problem space by $\Theta_M$, and if it is constrained to be countable, we denote the corresponding problem space by $\Theta_{\NN}$.
We consider algorithms $\cA$ that take as input a hypothesis class $\cF\in 2^\cM$ and a sample $\rS\sim P^{n}$ and output a predictor $\fhat\in \cM$, that is, a map $ \cA: 2^\cM \times \prn{\cX\times [0,1]}^* \to \cM$.
If the function class is clear from the context or fixed, we will often write simply $\cA(\rS)\equiv \cA(\cF,\rS)$.

It is important to distinguish estimators that are aware of the confidence level they will be evaluated on (see \cref{subsec:pruning}). In particular, some works have studied algorithms of the form $\cA:(0,1)\times 2^\cM \times (\cX\times [0,1])^*\to \cM$, where the first argument represents the confidence level. When the confidence parameter $\delta$ and function class $\cF$ are clear from context we write $\cA(\rS)\equiv \cA(\delta,\cF,\rS)$.

\paragraph{Uniform Learning.}
We begin by formally introducing the uniform learning framework in the finite case.
Uniform learning requires the algorithm to perform well, for each sample size $n$, across all problem instances $(P,\cF)\in\Theta$, leading to a minimax formulation.
\begin{definition}[Minimax optimality for finite hypothesis spaces]
\label{def:minimax-finite-spaces}
    An algorithm $\cA$ is said to achieve the following properties for finite hypothesis spaces:
    \begin{itemize}[leftmargin=*,itemsep=0pt,topsep=0pt]
        \item  \emph{minimax optimality in expectation}, if there exists a universal constant $C>0$ such that for all $M,n\in\NN$ and $(P,\cF)\in\Theta_M$, we have $\EE_{\rS\sim P^{ n}}\brk{\excessRisk(\cA(\cF,\rS))}\le C \log(M)/n$.
        \item \emph{minimax optimality for fixed confidence}, if it takes the confidence level as input and there exist universal constants $C,c>0$ such that for all $M,n\in \NN$, $\delta\in(0,c)$, and $(P,\cF)\in\Theta_M$, we have $\PP_{\rS\sim P^{ n}}(\excessRisk(\cA(\delta,\cF,\rS))> C \log(M/\delta)/n)\leq \delta $.
        \item \emph{minimax optimality along the tail}, if there exist universal constants $C,c>0$ such that for all $M,n\in \NN$, $\delta\in(0,c)$, and $(P,\cF)\in\Theta_M$, we have $\PP_{\rS\sim P^{ n}}(\excessRisk(\cA(\cF,\rS))> C \log(M/\delta)/n)\leq \delta $, cf. \eqref{eqn:finite-minimax-rate}.
    \end{itemize}
\end{definition}
The tightness of the bounds is well known; see \cite{tsybakov2003optimal}. The three notions are related as follows. Minimax optimality along the tail is strictly stronger than minimax optimality for fixed confidence. Indeed, an algorithm that is optimal along the tail yields a fixed-confidence algorithm simply by ignoring the confidence parameter. The converse fails, as we will see in \cref{sec:finite}. Integrating the tail bound shows that minimax optimality along the tail implies minimax optimality in expectation; see \cref{lem:minimax-along-tail-implies-expectation} with proof in \cref{sec:proofs-preliminaries}.
\begin{lemma}
\label{lem:minimax-along-tail-implies-expectation}
    For finite hypothesis spaces, if an algorithm $\cA$ is minimax optimal \emph{along the tail},
    then the same algorithm $\cA$ is also minimax optimal \emph{in expectation}.
\end{lemma}

The converse is false: an algorithm may be minimax optimal in expectation without being optimal along the tail \cite{audibert2007progressive}. Minimax optimality in expectation also does not imply minimax optimality for fixed confidence \cite{lecue2013optimality,hogsgaard2026aggregation}. Whether there are algorithms that are minimax optimal for fixed confidence and in expectation, but \emph{not} along the tail, remains open. Among these three notions, minimax optimality along the tail constitutes the strongest requirement and serves as the natural gold standard.

For (countably) infinite hypothesis spaces, we will only consider minimax optimality along the tail.
\begin{definition}[Minimax rate optimality along the tail for infinite hypothesis spaces]
\label{def:minimax-along-tail}
    Define the minimax function
    \begin{equation*}
        \fM(\Theta,\delta,n):=\inf_{\cA'}\ \inf\Big\{r>0: \sup_{(P,\cF)\in\Theta}\ \PP_{\rS\sim P^{ n}}\prn{\excessRisk(\cA'(\cF,\rS))>r}\leq \delta\Big\}.
    \end{equation*}
    We say the algorithm $\cA$ is \emph{minimax optimal on $\Theta$ along the tail} if there exist constants $C,c>0$ such that
    \begin{equation*}
        \forall n\in \NN, \delta\in(0,c): \qquad \sup_{(P,\cF)\in\Theta}\ \PP_{\rS\sim P^{ n}}\prn{\excessRisk(\cA(\cF,\rS))> C\fM(\Theta,\delta,n)}\leq \delta.
    \end{equation*}
\end{definition}
This definition of optimality along the tail does not imply the existence of such an algorithm \emph{per se}, because the minimax function takes the infimum over algorithms pointwise for each confidence parameter.

 \begin{remark}
While \cref{def:minimax-finite-spaces} tracks the dependence on the size of finite dictionaries via $\log(M)$, \cref{def:minimax-along-tail} captures function-class complexity only through the rate in $n$ at which the class can be learned (as opposed to some complexity measure such as VC dimension). This is merely for convenience, as our results do not require more fine-grained dependence.
 \end{remark}

\paragraph{Universal Learning.} We now turn from the \emph{uniform} requirements in the definitions above to \emph{universal} requirements.
We begin by defining universal exponential rates.
\begin{definition}[Universal exponential rates]
\label{def:exponential-rate-probability}
    An algorithm $\cA$ is said to achieve on $\Theta$, universally:
    \begin{itemize}[leftmargin=*,itemsep=0pt,topsep=0pt]
        \item \emph{exponential rate in expectation}, if for all $(P, \cF)\in\Theta$ there exist $c,C>0$ such that $\EE_{\rS\sim P^{n}}\brk{\excessRisk(\cA(\cF,\rS))} \leq Ce^{-cn}$ for all sample sizes $n\in \NN$.
        \item \emph{exponential rate with exponential probability}, if for all $(P,\cF)\in\Theta$ there exist $c,C>0$ such that $\PP_{\rS\sim P^{n}}\prn{\excessRisk(\cA(\cF,\rS))>Ce^{-cn}}\leq Ce^{-cn}$ for all sample sizes $n\in\NN$.
        \item \emph{zero excess risk with exponential probability}, if for all $(P,\cF)\in\Theta$ there exist $c,C>0$ such that $\PP_{\rS\sim P^{n}}\prn{\excessRisk(\cA(\cF,\rS))>0}\leq Ce^{-cn}$ for all sample sizes $n\in\NN$, cf.\ \eqref{eqn:universal-rate}.
    \end{itemize}
\end{definition}
The first notion is the direct adaptation to regression with squared loss of the universal exponential rate introduced in \cite{bousquet2021theory}.
We introduce the two probabilistic variants for two reasons: First, in the uniform sense, it is well-understood that because of the improper nature of optimal estimators, it is easier to achieve guarantees in expectation than with high probability. Second, as noted above, simple ERM already achieves the stronger guarantee of zero excess risk with exponential probability.
The fact that it is stronger is formalized in the following lemma. The proof of the \cref{lem:exponential-implications} can be found in \cref{sec:proofs-preliminaries}.
\begin{lemma}
\label{lem:exponential-implications}
    Any algorithm that achieves zero excess risk with exponential probability also achieves an exponential rate with exponential probability, and any algorithm that achieves an exponential rate with exponential probability also achieves an exponential rate in expectation.
\end{lemma}

In the universal learning framework, given a learning rate (exponential in our setting), dependence on the function-class size is no longer relevant because the constants in the bound may, and must, depend on the distribution, which means the function class complexity is absorbed into these constants.

Finally, if the dictionary is nontrivial in the sense of \cref{lem:exponential-is-optimal} (cf.\ Definition 7 in \cite{attias2024universal}), one cannot hope for universal learning rates faster than exponential. This follows directly from an argument similar to Proposition 1 in \cite{attias2024universal} and arguments appearing in \cite{schuurmans1997characterizing,bousquet2021theory}. The proof of \cref{lem:exponential-is-optimal} is in \cref{sec:proofs-preliminaries}.
\begin{lemma}
\label{lem:exponential-is-optimal}
    Let $\cF\subset\cM$ be such that there exist two points $x,x'\in\cX$ and functions $f_1,f_2\in\cF$ with $f_1(x)=f_2(x)$ and $f_1(x')\neq f_2(x')$. Then there exist constants $C,c>0$ such that for any learning algorithm $\cA$, there exists a distribution $P$ on $\cX\times [0,1]$ with $\p_{\rS\sim P^n}(\excessRisk(\cA(\rS))>0)\geq Ce^{-cn}$ and $\EE_{\rS\sim P^n}[\excessRisk(\cA(\rS))]\geq Ce^{-cn}$ for infinitely many $n\in\NN$.
\end{lemma}

\subsection{Summary of Main Results}
\label{subsec:summary-main-results}
We now summarize our main results, which can be grouped into two categories: finite dictionary aggregation, covered in \cref{sec:finite}, and (countably) infinite dictionary aggregation, discussed in \cref{sec:countably-infinite-hyp-spaces,sec:best-of-both-worlds}.

For finite dictionaries, we fully resolve the main question posed in the introduction.
\begin{theorem*}[Consequence of \cref{thm:Q-exponential}]
    For finite dictionaries, there exists an algorithm that is minimax optimal along the tail and achieves zero excess risk with exponential probability.
\end{theorem*}
We provide several results about which estimators can and cannot achieve these guarantees. An overview of common estimators appears in \cref{tab:minimax-vs-exponential} (question marks indicate open questions to the best of our knowledge).
We first show that while pruning-based estimators, including ERM, achieve zero excess risk with exponential probability (\cref{thm:pruning-exponential}), they cannot achieve minimax optimality along the whole tail (\cref{thm:pruning-subexp-tail-strong,thm:pruning-no-along-the-tail}).
We next point out that minimax-optimal estimators based on online-to-batch conversions via averaging cannot achieve exponential rates (\cref{thm:online-to-batch-averaging-impossibility}).
We prove that the star estimator also does not achieve exponential rates (\cref{prop:star-linear}).
Our main result for finite dictionaries (\cref{thm:Q-exponential}) shows that the $Q$-aggregation estimator achieves zero excess risk with exponential probability while maintaining minimax optimality along the whole tail.

\begin{table}[t]
\centering
\caption{Minimax optimality and exponential universal learning rates for known estimators in finite dictionary aggregation. For Bayesian estimators (defined by a distribution over the dictionary), we consider only the \emph{aggregated} estimator; that is, we do not consider randomization. \\
{\cmark: estimator achieves guarantee \xmark: estimator does not achieve guarantee \qmark: unknown}}
\vspace{-0.1cm}
\small
\setlength{\tabcolsep}{2.5pt}
\renewcommand{\arraystretch}{1.5}
\begin{tabular}{llccccccccc}
\toprule

\multicolumn{2}{l}{taxonomy}
& \multicolumn{2}{c}{basic}
& \multicolumn{2}{c}{\makecell{pruning\\\cshref{thm:pruning-exponential}, \ref{thm:pruning-no-along-the-tail}}}
& \multicolumn{3}{c}{\makecell{averaging \\ \cshref{thm:online-to-batch-averaging-impossibility}}}
& Star
& $Q$-agg. \\
\midrule

\multicolumn{2}{l}{guarantee} & \textbf{ERM} & \textbf{EW} & \textbf{PC} & \textbf{Mid.}
& \textbf{PM} & \textbf{BOA} & \textbf{Seq.} & \textbf{Star} & \textbf{$Q$-agg.} \\
\midrule

\multirow{3}{*}{\rotatebox[origin=c]{90}{\makecell{\textbf{uniform} \\ {\tiny $\exists C,c>0:\forall M,n\in\NN$} \\ {\tiny $(P,\cF)\in \Theta_M$}}}}
& \makecell[l]{in expectation \\ $\EE[\cE(\fhat)]\leq C\frac{\log(M)}{n}$}
& \makecell{\xmark \\ \cite{DEVROYE19951011}$^*$}
& \makecell{\cmark \\ \cite{hogsgaard2026aggregation}}
& \qmark
& \qmark
& \makecell{\cmark \\ \cite{catoni2004statistical,mourtada2023local}}
& \makecell{\cmark \\ \cite{wintenberger2017optimal}}
& \makecell{\cmark \\ \cite{van2023high}}
& \makecell{\cmark \\ \cite{audibert2007progressive}}
& \makecell{\cmark \\ \cite{lecue2014optimal}}
\\

& \makecell[l]{fixed confidence $\delta\in(0,c)$\\ $\PP(\cE(\fhat_\delta)>C\frac{\log(M/\delta)}{n})\leq \delta$}
& \makecell{\xmark \\ \cite{DEVROYE19951011}$^*$}
& \makecell{\xmark \\ \cite{lecue2013optimality}}
& \makecell{\cmark \\ \cite{lecue2009aggregation}$^\dagger$}
& \makecell{\cmark \\ \cite{kanade2024exponential}}
& \makecell{\xmark \\ \cite{audibert2007progressive}}
& \makecell{\cmark \\ \cite{wintenberger2017optimal}}
& \makecell{\cmark \\ \cite{van2023high}}
& \makecell{\cmark \\ \cite{audibert2007progressive,kanade2024exponential}}
& \makecell{\cmark \\ \cite{lecue2014optimal}}
\\

& \makecell[l]{along the tail, $\forall \delta\in(0,c)$ \\ $\PP(\cE(\fhat)>C\frac{\log(M/\delta)}{n})\leq \delta$}
& \makecell{\xmark \\ \cite{DEVROYE19951011}$^*$}
& \makecell{\xmark \\ \cite{lecue2013optimality}}
& \makecell{\xmark \\ \cshref{cor:pc-ERM-midpoint-no-along-the-tail}}
& \makecell{\xmark \\ \cshref{cor:pc-ERM-midpoint-no-along-the-tail}}
& \makecell{\xmark \\ \cite{audibert2007progressive}}
& \makecell{\cmark \\ \cite{wintenberger2017optimal}}
& \makecell{\cmark \\ \cite{van2023high}}
& \makecell{\cmark \\ \cite{audibert2007progressive,kanade2024exponential}}
& \makecell{\cmark \\ \cite{lecue2014optimal}}
\\

\midrule

\multirow{3}{*}{\rotatebox[origin=c]{90}{\makecell{\textbf{universal} \\ {\tiny$\forall M\in\NN, (P,\cF)\in \Theta_M$}\\{\tiny $\exists C,c>0:\forall n\in\NN$}}}}
& \makecell[l]{in expectation \\ $\EE[\cE(\fhat)]\leq Ce^{-cn}$}
& \makecell{\cmark \\ \cshref{prop:ERM-exponential}}
& \makecell{\cmark \\ \cshref{prop:EW-exponential}}
& \makecell{\cmark \\ \cshref{cor:pruned-erm-midpoint-exponential}}
& \makecell{\cmark \\ \cshref{cor:pruned-erm-midpoint-exponential}}
& \makecell{\xmark \\ \cshref{cor:PM-exponential}}
& \makecell{\xmark \\ \cshref{cor:PM-exponential}}
& \makecell{\xmark \\ \cshref{cor:PM-exponential}}
& \makecell{\xmark \\ \cshref{prop:star-linear}}
& \makecell{\cmark \\ \cshref{thm:Q-exponential}}
\\

& \makecell[l]{exponential probability \\ $\PP(\cE(\fhat)>Ce^{-cn})\leq Ce^{-cn}$}
& \makecell{\cmark \\ \cshref{prop:ERM-exponential}}
& \makecell{\cmark \\ \cshref{prop:EW-exponential}}
& \makecell{\cmark \\ \cshref{cor:pruned-erm-midpoint-exponential}}
& \makecell{\cmark \\ \cshref{cor:pruned-erm-midpoint-exponential}}
& \makecell{\xmark \\ \cshref{cor:PM-exponential}}
& \makecell{\xmark \\ \cshref{cor:PM-exponential}}
& \makecell{\xmark \\ \cshref{cor:PM-exponential}}
& \makecell{\xmark \\ \cshref{prop:star-linear}}
& \makecell{\cmark \\ \cshref{thm:Q-exponential}}
\\

& \makecell[l]{zero excess risk \\ $\PP(\cE(\fhat)>0)\leq Ce^{-cn}$}
& \makecell{\cmark \\ \cshref{prop:ERM-exponential}}
& \makecell{\xmark \\ \cshref{prop:EW-exponential}}
& \makecell{\cmark \\ \cshref{cor:pruned-erm-midpoint-exponential}}
& \makecell{\cmark \\ \cshref{cor:pruned-erm-midpoint-exponential}}
& \makecell{\xmark \\ \cshref{cor:PM-exponential}}
& \makecell{\xmark \\ \cshref{cor:PM-exponential}}
& \makecell{\xmark \\ \cshref{cor:PM-exponential}}
& \makecell{\xmark \\ \cshref{prop:star-linear}}
& \makecell{\cmark \\ \cshref{thm:Q-exponential}}
\\

\bottomrule
\multicolumn{11}{l}{\makecell[l]{\footnotesize $^*$The lower bound in \cite{DEVROYE19951011} can be adjusted to our setting to yield the sub-optimality of ERM. \\ \footnotesize $^\dagger$The dependence on $\delta$ in Theorem A of \cite{lecue2009aggregation} is suboptimal, but their proof yields the correct dependence.}}
\end{tabular}
\label{tab:minimax-vs-exponential}
\end{table}
For (countably) infinite dictionaries, we show that there is a trade-off between universal and uniform guarantees: no algorithm can achieve both. We summarize this result in the following informal theorem.

\begin{theorem*}[Consequence of \cref{thm:bob-general}]
    There exist a countably infinite dictionary and a family of distributions for which a minimax optimal algorithm $\Amini$ and an algorithm $\Aexp$ that achieves zero excess risk with exponential probability exist, but no algorithm can achieve both guarantees.
\end{theorem*}

\begin{wrapfigure}{r}{0.4\textwidth}
    \vspace{-0.7cm}
    \centering
    \resizebox{\linewidth}{!}{
    \begin{tikzpicture}[
    >=Latex,
    every node/.style={align=center},
    bluept/.style={circle, draw=black, fill=color1, minimum size=5mm, inner sep=0pt},
    redpt/.style={circle, draw=black, fill=color4, minimum size=5mm, inner sep=0pt},
    orangept/.style={circle, draw=black, fill=color3, minimum size=5mm, inner sep=0pt},
    callout/.style={draw=black, rounded corners=3pt, fill=white, inner sep=6pt, text width=4.2cm},
    scale=0.8
]

\draw[very thick,->] (0,0) -- (8.8,0);
\draw[very thick,->] (0,0) -- (0,6.6);

\draw[thick, dashed] (4,0) -- (4,3);
\draw[thick, dashed] (0,3) -- (4,3);

\draw[thick, dashed] (8,0) -- (8,1);
\draw[thick, dashed] (0,1) -- (8,1);

\draw[thick, dashed] (1,0) -- (1,6);
\draw[thick, dashed] (0,6) -- (1,6);

\node[font=\itshape] at (-1.25,6.0) {$1$};
\node[font=\itshape] at (-1.25,3.0) {$\exp(-\frac{n}{\varphi(n)})$};
\node[font=\itshape] at (-1.25,1.0) {$\exp(-n)$};
\node[font=\itshape] at (1,-0.45) {$\frac{1}{n}$};
\node[font=\itshape] at (4,-0.45) {$\frac{1}{\varphi(n)}$};
\node[font=\itshape] at (8,-0.45) {$1$};

\node at (4.2,-1.0)
    {Uniform bound on $\excessRisk(\fhat)$};

\node[rotate=90] at (-2.6,3.2)
    {Universal bound on $\PP(\excessRisk(\fhat)>0)$};

\coordinate (c1) at (1.0,6.0);
\coordinate (c2) at (8.0,1.0);

\fill[gray!15]
  (1.0,6.0)
  -- (8.0,6.0)
  -- (8.0,1.0)
  .. controls  (3.85,2.85) ..
  (1.0,6.0)
  -- cycle;

\draw[thick]
  (c1)
  .. controls (3.85,2.85) ..
  (c2);

\node[text width=4.8cm] at (6.5,5.5)
  {attainable};

\node[text width=4.8cm] at (2.5,1.5)
  {unattainable};

\node[bluept] (Amin) at (1.0,6.0) {};
\node[text=color1] at (2.0,5.7) {$\Amini$};

\node[redpt] (Afix) at (4.0,3.0) {};
\node[text=color4] at (5.6,3.6) {$Q$-aggregation \\ of $\Amini,\Aexp$};

\node[orangept] (Aexp) at (8.0,1.0) {};
\node[text=color3] at (7.4,1.7) {$\Aexp$};

\end{tikzpicture}
    }
    \caption{The trade-off curve between universal and uniform rates, parameterized by functions $\varphi:\NN\to [4,\infty)$.}
    \label{fig:trade-off}
    \vspace{-2.0cm}
\end{wrapfigure}
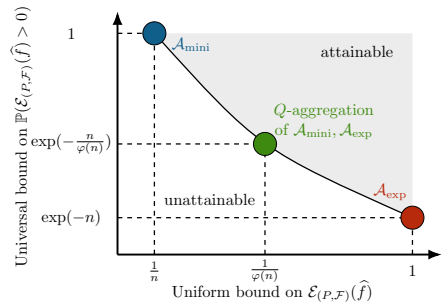

More specifically, in \cref{thm:bob-general}, we parameterize the trade-off between universal and uniform guarantees. There exist a function class $\cF$ and a set of distributions $\cP$ such that
for every $\varphi:\NN\to[4,\infty)$ and every $\cA$, there exist $(n_k)_{k=1}^\infty$ and $P,(P_k)_{k=1}^\infty\in \cP$ such that at least one of the following holds:
\begin{equation*}
\begin{aligned}
    &\forall k:\ \PP_{\rS\sim P^{n_k}}\prn{\excessRisk(\cA(\rS))>0}\geq c_1\exp\prn{-c_2\frac{n_k}{\varphi(n_k)}} \\
    \text{or} \quad &\forall k:\ \PP_{\rS\sim P_k^{n_k}}\prn{\excessRiskPar{P_k}{\cF}(\cA(\rS))>\frac{c_3}{\varphi(n_k)}}\geq c_4.
\end{aligned}
\end{equation*}
\cref{thm:QBOB} complements this result by showing that using $Q$-aggregation to combine $\Amini$ and $\Aexp$ with the prior and temperature depending on $\varphi$ achieves matching upper bounds whenever $\varphi(n)\lesssim n$.
The resulting trade-off is informally visualized in \cref{fig:trade-off}.
Formal versions of the preceding lower and upper bounds can be found in \cref{thm:bob-general,thm:QBOB}.

We also prove several structural results for countably infinite hypothesis spaces. In \cref{thm:arbitrarily-slow-rates-inf-not-realized}, we first show that there exists a countably infinite function class on which only arbitrarily slow universal rates can be expected. Without restricting the class of distributions, ``most'' function classes cannot achieve ``truly'' exponential rates (\cref{thm:nearly-exponential-rates-lower-bound}).
\cref{thm:countable-inf-realized-not-uniformly-learnable,thm:countable-inf-realized-exponential-rates} show that, for the broad class of distribution and function class pairs in which the infimum of the risk is attained, uniform learning is impossible, although almost exponential rates are achievable in the universal sense. This is similar to \cite[Example 2.3]{bousquet2021theory}.
From an algorithmic perspective, we show that---in contrast to the finite case---achieving exponential rates may require outputting a function that is \emph{not} a convex or finite combination of the functions in the function class (\cref{thm:arbitrarily-slow-rates-convex-finite-combination}).

\subsection{Related Work and Background}
\label{subsec:related-works}
\paragraph{Model Selection Aggregation and Uniform Learning.}
Aggregation has much of its foundation in PAC-Bayesian learning \cite{alquier2021user}, with some of the earlier works including \cite{barron1987bayes,nemirovski2000topics,leung2006information,catoni2004statistical,catoni2007pac,yang2000mixing,dalalyan2008aggregation}. The optimal rates of aggregation were established in \cite{tsybakov2003optimal}.
The fact that any proper method (also called a \emph{selector}) incurs a minimax-suboptimal error rate (e.g., $\gtrsim \sqrt{\log (M) / n}$) was shown in multiple works, such as \cite{juditsky2008learning,catoni2004statistical,rigollet2012sparse,DEVROYE19951011}. These results include any \emph{empirical risk minimizer} (ERM). The main reason is that, for squared loss, exploiting \emph{convexity} through \cref{eqn:Jensen-gap} is crucial \cite{mendelson2019unrestricted}. However, ERM on the convex hull of the dictionary suffers from increased complexity and is minimax-suboptimal for large dictionaries \cite{lecue2009aggregation}.

From the PAC-Bayesian perspective, one of the most fundamental estimators is the \emph{exponential weights} (EW) estimator (e.g., \cite{leung2006information,rigollet2012sparse}; see \cref{subsec:ERM-and-EW}). It has been shown to be minimax-suboptimal in expectation for low temperatures and minimax-suboptimal in deviation for all temperatures unless a Bernstein condition is satisfied \cite{lecue2013optimality,alquier2021user,dai2012deviation}. Recently, \cite{hogsgaard2026aggregation} showed that it is optimal in expectation for high temperatures. Previously, this was known only for fixed-design regression \cite{dalalyan2008aggregation}.

While we consider the batch setting, the problem has deep ties to the sequential setting \cite{hannan1957approximation,vovk1995game,cesa2006prediction}; see \cite{mourtada2023local} for a discussion and \cref{subsec:online-to-batch} for definitions. By averaging a sequence of EW estimators $\AEWestimator^{(i)}$ on the first $i$ samples, we obtain the \emph{progressive mixture} (PM) estimator \cite{barron1987bayes,catoni2004statistical,juditsky2008learning}. It is minimax optimal in expectation but not in deviation \cite{audibert2007progressive}. See also \cite{audibert2007progressive} for a detailed account of the estimator's origin.
The \emph{mirror averaging} estimator is equivalent to the PM estimator; it was merely derived from a different motivation \cite{juditsky2005recursive,juditsky2008learning,dalalyan2012mirror,lounici2007generalized}.
Later, the \emph{Bernstein online aggregation} (BOA) estimator \cite{wintenberger2017optimal} and a sequential estimator based on \emph{shifted loss} \cite{van2023high} use similar averaging techniques but achieve minimax-optimal tails.

In response to the deviation suboptimality of the PM estimator, J.-Y.\ Audibert developed the \emph{star estimator} in \citep{audibert2007progressive,audibert2009fast,gaiffas2009hyper} (cf.\ \cref{subsec:star-estimator}), which is the first and perhaps the simplest method to achieve the minimax rate both in expectation and in deviation. See \citep{liang2015learning,vijaykumar2021localization,kanade2024exponential} for analyses of the star estimator.
The star estimator creates a subset of the convex hull on which it performs ERM.
In \cite{lecue2009aggregation}, another method was developed for selecting a subset of the convex hull of $\cF$, on which ERM is minimax optimal in deviation, which we refer to as \emph{pruned-convex ERM}. See also \cite{gaiffas2009hyper}. Crucially, the algorithm requires the confidence level (i.e., $\delta$) at which it will be evaluated as input, a property shared by the \emph{midpoint} estimator, which is also known to be minimax optimal at the input confidence level \cite{kanade2024exponential}; cf.\ \cref{sec:proofs-pruning}.

The deviation-optimal procedures described above do not account for any possible prior over the models, whereas earlier PAC-Bayesian methods had this as a key feature. As argued in \cite{lecue2014optimal}, this may be important in practice. Motivated by this consideration, the \emph{$Q$-aggregation estimator} from \citep{dai2012deviation,lecue2014optimal} uses a prior; cf.\ \cref{subsec:Q-aggregation}.
It is known to be minimax optimal both in expectation and in deviation \cite{lecue2014optimal}. In \cite{mourtada2023local}, the authors prove a ``local'' risk bound for $Q$-aggregation that adapts to the risk gaps of suboptimal models, but still takes the form of a uniform oracle inequality rather than a universal, instance-wise learning rate.
Computational aspects of $Q$-aggregation are discussed in \cite{dai2012deviation}.

\paragraph{Data-dependent Dictionaries.} All results discussed above concern dictionaries that are independent of the data. This includes the setting of sample splitting, where the dictionary is first computed on an independent sample. The universal setting does not include sample splitting because it requires the dictionary to be fixed across sample sizes, a point central to our best-of-both-worlds results in \cref{sec:best-of-both-worlds}.
In contrast, a related line of work considers aggregation when the dictionary depends on the data used for aggregation. Most results of this kind rely on assumptions about this dependence. In fixed-design Gaussian regression, \cite{leung2006information} study mixtures of least-squares projection estimators fitted on the same observations used to construct the mixing weights, and related work studies aggregation of least-squares projections while taking optimal sparsity into account \cite{rigollet2012sparse}.
\emph{Affine estimators} form a broad class in fixed-design regression. For possibly uncountable dictionaries of this form, \cite{DalalyanS12} establish sharp oracle inequalities in expectation for exponential weights. In the same setting, \cite{dai2014aggregation} obtain high-probability guarantees and show that a generalized $Q$-aggregation procedure satisfies sharp oracle inequalities in deviation. Subsequently, \cite{bellec2018optimal} shows that, for finite collections of affine estimators, data dependence incurs no additional minimax cost. This line of work is extended in \cite{BellecZ21} beyond affine estimators to finite collections of data-dependent estimators satisfying a global Lipschitz condition.

\paragraph{Universal Learning Theory.}
The universal learning paradigm was first formally introduced in \cite{bousquet2021theory} for realizable binary classification, while earlier works had described the phenomenon. Specifically, in \cite{antos1996strong,bendavid1995parameterization,schuurmans1997characterizing}, the distinction between linear and exponential learning curves was first described, including the observation that ERM achieves exponential rates on finite hypothesis classes.
Later works provide a complete analysis of empirical risk minimization for zero-one loss, both in the realizable \cite{hanneke2024erm} and agnostic settings\cite{hanneke2025universalratesermagnostic}. Other works have extended the results to multiclass classification \cite{kalavasis2022multiclass}, multiclass learning with bandit feedback \cite{hanneke2025equivalent}, online learning \cite{hanneke2024optimistically,kalocinski2025computable}, (inter)active learning \cite{hanneke2022interactive,hanneke2024active}, revenue maximization \cite{HannekeKMV26}, language identification and generation \cite{KalavasiMV25}, and agnostic classification \cite{hanneke2026theory}.
The most closely related work analyzes universal learning for regression with cut-off and absolute loss \cite{attias2024universal} in the realizable setting. In that work, realizability means that for any distribution $ P $ and function class $ \cF $ considered $ \inf_{f\in \cF}\PP_{(X,Y)\sim P}( f(X)=Y)=0 $. For cut-off loss, the authors characterize the optimal universal rates for a given hypothesis class. For absolute loss, they show that infinitely many learning rates are possible. Their results can be extended to realizable regression with squared loss.
In contrast to this work, we consider the agnostic setting with squared loss, which, to the best of our knowledge, has not been studied in the universal learning literature on regression.

The overarching goal of the literature on universal learning \cite{attias2024universal,bousquet2021theory,hanneke2022interactive,hanneke2024active,hanneke2026theory,hanneke2024erm,hanneke2025universalratesermagnostic,kalavasis2022multiclass} is to characterize, for a given learning setting, which universal learning rates are achievable and under what conditions. This goal is usually pursued by identifying an appropriate combinatorial or geometric dimension of the function class, analogous to the VC dimension for minimax learning in binary classification. This paper instead addresses a different question: whether universal and uniform learning rates are compatible \emph{if we assume} they are achievable. We study this question specifically for dictionary learning under squared loss.

\section{Finite Hypothesis Spaces}\label{sec:finite}

In this section, we investigate the compatibility of uniform and universal rates for different known estimators for finite hypothesis spaces. Throughout this section, we let $M\equiv \abs{\cF}<\infty$.
Before we study specific estimators, we introduce some notation and basic observations.
We denote the subset of optimal models in the dictionary as
$\cFstar\equiv\cFstar(P):= \crl{f\in\cF:\excessRisk(f)=0}$,
and the minimal positive risk gap between optimal and suboptimal models as $$\Deltamin \equiv \Deltamin(P,\cF) :=  \min_{f\in\cF\setminus\cFstar} \excessRisk(f),$$
where we set $\Deltamin=\infty$ if $\cFstar=\cF$ as then any estimator in $\conv(\cF)$ has non-positive excess risk. We define the \emph{empirical risk} on the sample $\rS=(X_i,Y_i)_{i=1}^n$ as
\begin{equation}
\label{eqn:empirical-risk-def}
    \Rhat_{\rS}(f):= \frac{1}{n} \sum_{i=1}^n (f(X_i)-Y_i)^2.
\end{equation}
The following lemma is at the heart of many of the positive results below: it shows that, with high probability, empirical risk separates every suboptimal model from an optimal one by a fixed fraction of its true excess risk, which is what enables exponential universal rates via empirical risk comparisons. The lemma follows from a simple application of Hoeffding's inequality and a union bound.
\begin{lemma}\label{lem:Hoeffding-separation}
    For any $(P,\cF)\in\Theta_M$ and any fixed $\fstar\in\cFstar$, it holds that
    \begin{equation*}
        \PP_{\rS\sim P^{n}}\prn{\forall f\in\cF\setminus\cFstar: \ \Rhat_{\rS}(f)-\Rhat_{\rS}(\fstar) > \frac{1}{2}\excessRisk(f)} \geq 1-M\exp\prn{-\frac{1}{8}\Deltamin^2(P,\cF)n}.
    \end{equation*}
    We call the event above $\Esep(\rS)$.
\end{lemma}
\begin{proof}
    Denote the events $E_f=\{\Rhat_\rS(f)-\Rhat_\rS(\fstar) > \excessRisk(f)/2\}$.
    Since $f,\fstar$, and $y$ lie in $[0,1]$, we know that $\xi_f (x,y) := (f(x)-y)^2-(\fstar(x)-y)^2 \in[-1,1]$.
    It holds that $\EE_{(X,Y)\sim P} \brk{\xi_f(X,Y)} = \excessRisk(f)$ and we can write $\Rhat_\rS(f) -\Rhat_\rS(\fstar) = \frac{1}{n}\sum_{i=1}^n \xi_f(X_i,Y_i)$, where $\rS=(X_i,Y_i)_{i=1}^n$.
    Hoeffding's inequality then gives
    \begin{align*}
        \PP_{\rS\sim P^n}\prn{E_f^c} &= \PP_{\rS\sim P^n}\prn{\frac{1}{n}\sum_{i=1}^n \xi_f(X_i,Y_i)-\excessRisk(f) \leq -\frac{\excessRisk(f)}{2}} \leq \exp\prn{-\frac{n}{8}\excessRisk^2(f)}.
    \end{align*}
    Since $\excessRisk(f)\geq \Deltamin$ for all $f\in\cF\setminus\cFstar$, a union bound shows that with probability at least $1-M\exp\prn{-\frac{n}{8}\Deltamin^2}$ the event $\Esep(\rS)=\bigcap_{f\in\cF\setminus \cFstar} E_f$
    holds, which was the claim.
\end{proof}

\subsection{Warm-up: Empirical Risk Minimization and Exponential Weights}
\label{subsec:ERM-and-EW}

We begin by considering two basic but fundamental estimators that highlight the mechanism by which exponential rates are achievable in the universal sense, namely the empirical risk minimizer and its PAC-Bayesian counterpart, the exponential weights estimator.
Let $\erm=\Aerm(\rS)$ be an empirical risk minimizer (ERM) on the dictionary $\cF$; that is, given an i.i.d.\ sample $\rS=(X_i,Y_i)_{i=1}^n\sim P^{n}$, choose any minimizer
\begin{equation*}
    \erm\in\argmin_{f\in\cF} \Rhat_{\rS}(f).
\end{equation*}
\begin{proposition}
\label{prop:ERM-exponential}
The empirical risk minimizer achieves zero excess risk with exponential probability on $\Theta_M$.
Specifically, for every $(P,\cF)\in\Theta_M$, there exists a constant $c>0$ such that $\PP_{\rS\sim P^n}(\excessRisk(\erm)=0)\geq 1-Me^{-cn}$.
One such constant $c$ is $c=\frac{1}{8}\Deltamin^2$.
\end{proposition}
\begin{proof}
    By \cref{lem:Hoeffding-separation}, we know that for any fixed $\fstar\in\cFstar$, with probability at least $1-M\exp\prn{-\frac{n}{8}\Deltamin^2}$ the event $\Esep(\rS)$ holds, that is, for all $ f\in\cF\setminus\cFstar$ it holds $\Rhat_\rS(f)-\Rhat_\rS(\fstar) > \frac{1}{2}\excessRisk(f)$, implying $\Rhat_\rS(f)>\Rhat_\rS(\fstar)$.
    On the event $\Esep(\rS)$, since $\erm\in\cF$ (i.e., it is proper) and $\Rhat_\rS(\erm)\leq \Rhat_\rS(\fstar)$ by definition, we must have that $\erm\notin\cF\setminus\cFstar$, so $ \erm\in \cFstar $, which concludes the proof.
\end{proof}
The key observation above is that for a fixed distribution and finite dictionary, i.e., the universal learning setting, optimal and suboptimal models are separated by a fixed risk gap. As the sample size grows, the empirical risks concentrate around their expectations, so empirical risk distinguishes the two groups with exponentially high probability.
We now consider the PAC-Bayesian counterpart to the empirical risk minimizer.
The exponential weights estimator $\EW$ with temperature $\beta>0$ and uniform prior $\pi=(1/M,\ldots,1/M)$ is defined through the Gibbs posterior $\rhohat\in \triangle_M$ that has components
\begin{equation}
\label{eqn:exponential-weights-def}
     \rhohat_k = \frac{\exp(-\frac{n}{\beta} \Rhat_\rS(f_k))}{\sum_{j=1}^M  \exp(-\frac{n}{\beta} \Rhat_\rS(f_j))} \quad \iff \quad  \rhohat\in \argmin_{\rho\in\triangle_{M}} \crl{\EE_{k\sim \rho}\brk{\Rhat_\rS(f_k)}+ \frac{\beta}{n}\KL(\rho,\pi)}.
\end{equation}
The aggregated estimator is then defined as $\EW(\cdot) = \EE_{k\sim \rhohat}\brk{f_k(\cdot)}$: it is the convex combination of the dictionary elements with weights proportional to the exponentiated negative empirical risk, meaning that it is a soft version of the ERM that puts most weight on models with smaller empirical risk and less weight on models with larger empirical risk, whereas the ERM puts all the weight on the model with the smallest empirical risk.
\begin{proposition}
\label{prop:EW-exponential}
    The exponential weights estimator with constant temperature $\beta>0$ achieves exponential rates with exponential probability on $\Theta_M$. Specifically,
    for every $(P,\cF)\in\Theta_M$, there exists $c>0$ such that
$\PP_{\rS\sim P^n}(\excessRisk(\EW) \geq Me^{-cn})\leq Me^{-cn}$ for all $n\in\NN$. One such $c$ is $c=\min\{\tfrac{1}{8}\Deltamin^2,\tfrac{1}{2\beta}\Deltamin\}$.
    Moreover, there exists $(P,\cF)\in \Theta_M$ and constants $C,c>0$ such that $P^n$-almost surely, $\excessRisk(\EW)\geq Ce^{-cn}$ and the exponential weights estimator does not achieve zero excess risk with exponential probability.
\end{proposition}
\begin{proof}
    By \cref{lem:Hoeffding-separation}, we know that for any fixed $\fstar\in\cFstar$, with probability at least $1-M\exp\prn{-\frac{n}{8}\Deltamin^2}$ the event $\Esep(\rS)$ holds, that is, for all $f\in\cF\setminus\cFstar$ it holds $\Rhat_\rS(f)-\Rhat_\rS(\fstar) > \frac{1}{2}\excessRisk(f)\geq \frac{1}{2}\Deltamin$.
    On that event, we can bound the weights on each element $f_{k}\in\cF\setminus\cFstar$ as
    \begin{align*}
        \rhohat_k &= \frac{\exp(-\frac{n}{\beta}\Rhat_\rS(f_k))}{\sum_{j=1}^{M} \exp(-\frac{n}{\beta}\Rhat_\rS(f_j))} \leq \exp\prn{-\frac{n}{\beta}\prn{\Rhat_\rS(f_k)-\Rhat_\rS(\fstar)}} \leq \exp\prn{-\frac{n}{2\beta}\Deltamin(P,\cF)}.
    \end{align*}

    By Jensen's inequality, on the same event the exponential weights estimator satisfies
    \begin{equation*}
        \excessRisk(\EW) \leq \sum_{k=1}^M \rhohat_k \excessRisk(f_k) \leq M\exp\prn{-\frac{n}{2\beta}\Deltamin(P,\cF)},
    \end{equation*}
    where we have used that $\excessRisk\leq 1$, for $ f_k\in \cF \setminus \cFstar $ we have $\rhohat_{k}\leq \exp(-\frac{n}{2\beta}\Deltamin)$, and for $f\in\cFstar$ we have $\excessRisk(f)=0$.
    This proves the exponentially small excess-risk bound with exponentially high probability, where we may choose the constant $c=\min\{\tfrac{1}{8}\Deltamin^2,\tfrac{1}{2\beta}\Deltamin\}$ as specified.

    For the lower bound, take any $(P,\cF)\in\Theta_M$ such that $f^\star\in\cF$ is Bayes optimal and there exists a point $x\in\cX$ with $P_X(x)>0$ such that $f(x)>f^\star(x)$ for all $f\in\cF\setminus\{f^\star\}$. Then, for any $\rho\in\triangle_{M}$,
    \begin{align*}
        \cR_P(f_\rho)
        =
        \cR_P(f^\star)
        +
        \EE_{X\sim P_X}\brk{(f_\rho(X)-f^\star(X))^2}
        \geq
        \cR_P(f^\star)
        +
        P_X(x)\bigl(f_\rho(x)-f^\star(x)\bigr)^2 .
    \end{align*}
    Since $f_\rho(x)-f^\star(x)=\sum_{k\in[M]}\rho_k(f_k(x)-f^\star(x))$ and all summands are non-negative, for any fixed $f_j\in\cF\setminus\{f^\star\}$ we have $f_\rho(x)-f^\star(x)\geq \rho_j(f_j(x)-f^\star(x))$. For the exponential weights estimator, since empirical risks lie in $[0,1]$, \cref{eqn:exponential-weights-def} gives $\rhohat_j\geq e^{-n/\beta}/M$. Consequently, $\excessRisk(\EW)\geq P_X(x)\bigl(f_j(x)-f^\star(x)\bigr)^2 e^{-2n/\beta}/M^2$ almost surely. When the temperature $\beta$ is bounded below by a positive constant independent of $n$, the excess risk of EW is almost surely bounded below by an exponentially small term.
\end{proof}
The fact that both ERM and EW achieve exponential rates in expectation follows from \cref{lem:exponential-implications}. From the above propositions, we see that ERM achieves \emph{zero excess risk with exponential probability} whereas EW only achieves the weaker guarantee of \emph{exponential rates with exponential probability}. The reason for this is that ERM selects the best model from the dictionary, which by \cref{lem:Hoeffding-separation} eventually belongs to $ \cFstar $, while EW hedges and puts weight on all models, including suboptimal models, which can prevent it from achieving zero excess risk.

To summarize, while ERM, owing to its extreme sparsity, achieves the best possible universal guarantee of zero excess risk with exponential probability (\cref{prop:ERM-exponential}), it does not achieve minimax optimality of any kind. By hedging more, exponential weights achieves the minimax rate in expectation \cite{hogsgaard2026aggregation}, but at the cost of the weaker universal guarantee of an exponential rate with exponential probability (\cref{prop:EW-exponential}). This again portrays the seeming tension between the universal and uniform viewpoints.

Motivated by the observation that ERM eventually selects an optimal hypothesis by choosing the one with the smallest empirical risk, we next study pruning-based algorithms that aggregate hypotheses from a restricted subset of the dictionary with low empirical risk. By balancing sparsity and hedging, these algorithms are natural candidates for achieving both optimal minimax rates and exponential universal rates.

\subsection{Pruning with a Threshold}
\label{subsec:pruning}
We now introduce a class of estimators that we call \emph{pruning-based} estimators.
To this end, we consider algorithms $\cA(\delta,\rS,\cF)$ that take a confidence parameter $\delta$ as input, cf.\ \cref{subsec:definitions-and-preliminaries}.
\begin{definition}
\label{def:pruning-based}
    An estimator $\fhat=\cA(\delta,\rS,\cF)$ is \emph{pruning-based with parameters $\delta,\alpha\in(0,1]$ and threshold $\threshold\geq 0$} if it holds $\fhat\in \conv(\cFhat(\rS_1))$ almost surely for $\rS_1\subseteq\rS$ being the first $\abs{\rS_1}=m\geq \alpha n$ samples in $\rS$ and a random, non-empty subset
    \begin{equation*}
        \cFhat(\rS_1)\subseteq\crl{f\in\cF:\ \Rhat_{\rS_1}(f)\le \Rhat_{\rS_1}(\Aerm(\rS_1))+\threshold}.
    \end{equation*}
\end{definition}
We will consider pruning-based estimators where the threshold satisfies $\threshold\to 0$ as $n \to \infty$; see \cref{subfig:pruning-conv,subfig:pruning-collapse} for a visualization.
An example of pruning-based estimators is ERM with $\threshold=0$, $ \alpha=1 $, but the class of pruning-based estimators also contains two other prominent estimators.
\begin{itemize}[leftmargin=*,itemsep=0pt,topsep=2pt]
    \item In \cite{lecue2009aggregation}, it is shown that first pruning the set of estimators and then running ERM on the convex hull of that pruned set, which we refer to as \emph{pruned-convex ERM} and describe in \cref{alg:pruned-convex-ERM} of \cref{sec:proofs-pruning}, achieves pointwise minimax optimality for the same $\delta$ provided to the algorithm.\footnote{Theorem 4.2 in \cite{lecue2009aggregation} displays the wrong dependence on $\delta$ to satisfy their Definition 1.1, but the proof actually yields the correct dependence.}
    \item In \cite{kanade2024exponential}, it is shown that the \emph{midpoint} estimator (defined in \cref{alg:midpoint}) achieves pointwise minimax optimality when the confidence $\delta$ is the same as the one provided to the algorithm.
\end{itemize}
Both satisfy the definition of pruning-based estimators (see \cref{lem:pruned-convex-ERM-is-pruning-based,lem:midpoint-is-pruning-based} in \cref{sec:proofs-pruning}). We show that any such estimator achieves exponential rates.
The following theorem can be understood as a generalization of \cref{prop:ERM-exponential}.

\begin{theorem}
\label{thm:pruning-exponential}
For any pruning-based estimator $\fhat$ with a threshold satisfying $\limsup_{m\to\infty}\threshold=0$,
and for any $(P,\cF)\in\Theta_M$, there exists a constant $c>0$ such that
$\PP\bigl(\excessRisk(\fhat)>0\bigr)\le M e^{-cn}$.
\end{theorem}
\begin{remark}
    Note that here the constant $c$ also depends on the hyperparameters $\alpha$, $M$, and $\delta$.
\end{remark}
\begin{proof}
If $\cF=\cFstar$ (in particular if $M=1$), then the conclusion follows by Jensen's inequality and convexity of the risk, so assume $\cF\setminus\cFstar\neq\emptyset$.
We first show that for all sufficiently large $n$, on a high-probability event, one has
$\cFhat\equiv \cFhat(\rS_1)\subseteq \cFstar$.
Recall from \cref{lem:Hoeffding-separation} that, for any fixed $\fstar\in\cFstar$, the event $\Esep(\rS_1)$ on which, for all $ f\in\cF\setminus\cFstar$, we have $\Rhat_{\rS_1}(f)-\Rhat_{\rS_1}(\fstar)>\tfrac12 \excessRisk(f)$
holds with probability at least $1-M\exp\prn{-\frac{m}{8}\Deltamin^2}$. From the proof of \cref{prop:ERM-exponential}, we also know that on $\Esep(\rS_1)$, $\erm\in\cFstar$.
Since $\limsup_{m\to\infty}\threshold=0$,
there exists $m_0\in\NN$ such that for all $m\ge m_0$, $\threshold\le \frac{\Deltamin}{2}$.
Whenever $m\ge m_0$, the definition of $\cFhat$ gives, for every $f\in\cFhat$,
\begin{equation*}
    \Rhat_{\rS_1}(f)\leq\Rhat_{\rS_1}(\erm(\rS_1))+\threshold\leq
\Rhat_{\rS_1}(f^\star)+\frac{\Deltamin}{2}.
\end{equation*}
Therefore, if $m\geq m_0$, then on $\Esep(\rS_1)$, we have $\cFhat\subseteq \cFstar$.

In that case, since $\fhat\in\conv(\cFhat)\subseteq \conv(\cFstar)$ almost surely, it follows by Jensen's inequality and convexity of the risk that $\excessRisk(\fhat)\le 0$.
Finally, choose $n_0$ so that for all $n\ge n_0$, we have $m\geq\alpha n \geq \alpha n_0 \geq m_0$, and set
\begin{equation*}
    c:=\min\crl{\frac{\alpha\Deltamin^2}{8},\frac{\log M}{\max\crl{2,n_0}}}.
\end{equation*}
Then for all $n\geq \max\crl{2,n_0}$, $\PP\bigl(\excessRisk(\fhat)>0\bigr)\le M e^{-cn}$,
and for $2\leq n<\max\crl{2,n_0}$, $\PP\bigl(\excessRisk(\fhat)>0\bigr)\le 1\le M e^{-cn}$
since $cn\le c\max\crl{2,n_0}\le \log M$.
This proves the claim.
\end{proof}
As a consequence, we immediately get that the two mentioned algorithms, pruned-convex ERM and the midpoint estimator, achieve zero excess risk with exponential probability.
\begin{corollary}
\label{cor:pruned-erm-midpoint-exponential}
    The pruned-convex ERM (\cref{alg:pruned-convex-ERM}) (for $ n\geq2 $) and the midpoint estimator (\cref{alg:midpoint}) both achieve zero excess risk with exponential probability; for any $(P,\cF)\in\Theta_M$, there exists a constant $c>0$ such that $\PP_{\rS\sim P^n}(\excessRisk(\prunedERM)>0)\leq Me^{-c n}$ and a constant $c>0$ such that $\PP_{\rS\sim P^n}(\excessRisk(\midpoint)>0) \leq Me^{-cn}$ for all $n\in\NN$ ($n\geq 2$ for $\prunedERM$). The constant depends on the hyperparameters $\alpha$, $M$, and $\delta$.
\end{corollary}
\begin{proof}
    This follows from \cref{thm:pruning-exponential} and \cref{lem:pruned-convex-ERM-is-pruning-based,lem:midpoint-is-pruning-based} in \cref{sec:proofs-pruning}.
\end{proof}
Here, $c$ also depends on $\delta$ used by the algorithms.
While \cref{thm:pruning-exponential} yields a large class of algorithms that achieve strong guarantees, we now demonstrate that any pruning-based estimator with a vanishing threshold must be minimax suboptimal on some part of the tail. This implies that pointwise minimax optimality (for fixed confidence $\delta$), as satisfied by the pruned-convex ERM or midpoint estimators, does not necessarily imply minimax optimality along the tail. The proof of \cref{thm:pruning-subexp-tail-strong} can be found in \cref{proof:pruning-subexp-tail-strong}.

\begin{theorem}\label{thm:pruning-subexp-tail-strong}
    There exist a dictionary $\cF\subset\cM$ of size $M=2$ and universal constants
$c_1,c_2,c_3>0$ such that the following holds. Fix any
$\delta\in(0,1)$, let $\fhat=\cA(\delta,\rS,\cF)$ be pruning-based
with pruning subsample $\rS_1\subseteq\rS$ of size $m\ge \alpha n$.
Assume that the threshold $\threshold$ is deterministic and satisfies
$\limsup_{m\to\infty}\threshold=0$. Then, for all sufficiently large
$n$,
\begin{equation*}
    \sup_{P\in\Pall} \PP_{\rS\sim P^n}\prn{\excessRisk(\fhat)\ge c_1\sqrt{\frac{\log(1/\eta)}{n}}}\geq \eta,
    \qquad \text{where}\quad \eta:=c_2\,m^{-1/2}\exp\prn{-c_3 m\threshold^2}.
\end{equation*}
\end{theorem}
\begin{remark}
    We first remark that “sufficiently large $ n $” may depend on $\delta, \alpha$ and on the threshold sequence.
    Furthermore, because $\threshold=o(1)$, we have that $m\threshold^2 =o(m)$ and so $\sqrt{\log(1/\eta)/n}  \lesssim \sqrt{(\log(m)+o(m))/n} \to 0$.
    Hence, the lower bound is much larger than the optimal tail scale $\log(1/\eta)/n$ at confidence level $\eta$, and is not in the ``trivial'' regime of constant excess risk.
\end{remark}

We can strengthen the preceding bound for pruning-based estimators that prune at a ``sufficiently small'' threshold $\threshold \leq B_M(\delta)/\sqrt{m}$. In that case, the confidence level at which the estimator is suboptimal can be chosen independently of the sample size, essentially due to the Berry-Esseen theorem. This can be viewed as a generalization of the lower bounds by \cite{catoni2004statistical,juditsky2008learning} for selectors. The proof of \cref{thm:pruning-no-along-the-tail} can be found in \cref{proof:pruning-no-along-the-tail}.
\begin{theorem}
    \label{thm:pruning-no-along-the-tail}
    There exists a dictionary $\cF\subset \cM$ of size $ M=2 $ such that the following holds.
    Let $\fhat = \cA(\delta,\rS,\cF)$ be any pruning-based estimator
    whose threshold satisfies that, for every $\delta\in(0,1)$ and
    $M\in\NN$, there exist constants $B_M(\delta)>0$ and $m_0\in\NN$
    such that
    $\threshold \leq \frac{B_M(\delta)}{\sqrt m}$ for all $m\geq m_0$.
    Then, for any fixed $\delta\in(0,1)$, it holds for all sufficiently large $n\geq n_0(\delta)$,
    \[
        \sup_{P\in\Pall}
        \PP_{\rS\sim P^n}\prn{
            \excessRisk(\fhat)
            \geq \frac{B_M(\delta)}{2\sqrt n}
        }
        \geq \frac{1}{2}\prn{1-\Phi(\sqrt{3} B_M(\delta))}.
    \]
    Here $\Phi$ denotes the standard Gaussian cumulative distribution function.
\end{theorem}
\begin{remark}
    Note that this yields yet another proof that the ERM is not minimax optimal.
\end{remark}
Pruned-convex ERM and the midpoint estimator both satisfy the stronger assumption of \cref{thm:pruning-no-along-the-tail}, with $ B_{M}(\delta)\asymp \sqrt{\log{(4/\delta)}} $, which means that they do not achieve minimax optimality \emph{along the tail} for any fixed confidence level depending on the $\delta$ provided to them.
\begin{corollary}\label{cor:pc-ERM-midpoint-no-along-the-tail}
    There exists a dictionary $\cF\subset\cM$ of size $M=2$ and constants $c_1,c_2>0$ such that for any $\delta\in(0,1)$ and $\fhat\in\{\prunedERM,\midpoint\}$, where we denote the pruned-convex ERM $\prunedERM\equiv \cA_{\operatorname{PC}}(\delta,\rS,\cF)$ and the midpoint estimator $\midpoint\equiv \cA_{\circ}(\delta,\rS,\cF)$, for sufficiently large $n\geq n_0(\delta)$, we have
    \begin{equation*}
        \sup_{P\in\Pall} \PP_{\rS\sim P^n}\prn{\excessRisk(\fhat) \geq \frac{c_1\sqrt{\log(4/\delta)}}{2\sqrt{n}}} \geq \frac{1}{2}\prn{1-\Phi( c_2\sqrt{3\log(4/\delta)})}.
    \end{equation*}
    Neither pruned-convex ERM nor the midpoint estimator achieves minimax optimality \emph{along the tail}.
\end{corollary}
\begin{proof}
    This follows from \cref{thm:pruning-no-along-the-tail} and \cref{lem:pruned-convex-ERM-is-pruning-based,lem:midpoint-is-pruning-based} with $M=2$, $c_1=\min\{C_1,2a\}$ and $c_2=\max\{C_1,2a\}$.
\end{proof}
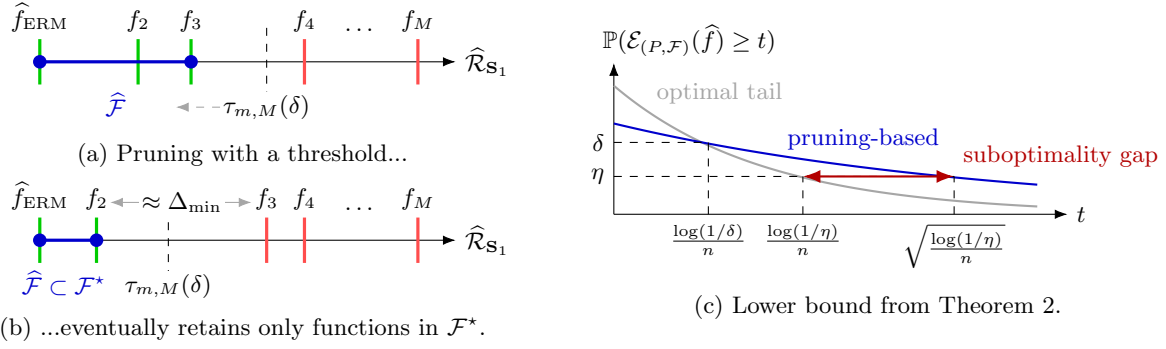
\begin{figure}
    \centering
    \begin{subfigure}[c]{0.4\textwidth}
        \begin{tikzpicture}[
    >=Latex,
    font=\small,
    x=1cm,y=1cm,
good/.style={
    inner sep=0pt,
    minimum height=6mm,
    minimum width=2pt,
    path picture={
        \draw[green!80!black, very thick]
            (path picture bounding box.south) --
            (path picture bounding box.north);
    }
},
bad/.style={
    inner sep=0pt,
    minimum height=6mm,
    minimum width=2pt,
    path picture={
        \draw[red!70, very thick]
            (path picture bounding box.south) --
            (path picture bounding box.north);
    }
},
sel/.style={
    inner sep=0pt,
    minimum height=6mm,
    minimum width=2pt,
    path picture={
        \draw[green!80!black, very thick]
            (path picture bounding box.south) --
            (path picture bounding box.north);
    }
}
]

\draw[->] (0,0) -- (5.5,0) node[right] {$\Rhat_{\rS_1}$};

\node[good] (g) at (0,0) {};

\draw[black,dashed] (3,-0.5) -- (3,0.5);
\node[black, fill=white,inner sep=1pt] (t) at (3,-0.6) {$\threshold$};
\draw[->,gray!70,dashed] (t) -- (1.8,-0.6);

\node[sel] (s1) at (1.3,0) {};
\node[sel] (s2) at (2,0) {};

\draw[very thick,blue!80!black] (g.center) -- (s2.center);
\node[circle,fill=blue!80!black,inner sep=1.8pt] (fh) at (g) {};
\node[circle,fill=blue!80!black,inner sep=1.8pt] (fh) at (s2) {};
\node[blue!80!black,below=8pt] at (1,0) {$\cFhat$};

\node[bad] (b1) at (3.5,0) {};
\node[bad] (b2) at (5,0) {};

\node[above=8pt] at (g) {$\erm$};
\node[above=8pt] at (s1) {$f_2$};
\node[above=8pt] at (s2) {$f_3$};
\node[above=8pt] at (b1) {$f_4$};
\node[above=8pt] at ($(b1)!0.5!(b2)$){\dots};
\node[above=8pt] at (b2) {$f_M$};

\end{tikzpicture}
        \caption{Pruning with a threshold...}
        \label{subfig:pruning-conv}
        \begin{tikzpicture}[
    >=Latex,
    font=\small,
    x=1cm,y=1cm,
good/.style={
    inner sep=0pt,
    minimum height=6mm,
    minimum width=2pt,
    path picture={
        \draw[green!80!black, very thick]
            (path picture bounding box.south) --
            (path picture bounding box.north);
    }
},
bad/.style={
    inner sep=0pt,
    minimum height=6mm,
    minimum width=2pt,
    path picture={
        \draw[red!70, very thick]
            (path picture bounding box.south) --
            (path picture bounding box.north);
    }
},
sel/.style={
    inner sep=0pt,
    minimum height=6mm,
    minimum width=2pt,
    path picture={
        \draw[green!40!black, very thick]
            (path picture bounding box.south) --
            (path picture bounding box.north);
    }
}
]

\draw[->] (0,0) -- (5.5,0) node[right] {$\Rhat_{\rS_1}$};

\node[good] (g) at (0,0) {};
\node[good] (g2) at (0.75,0) {};

\draw[black,dashed] (1.7,-0.5) -- (1.7,0.5);
\node[black, fill=white,inner sep=1pt] (t) at (1.7,-0.6) {$\threshold$};

\draw[very thick,blue!80!black] (g.center) -- (g2.center);
\node[circle,fill=blue!80!black,inner sep=1.8pt] (fh) at (g) {};
\node[circle,fill=blue!80!black,inner sep=1.8pt] (fh) at (g2) {};
\node[blue!80!black,below=8pt] at (0.32,0) {$\cFhat\subset \cF^\star$};

\node[bad] (b2) at (3.5,0) {};
\node[bad] (b3) at (5,0) {};
\node[bad] (b1) at (3,0) {};

\node[above=8pt] at (g) {$\erm$};
\node[above=8pt] at (g2) {$f_2$};
\node[above=8pt] at (b1) {$f_3$};
\node[above=8pt] at (b2) {$f_4$};
\node[above=8pt] at ($(b2)!0.5!(b3)$){\dots};
\node[above=8pt] at (b3) {$f_M$};

\draw[<->,gray!70] (0.93,0.5) -- (2.8,0.5);
\node[fill=white,inner sep=1pt] at (1.88,0.5) {$\approx \Deltamin$};

\end{tikzpicture}
         \vspace{-0.5cm}
        \caption{...eventually retains only functions in $\cFstar$.}
        \label{subfig:pruning-collapse}
    \end{subfigure}
    \hfill
    \begin{subfigure}[c]{0.55\textwidth}
        \centering
        \begin{tikzpicture}[
    >=Latex,
    x=1cm,y=1cm,
    font=\small
]

\draw[->] (0,0) -- (6,0) node[right] {$t$};
\draw[->] (0,0) -- (0,2);
\node[] (Y) at (1,2.3) {$\mathbb P(\excessRisk(\fhat)\ge t)$};

\draw[gray!70,thick,domain=0:5.6,samples=200,smooth]
    plot (\x,{1.7*exp(-0.5*\x)});
\draw[blue!80!black,thick,domain=0:5.6,samples=200,smooth]
    plot (\x,{1.2*exp(-0.2*\x)});
\node[gray!70] at (1.4,1.6) {optimal tail};
\node[blue!80!black] at (3.3,1) {pruning-based};

\draw[dashed] (4.5,0) -- (4.5,0.5);
\draw[dashed] (2.5,0) -- (2.5,0.5);
\draw[dashed] (1.25,0) -- (1.25,0.95);
\draw[dashed] (0,0.5) -- (4.5,0.5);
\draw[dashed] (0,0.95) -- (1.2,0.95);
\draw[<->,red!70!black, thick] (2.5,0.5) -- (4.5,0.5);

\node[below] at (4.5,0){$\sqrt{\frac{\log(1/\eta)}{n}}$};
\node[below] at (1.25,0){$\frac{\log(1/\delta)}{n}$};
\node[below] at (2.5,0){$\frac{\log(1/\eta)}{n}$};
\node[left] at (0,0.5){$\eta$};
\node[left] at (0,0.95){$\delta$};

\node[red!70!black,above right] at (4.5,0.5) {suboptimality gap};

\end{tikzpicture}
        \caption{Lower bound from \cref{thm:pruning-subexp-tail-strong}.}
        \label{subfig:pruning-lower-bound}
    \end{subfigure}
    \caption{Pruning with a threshold (\protect\ref{subfig:pruning-conv}) achieves exponential rates by eventually retaining functions only in $\cFstar$ (\protect\ref{subfig:pruning-collapse}), but despite perhaps achieving the optimal tail at confidence $\delta$, it cannot achieve it uniformly and there exists an $\eta$ (depending on $\delta$) where it is lower bounded by $\sqrt{\log(M/\eta)/n}$ (\protect\ref{subfig:pruning-lower-bound}).}
    \label{fig:pruning}
\end{figure}

To summarize, the ERM argument generalizes to all pruning-based estimators, and they do achieve zero excess risk with exponential probability, essentially due to the same separation from \cref{lem:Hoeffding-separation}. Together with pointwise minimax optimality, this seems to \emph{almost} resolve the tension between universal and uniform guarantees. However, \cref{thm:pruning-no-along-the-tail,thm:pruning-subexp-tail-strong} show that this approach is inherently limited to pointwise guarantees, and there must always be a part of the minimax tail that is suboptimal.

This suggests studying algorithms that are already known to achieve minimax optimality along the whole tail. To that end, in the next sections we study online-to-batch conversions, the star estimator, and $Q$-aggregation, all of which are known to achieve optimal minimax tails.

\subsection{Online-to-Batch Conversion by Averaging}
\label{subsec:online-to-batch}
Next, we show that a large class of algorithms based on averaging sequential predictors cannot achieve exponential rates.
Specifically, we consider the class of algorithms of the following form. Let $\fhat^{(t)}$ be any deterministic estimator that has seen the first $t$ samples (where $t=0$ means it has seen none), is independent of $n$, and takes values in $[0,1]$. Let the averaged estimator be
\begin{equation}
\label{eqn:sequential-class}
     \fhat = \frac{1}{n+a} \sum_{t=b}^n \fhat^{(t)} \qquad \text{for some } b\in\{0,\ldots,n\}, \, a\in\ZZ,\,  a\geq 1-b,
\end{equation}
where $a,b$ are also independent of $n$.
The formulation of this class is motivated by the following three well-known estimators, which are in the form of \cref{eqn:sequential-class}.
\begin{itemize}[leftmargin=*,itemsep=3pt,topsep=2pt]
    \item The \emph{progressive mixture} (PM) estimator $\PM$ \cite{catoni2004statistical,juditsky2008learning,mourtada2023local} is defined as $\PM:= \frac{1}{n+1} \sum_{t=0}^n \EW^{(t)}$,
    where $\EW^{(t)}$ is defined as the exponential weights estimator from \eqref{eqn:exponential-weights-def} using the first $t$ samples. It is known to be minimax optimal in expectation, but not in deviation.
    \item The \emph{Bernstein online aggregation (BOA)} estimator $\BOA=\frac{1}{n+1}\sum_{t=0}^n \fhat^{(t)}$
    from \cite{wintenberger2017optimal} adds a second-order correction to the exponential weights update that accounts for the variance of each expert’s excess loss (similar to \cref{eqn:Jensen-gap}). It is minimax optimal along the tail.
    \item The \emph{sequential estimator with shifted loss (SQ)} $\Sequential=\frac{1}{n}\sum_{t=1}^n \fhat^{(t)}$ from \cite{van2023high} performs a similar second-order correction and
    was also shown to satisfy minimax optimality along the tail.
\end{itemize}

The following theorem shows that all sequential-to-batch conversions of the form \eqref{eqn:sequential-class} that average estimators after a fixed (sample-size-independent) burn-in phase cannot achieve exponential rates. The reason is that by an argument similar to \cref{lem:exponential-is-optimal}, the estimator $\fhat^{(b)}$ must have risk lower bounded by some constant independent of $n$, and its contribution to the averaged estimator only decays polynomially.
In contrast, if we employ a sample-dependent burn-in phase, such as $b=n/2$, combining \cref{prop:EW-exponential} with a union bound reveals that the progressive mixture estimator achieves exponential rates with exponential probability. However, to the best of our knowledge, it is currently unknown whether this prefix-averaged progressive mixture estimator retains its minimax optimality in expectation. Resolving whether this estimator, or similarly modified versions of the others, can simultaneously achieve both exponential rates and minimax optimality remains open.
\begin{theorem}[Online-to-batch averaging cannot achieve exponential rates]
\label{thm:online-to-batch-averaging-impossibility}
    Suppose that the estimator $\fhat$ is of the form in \cref{eqn:sequential-class}. Then there exists $(P,\cF) \in \Theta_2$ such that
    \begin{equation*}
        \forall n\in \NN, n\geq b:\qquad \EE_{\rS\sim P^n} \brk{\excessRisk(\fhat)} \geq \frac{2^{-(b+3)}}{(n+a)^2},
    \end{equation*}
    and in particular, the estimator does not achieve exponential rates (of any kind).
\end{theorem}
\begin{proof}
    First, note that because $b\in\crl{0,\ldots,n}$, we have $n+a \geq n+1-b \geq 1$.
    The proof idea then is very similar to \cref{lem:exponential-is-optimal} or Proposition 1 in \cite{attias2024universal}.
    Take $\cX=\crl{0,1}$ and let $\cF=\crl{f_0,f_1}$ with $f_0\equiv 0$ and $f_1 = \one\crl{x=1}$.
    Let $P_0,P_1$ be defined by $X\sim \Bernoulli(1/2)$ and $Y=f_0(X)$ or $Y=f_1(X)$, respectively. Then, for any predictor $f$, we have under $P_\theta$ with $\theta\in\crl{0,1}$
    \begin{equation*}
        \cR_{P_\theta}(f)-\cR_{P_\theta}(f_\theta) \geq \frac{1}{2}(f(1)-\theta)^2.
    \end{equation*}
    Now define the event $E= \crl{X_1=0,\ldots,X_b=0}$, which has probability
    $\PP\prn{E}=2^{-b}$ under both $P_0,P_1$.
    On $E$, $\fhat^{(b)}$ is a deterministic function, and clearly $\fhat^{(b)}\equiv\fhat^{(b)}(X_1,\ldots,X_b)=\fhat^{(b)}(0,\ldots,0)$ has no information about $\theta$ because all observed $Y$ values were $0$, which is the same under both distributions. Hence, we may choose $\theta$ depending on $\fhat^{(b)}(1)$ so that
    \begin{equation*}
        \abs{\fhat^{(b)}(1)-\theta} \geq \frac{1}{2}.
    \end{equation*}
    Consider the two cases $\theta=1$ and $\theta=0$ as a function of $\fhat^{(b)}(1)$.

    \paragraph{Case $\theta=1$.} Then $\fhat^{(b)}(1)\leq 1/2$, and since $\fhat^{(t)}(1)\in[0,1]$ for all $t$, we obtain $\fhat(1) = \frac{1}{n+a} \sum_{t=b}^n \fhat^{(t)}(1)  \leq \frac{1/2+(n-b)}{n+a}\leq 1$, which implies $\frac{1}{2}(\fhat(1)-\theta)^2 \geq \frac{1}{2}\prn{\frac{a+b-1/2}{n+a}}^2\geq  \frac{1}{8(n+a)^2}$.

    \paragraph{Case $\theta=0$.} Then $\fhat^{(b)}(1)\geq 1/2$, and since $\fhat^{(t)}(1)\in[0,1]$ for all $t$, we obtain $\fhat(1) = \frac{1}{n+a} \sum_{t=b}^n \fhat^{(t)}(1)  \geq \frac{1/2}{n+a}\geq 0$, which implies $\frac{1}{2}(\fhat(1)-\theta)^2 \geq \frac{1}{8(n+a)^2}$.

    Thus, for at least one of $P\in\{P_0,P_1\}$ and all $n\geq b$, we have
    \begin{equation*}
        \EE_{\rS\sim P^n}\brk{\excessRisk(\fhat)} \geq \PP_{\rS\sim P^n}\prn{E} \frac{1}{8(n+a)^2} \geq \frac{2^{-b}}{8(n+a)^2} = \frac{2^{-(b+3)}}{(n+a)^2},
    \end{equation*}
    which concludes the proof.
\end{proof}

Since the three estimators outlined above are of this form, we get as an immediate consequence of \cref{thm:online-to-batch-averaging-impossibility} that these estimators do not achieve exponential rates.

\begin{corollary}\label{cor:PM-exponential}
    The progressive mixture estimator, the Bernstein online aggregation estimator, and the sequential estimator with shifted loss do not achieve exponential rates: for each one, there exists $(P,\cF)\in\Theta_2$ such that $\EE_{\rS\sim P^{n}}[\excessRisk(\PM)] \geq \frac{1}{8(n+1)^2}$, $\EE_{\rS\sim P^{n}}[\excessRisk(\BOA)] \geq \frac{1}{8(n+1)^2}$, and $\EE_{\rS\sim P^{n}}[\excessRisk(\Sequential)] \geq \frac{1}{16n^2}$.
    No version of these estimators with a sample-size-independent burn-in time achieves exponential rates.
\end{corollary}
\begin{proof}
    This follows from \cref{thm:online-to-batch-averaging-impossibility} with $a=1$, $b=0$ for PM and BOA, and $a=0$, $b=1$ for SQ.
\end{proof}

\subsection{Star Estimation}
\label{subsec:star-estimator}

\emph{Star estimation}, another popular aggregation procedure \cite{audibert2007progressive,audibert2009fast,gaiffas2009hyper,kanade2024exponential,liang2015learning}, first builds the star hull of the dictionary around the dictionary ERM and then performs ERM on this star hull, formally
\begin{equation*}
    \starestimator\in\argmin_{f\in \starhull(\cF,\erm)} \Rhat_\rS(f) \quad \text{where} \quad \starhull(\cF,h):=\crl{\lambda f +(1-\lambda)h : \lambda\in [0,1], f\in \cF },
\end{equation*}
where we recall that $\erm\in\argmin_{f\in\cF} \Rhat_\rS(f)$. It is known to be minimax optimal along the tail \cite{audibert2007progressive,kanade2024exponential}. We now show that the star estimator does not achieve exponential rates.

\begin{proposition}
\label{prop:star-linear}
    The star estimator does not achieve any exponential universal learning rate. There exists $(P,\cF)\in\Theta_2$ such that for all $n\in\NN$, it holds $\EE_{\rS\sim P^n}[\excessRisk(\starestimator)] = \frac{1}{8n}$.
\end{proposition}

\begin{proof}
We begin by constructing the pair $(P,\cF)\in\Theta_2$ on which the bound holds.
Let $\cX=\{0,1\}$ and $X\sim\uniform{\crl{0,1}}$.
Let $Y\sim\Bernoulli(1/2)$ independently of $X$, so that $Y\in[0,1]$.
Define two predictors $f_0(x)\equiv 1/2$ and $f_1(x) := \mathbf 1\{x=1\}$,
and let $\cF=\{f_0,f_1\}$. We compute their risks under $P$:
\[
\cR_P(f_0)
= \EE[(\tfrac{1}{2}-Y)^2]
= \tfrac14,
\qquad \cR_P(f_1)
= \tfrac{1}{2}\,\EE[(1-Y)^2]
 + \tfrac{1}{2}\,\EE[Y^2]
= \tfrac{1}{2}.
\]
Hence $\fstar = f_0$, and the excess risk of $f_1$ is $\excessRisk(f_1)=1/4$.

Let $h := f_1 - f_0$, so $h(1)=1/2$ and $h(0)=-1/2$.
For $\lambda\in[0,1]$ define
$
g_\lambda := f_0 + \lambda h.
$
One verifies that for this dictionary, $\starhull(\cF,\erm) = \{g_\lambda : \lambda\in[0,1]\}$,
regardless of which $f\in\cF$ is chosen as $\erm$.
Thus, the star estimator always has the form $\starestimator = g_{\lambdahat}$ where
$\lambdahat\in[0,1]$ minimizes the empirical risk $\lambda\mapsto \Rhat_\rS(g_\lambda)$.
The empirical risk is
\begin{align*}
    n\Rhat_\rS(g_\lambda)=\sum_{i=1}^{n}(g_\lambda(X_i)-Y_i)^2=\sum_{i=1}^{n}(f_0(X_i)+\lambda h(X_i)-Y_i)^2,
\end{align*}
which is a positive quadratic in $ \lambda $. Thus, the empirical risk as a function of $\lambda\in\RR$ is minimized by differentiating with respect to $ \lambda $  and solving for zero. This yields the minimizer $$
\widetilde\lambda:= -\frac{\sum_{i=1}^{n}h(X_{i})(f_0(X_{i})-Y_i)}{\sum_{i=1}^{n}h(X_{i})^2}=\frac{1}{n}\sum_{i=1}^{n}4h(X_{i})(Y_i-f_0(X_{i})),
$$
where we have used that $ h^{2}=1/4 $. The above minimizer is a number in $[-1,1]$ since $Y_i-f_0(X_{i}),h(X_{i})\in\{  -1/2,1/2\}  $. Since we minimized a positive quadratic, the constrained minimizer $\lambdahat \in [0,1] $ must be the leftmost point in this interval closest to the unconstrained minimizer, implying that $\lambdahat=\max\{0,\frac{1}{n}\sum_{i=1}^{n}4h(X_{i})(Y_i-f_0(X_{i}))\}$. We notice that $4h(X_{i})(Y_i-f_0(X_{i}))$ are i.i.d.\ Rademacher random variables, so $\lambdahat$ is equal in distribution to the positive part of $S_{n}/n$, where $ S_{n}=\sum_{i=1}^{n}\xi_i $ is a sum of $n$ i.i.d.\ Rademacher random variables.
Now, for any $\lambda$,
\begin{equation}
\label{eq:excess_risk_lambda}
    \excessRisk(g_\lambda)=\cR_P(g_\lambda) - \cR_P(f_0)
= 2\lambda\,\EE[h(X)(f_0(X)-Y)]
  + \lambda^2\,\EE[h(X)^2] = \frac{1}{4}\lambda^2,
\end{equation}
where the last equality holds since $h(X)$ takes values $\pm 1/2$ with equal probability and $f_0(X)-Y$ is
independent of $X$ with mean $0$ so $\EE[h(X)(f_0(X)-Y)]=0$ and
$\EE[h(X)^2]=1/4$.
Substituting $\lambdahat$ into the above expression for the excess risk \eqref{eq:excess_risk_lambda}, we get that the excess risk of the star estimator is distributed as $\frac{1}{4n^{2}}(S_{n}\ind\{S_{n}\geq 0\})^2$. Since $S_{n}$ is a sum of i.i.d.\ Rademacher random variables, symmetry gives $\EE[S_{n}^2\ind\{S_{n}\geq 0\}]=\frac{1}{2}\EE[S_{n}^2]=\frac{n}{2}$. Hence, $\EE[\excessRisk(\starestimator)]=\frac{1}{4n^{2}} \frac{n}{2}=\frac{1}{8n}$, which concludes the proof.
\end{proof}

Hence, similar to the online-to-batch conversions, while star estimation achieves minimax optimality along the whole tail, it cannot achieve exponential rates of any kind because it cannot reliably commit to a single dictionary element even when the evidence is overwhelming. That leaves us with $Q$-aggregation.

\subsection{$Q$-aggregation}
\label{subsec:Q-aggregation}

Recall the notation $f_\rho = \EE_{k\sim \rho}f_k$ for some distribution $\rho \in \triangle_M=\{\rho\in[0,1]^M:\sum_{k=1}^M\rho_k = 1\}$ over $[M]$.
In \cite{dai2012deviation,lecue2014optimal}, the $Q$-aggregation estimator is introduced using the distribution $\rhohat_{Q}\in \triangle_{M}$ defined by
\begin{equation}
\label{eqn:Q-aggregation-definition}
    \rhohatQ \in \argmin_{\rho\in\triangle_{M}} \crl{\frac{1}{2}\Rhat_\rS(f_\rho)+\frac{1}{2}\EE_{k\sim \rho} \Rhat_\rS(f_k) + \frac{\beta}{n} K_\phi(\rho,\pi)} \quad \text{where} \quad K_\phi(\rho,\pi) =  \sum_{k=1}^M \rho_k \log\prn{\frac{\phi(\rho_k)}{\pi_k}},
\end{equation}
and is given by $\Qestimator=f_{\rhohatQ}=\EE_{k\sim \rhohatQ}f_k$. Here $\beta$ is a temperature parameter and $\pi$ is a prior distribution, chosen for now to be uniform. The function $\phi:[0,1]\to\RR$ yields different penalties; in particular, if $\phi(x)=x$ we have $K_\phi(\rho,\pi)=\KL(\rho,\pi)$, whereas for a uniform prior and $\phi\equiv 1$, the term $K_\phi$ is constant in $\rho$ and does not affect the optimization problem. In that case, the temperature is also irrelevant.
In \cite{lecue2014optimal}, the estimator using $\phi\equiv 1$ was shown to be minimax along the tail; we restate the formal result for convenience in \cref{thm:Qaggregationrestatementlecue2014optimal} of the appendix. In \cite{mourtada2023local}, the estimator using $\phi(x)=x$ was shown to be minimax optimal along the tail as a corollary of a local risk bound.

For $Q$-aggregation with $\phi\equiv 1$ and uniform prior $\pi=(1/M,\ldots,1/M)$, we can prove the following result, which is the main result of this section. The proof of \cref{thm:Q-exponential} is in \cref{proof:Q-exponential}.
\begin{theorem}
\label{thm:Q-exponential}
    The $Q$-aggregation estimator with $\phi\equiv 1$ and uniform prior $\pi$ achieves zero excess risk with exponential probability on $\Theta_M$. For every $(P,\cF)\in\Theta_M$, there exist $C,c>0$ such that for all $n\in\NN$
    \begin{equation*}
        \PP_{\rS\sim P^n}\prn{\excessRisk(\Qestimator) >0}\leq Ce^{-c n}.
    \end{equation*}
\end{theorem}

Combining \cref{thm:Q-exponential} with the minimax optimality along the tail from \cite{lecue2014optimal} shows that there is \emph{no trade-off between exponential universal rates and minimax optimality}. Remarkably, among all estimators considered in this work, only $Q$-aggregation with $\phi\equiv 1$ and uniform prior achieves both properties.

\begin{remark}
    It can be shown that the $Q$-aggregation estimator with a KL penalty (that is, $\phi(x)=x$) achieves an exponential rate with exponential probability, but it does \emph{not} achieve zero excess risk with exponential probability; there exists $(P,\cF)$ such that $\excessRisk(\Qestimator)>0$ almost surely by a similar argument as the exponential weights estimator (\cref{prop:EW-exponential}):
    Take any distribution $P$ and dictionary $\cF$ so that $\excessRisk(f_\rho)\leq 0 \Rightarrow \exists j\in[M]: \rho_{j}=0$. At the boundary, we have that
    \begin{equation*}
        \frac{\partial}{\partial \rho_j} \KL(\rho,\pi) = \log\frac{\rho_j}{\pi_j} + 1 \to -\infty \quad \text{as }\rho_j\to 0
    \end{equation*}
    while the other terms in the objective remain bounded, so the KL term in the objective of the optimization problem of the $Q$-estimator forces the solution to be in the interior of $\triangle_M$. That implies $\excessRisk(\Qestimator)>0$ almost surely.
\end{remark}

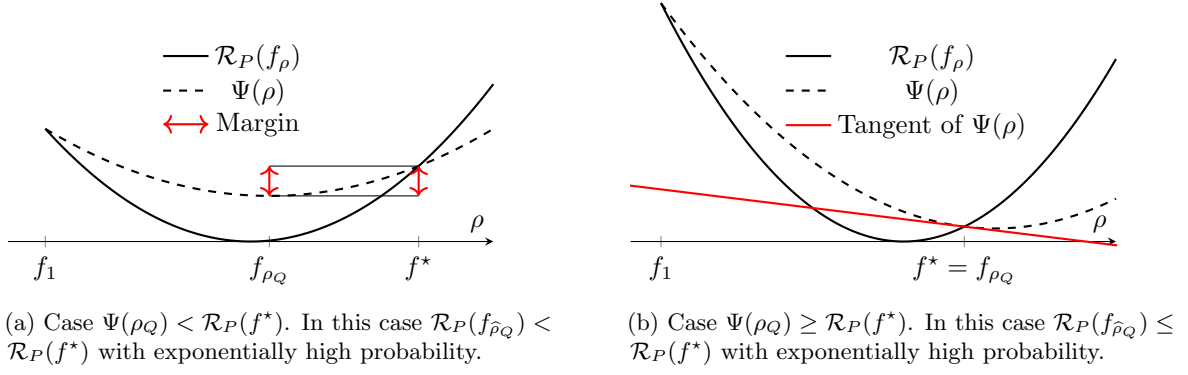
\begin{figure}
    \centering
    \begin{subfigure}[b]{0.45\textwidth}
        \begin{tikzpicture}
\begin{axis}[
    width=8cm,
    height=7cm,
    domain=0:1.2,
    samples=200,
    axis x line=middle,
    axis y line=none,
    xlabel={$\rho$},
    ylabel={$y$},
    xmin=-0.1, xmax=1.2,
    ymin=-0.1, ymax=1,
    xtick={0,0.6,1},
    xticklabels={$f_1$,$f_{\rho_Q}$,$\fstar$},
    ytick=\empty,
    legend style={draw=none, fill=none, at={(0.3,0.6)}, anchor=north west},
]

\addplot[thick]
    {(x-0.55)^2};
\addlegendentry{$\cR_P(f_\rho)$}

\addplot[thick, dashed]
    {0.5*(x-0.55)^2 + 0.5*x*(0.45)^2 + 0.5*(1-x)*(0.55)^2};
\addlegendentry{$\Psi(\rho)$}
\draw[-] (axis cs:0.6,0.1225) -- (axis cs:1,0.1225);
\draw[-] (axis cs:0.6,0.2025) -- (axis cs:1,0.2025);
\draw[<->, red, thick] (axis cs:0.6,0.2025) -- (axis cs:0.6,0.1225);
\draw[<->, red, thick] (axis cs:1,0.2025) -- (axis cs:1,0.1225);

\addlegendimage{<->, red, thick}
\addlegendentry{Margin}

\end{axis}
\end{tikzpicture}
        \caption{Case $\Psi(\rho_Q)<\cR_P(\fstar)$. In this case $\cR_P(f_{\rhohatQ})<\cR_P(\fstar)$ with exponentially high probability.}
        \label{subfig:Q-proof-case-1}
    \end{subfigure}
    \hspace{0.8cm}
    \begin{subfigure}[b]{0.45\textwidth}
        \begin{tikzpicture}
\begin{axis}[
    width=8cm,
    height=7cm,
    domain=0:1.5,
    samples=200,
    axis x line=middle,
    axis y line=none,
    xlabel={$\rho$},
    ylabel={$y$},
    xmin=-0.1, xmax=1.5,
    ymin=-0.1, ymax=1,
    xtick={0,1},
    xticklabels={$f_1$,$\fstar=f_{\rho_Q}$},
    ytick=\empty,
    legend style={draw=none, fill=none, at={(0.3,0.6)}, anchor=north west},
]

\addplot[thick]
    {(x-0.8)^2};
\addlegendentry{$\cR_P(f_\rho)$}

\addplot[thick, dashed]
    {0.5*(x-0.8)^2 + 0.5*x*(0.2)^2 + 0.5*(1-x)*(0.8)^2};
\addlegendentry{$\Psi(\rho)$}

\addplot[red, thick, domain=-1:2]
    {0.14 - 0.1*x};
\addlegendentry{Tangent of $\Psi(\rho)$}
\end{axis}
\end{tikzpicture}
        \caption{Case $\Psi(\rho_Q)\geq \cR_P(\fstar)$. In this case $\cR_P(f_{\rhohatQ})\leq \cR_P(\fstar)$ with exponentially high probability.}
        \label{subfig:Q-proof-case-2}
    \end{subfigure}
    \caption{Illustration of the two cases of the $Q$-aggregation estimator, by which the proof of \cref{thm:Q-exponential} is structured. In both cases, there is a hidden margin the estimator can exploit for exponential rates: in the first case (\ref{subfig:Q-proof-case-1}), there is a margin between $\Psi(\rho_Q)$ and $\cR_P(\fstar)$, implying that $\cR_P(f_{\rho_Q})-\cR_P(\fstar)<0$ by a fixed margin. In the second case (\ref{subfig:Q-proof-case-2}), every direction away from \(\fstar\) is either strictly unfavorable at first order (it is a direction with a ``nonzero gradient'') with some margin, or harmless.}
    \label{fig:Q-proof-two-functions}
\end{figure}

\paragraph{Proof Outline.}
The proof analyzes the population unregularized (i.e., for uniform prior) objective
$$
\Psi(\rho) =  \frac{1}{2}\cR_P(f_\rho)+\frac{1}{2}\EE_{k\sim \rho} \cR_P(f_k)
$$
relative to an optimal model $\fstar$, and then transfers the conclusion to the empirical $Q$-aggregation solution $\rhohatQ$ via a convergence lemma.
There are two cases, cf.\ \cref{fig:Q-proof-two-functions}. In the first case (\ref{subfig:Q-proof-case-1}), there exists a population minimizer $\rho_Q\in \argmin_{\rho\in \triangle_{M}} \Psi(\rho)$ such that $\Psi(\rho_Q)<\cR_P(\fstar)$, which is shown to imply $\cR_P(f_{\rho_Q})<\cR_P(\fstar)$. Hence, there is a fixed margin between the population minimum and the risk of the best dictionary element. The convergence lemma then shows that $f_{\rhohatQ}$ is close to $f_{\rho_Q}$ in $L_2(P_X)$-norm and that eventually $\cR_P(f_{\rhohatQ})<\cR_P(\fstar)$ for $n$ sufficiently large. So the margin yields exponential rates.
In the second case (\ref{subfig:Q-proof-case-2}), every population minimizer \(\rho_Q\) satisfies \(\Psi(\rho_Q)\ge \cR_P(\fstar)\). Here the key point is that every direction away from \(\fstar\) is either strictly unfavorable at first order or else degenerate or harmless for the risk. For the strictly unfavorable directions, the population first-order term has a fixed positive margin, so with exponentially high probability the empirical objective has the same sign, and the KKT conditions force the corresponding coordinates of \(\rhohatQ\) to vanish. The remaining directions are either equivalent to \(\fstar\) or contribute non-positively once \(\rhohatQ\) is close to \(\rho_Q\). This gives \(\cR_P(f_{\rhohatQ})\le \cR_P(\fstar)\) for all sufficiently large \(n\). The exponential rate again comes from preserving a fixed first-order margin, now only in the directions that could otherwise move the empirical solution away from the correct corner.
Depending on the geometry of the population problem, the empirical $Q$-aggregation estimator either asymptotically improves on $\fstar$ or is asymptotically no worse than $\fstar$ with exponential probability.

\paragraph{Sparsity.} The proof of \cref{thm:Q-exponential} also implies the following result about \emph{sparsity} of the $Q$-aggregation estimator, which may be of independent interest. The proof of \cref{prop:Q-aggregation-sparsity} is in \cref{proof:Q-aggregation-sparsity}.
\begin{proposition}
\label{prop:Q-aggregation-sparsity}
    There exist two universal constants $c,C>0$ such that the following holds. Let $\rho_Q\in\argmin_{\rho\in\triangle_M} \Psi(\rho)$ and $\fstar\in \cFstar$ be arbitrary and fixed.
    Define $\operatorname{supp}(\rho)=\crl{j\in[M]:\rho_j>0}$ and $\gamma = \tfrac{1}{2}\EE_{X\sim P_X}\brk{(f_{\rho_Q}(X)-\fstar(X))^2}$. Then, for every $\eps\in(0,1]$, with probability at least $1-CM^2\exp(-c\frac{\eps^2n}{M^2})$,
    \begin{equation*}
        \operatorname{supp} (\rhohatQ) \subseteq \crl{j\in[M]: \excessRisk(f_j)+\gamma-\frac{1}{2}\EE_{X\sim P_X}\brk{(f_j(X)-f_{\rho_Q}(X))^2}\leq \sqrt{\eps}+\eps}.
    \end{equation*}
\end{proposition}
In words, $Q$-aggregation is only dense in those coordinates where the excess risk is small enough or where the benefit of mixing outweighs the excess risk. For instance, if $f_{\rho_Q}=\fstar$, then for large enough $n$ the nonzero components of the $Q$-aggregation estimator all satisfy $\excessRisk(f_j)\leq \frac{1}{2} \EE_{X\sim P_X}\brk{(f_j(X)-\fstar(X))^2}$. This implicit sparsity of the $Q$-aggregation estimator helps in both the minimax and universal sense.

That concludes the study of finite hypothesis spaces, as \cref{thm:Q-exponential} shows that there is no conflict between exponential universal rates and minimax optimality for finite dictionaries.

\section{Countably Infinite Hypothesis Spaces: Structural Results}
\label{sec:countably-infinite-hyp-spaces}

Recall that the overarching question of this work is whether exponential rates ever come at the cost of uniform guarantees. In the finite setting, we have answered this question in the negative in full generality: on the problem space $\Theta_M$, the $Q$-aggregation estimator achieves both minimax rates along the tail and zero excess risk with exponential probability.
The (countably) infinite setting is more nuanced. Before answering the question in the infinite case in \cref{sec:best-of-both-worlds}, we present the following structural results that highlight differences between the finite and infinite cases in terms of both learnability and algorithmic principles, motivating the problem formulation of \cref{sec:best-of-both-worlds}.

\begin{enumerate}[leftmargin=*,itemsep=3pt,topsep=2pt]
    \item In the finite case, achieving zero excess risk with exponential probability is guaranteed to be possible. The same is not true in the infinite case.
    Specifically, we show next (in \cref{subsec:arbitrarily-slow-and-near-exponential-lower-bound}) that one cannot generally achieve exponential rates in the problem class $\Theta_\NN$;  in fact, only arbitrarily slow rates are possible, as specified in \cref{thm:arbitrarily-slow-rates-inf-not-realized}. To study exponential rates, we therefore have to restrict the problem space either by restricting the hypothesis classes or the space of distributions.
    However, we also demonstrate that if the space of distributions is left unrestricted, the space of function classes must be restricted substantially (\cref{thm:nearly-exponential-rates-lower-bound}).

    \item Moreover, in the finite case, learnability in the uniform minimax sense is guaranteed to be possible at a decreasing rate of $\log(M)/n$. As we show in \cref{subsec:inf-realized}, this is not the case for infinite function classes, where the rate can be constant (\cref{thm:countable-inf-realized-not-uniformly-learnable}) even when it is possible to learn at nearly exponential rates in the universal sense (\cref{thm:countable-inf-realized-exponential-rates}).

    \item Lastly, in the finite case, algorithms can always output a finite convex combination. However, as we show in \cref{subsec:arbitrarily-slow-convex-and-finite}, there are countably infinite dictionaries for which exponential universal and fast minimax rates are attainable (even by the same algorithm), but any algorithm that outputs a convex or finite combination of models achieves only arbitrarily slow rates (\cref{thm:arbitrarily-slow-rates-convex-finite-combination}).
\end{enumerate}

Together, these results motivate the definition of \emph{learnability in both worlds} (\cref{def:learnable-in-both-worlds}) introduced in \cref{sec:best-of-both-worlds}, where we condition \emph{only} on the existence of two algorithms, $\Amini$ and $\Aexp$, that achieve a vanishing minimax rate and zero excess risk with exponential probability, respectively.

\subsection{Lower Bounds for Arbitrarily Slow and Nearly Exponential Rates}
\label{subsec:arbitrarily-slow-and-near-exponential-lower-bound}
In this section, we demonstrate that, contrary to the finite case, we cannot study exponential rates over the entire space $\Theta_\NN$ of countably infinite hypothesis spaces and all distributions. Hence, in the following sections, we consider the specific subset $\Theta\subseteq \Theta_\NN$ of pairs of hypothesis spaces and families of distributions for which exponential rates are achievable \emph{by definition}.

We start with a no-free-lunch result showing that there exists a countably infinite hypothesis class that can be learned only at arbitrarily slow rates. This result follows from constructing a binary hypothesis class with a certain combinatorial property (an infinite VCL tree). This combinatorial property, together with Theorem~5.11 from \cite{bousquet2021theory}, implies that for any rate function $ R $ and algorithm $ \cA $, there exists a distribution such that the binary classification loss of the algorithm is $ R $ and the infimum of the binary classification loss over all functions in the class is $0$. Using this fact, one can lower bound the squared loss of any algorithm by a universal constant times the binary classification loss of its thresholded prediction and obtain the following result. In the spirit of this result, Attias et al.~\cite[Theorem~13]{attias2024universal} show that, under the expected absolute loss, for every rate function \(R(n)\) with \(R(n)\) nonincreasing and \(nR(n)\) nondecreasing, there exists a hypothesis class \(\mathcal H\) that is learnable at rate \(R(n)\), but not at any rate faster than \(o(R(n))\). The proof of \cref{thm:arbitrarily-slow-rates-inf-not-realized} is in \cref{proof:arbitrarily-slow-rates-inf-not-realized}.

\begin{theorem}
\label{thm:arbitrarily-slow-rates-inf-not-realized}
    There exist a universal constant $c>0$ and a countable function class $ \cF \subset\cM $ such that for any decreasing function $ R(n) $ converging to zero as $ n\rightarrow\infty $ and any learning algorithm $ \cA $, there exists a distribution $ P $ over $ \cX\times [0,1] $ such that, for infinitely many $ n\in\NN $,
    \begin{align*}
    \EE_{\rS\sim P^{n}}\brk{\excessRisk(\cA(\rS))} \geq cR(n).
    \end{align*}
    For this distribution $P$, it holds that $\inf_{f\in\cF}\cR_P(f)=0$,
    but no $ f\in\cF $ achieves zero risk.
\end{theorem}
It may be natural to constrain the set of hypothesis spaces that we consider while remaining entirely agnostic about the family of distributions. We now argue that this would exclude a \emph{very} large class of hypothesis spaces.
In \cite[Theorem 19]{hanneke2026theory}, it is shown that, in agnostic binary classification, no (nontrivial) infinite hypothesis class can be learned at an exponential rate because it must contain a so-called Eluder sequence. As we show next, this result can be translated to the regression setting with squared loss. However, as the next example shows, it cannot hold in the same generality as in binary classification.

\begin{example}
Let
$
\cF=\{x\mapsto 0\}\cup\{x\mapsto 2^{-k}:k\in\NN\}.
$
This class is infinite, but under squared loss it admits exponential rates for every distribution $P$ on $\cX\times[0,1]$: writing $\mu=\EE_P[Y]$, the risk of the constant predictor $a$ is $\cR_P(a)=(a-\mu)^2+\Var_P(Y)$,
so the best-in-class predictor is the nearest point in $A:=\{0\}\cup\{2^{-k}:k\in\NN\}$
to $\mu$. If $\mu=0$, then $Y=0$ almost surely, so the constant predictor $0$ is optimal almost surely and has zero excess risk. Assume now that $\mu>0$. Since the only accumulation point of $A$ is $0$, the set $A$ has no accumulation at $\mu$, and there exists $r_P>0$ such that every projection of any $\nu\in(\mu-r_P,\mu+r_P)$ onto $A$ is optimal in $A$ for the true mean $\mu$. Let $\ahat$ be a projection of the sample mean $\Ybar$ onto $A$ (with arbitrary tie-breaking), with the convention $\ahat=0$ if $\Ybar=0$. Then $\excessRisk(\ahat)=0$ on the event $\{|\Ybar-\mu|<r_P\}$, and  $\excessRisk(\ahat)\le 1$ always. Hence, by Hoeffding,
\[
\EE_{\rS\sim P^n}[\excessRisk(\ahat)]\le
\PP_{\rS\sim P^n}(|\Ybar-\mu|\ge r_P)\le 2\exp\prn{-2r_P^2 n}.
\]
\end{example}
Therefore, for a function class to not be learnable at an exponential rate, we must exclude classes with such accumulation points at the boundary. The next theorem formalizes this idea.

\begin{theorem}[Adaptation of Theorem 19 in \cite{hanneke2026theory}]
\label{thm:nearly-exponential-rates-lower-bound}
Fix $\cF\subseteq \cM$.
Assume there exist constants $\gamma\in (0,1]$, $\varepsilon\in(0,1/2]$, distinct points $x_1,x_2,\ldots\in\cX$,
values $y_1,y_2,\ldots\in[\varepsilon,1-\varepsilon]$, and functions $f_1,f_2,\ldots\in\cF$
so that
\begin{align}
\forall i\in\NN, \forall j<i, \hspace{1.6cm}  f_i(x_j) &= y_j, \label{eq:prefix-reg}\\
\forall i\in\NN, \qquad |f_i(x_i)-y_i| &\ge \gamma. \label{eq:gap-reg}
\end{align}
Then for every deterministic learning algorithm
$
\cA:(\cX\times[0,1])^{*}\to\cM,
$
there exist a function $\psi(n)=o(n)$ and a distribution $P$ on
$\cX\times[0,1]$ such that
$
\EE_{\rS\sim P^n}\brk{\excessRisk(\cA(\rS))}\ge e^{-\psi(n)}
$
for infinitely many $n$.
\end{theorem}
The proof of \cref{thm:nearly-exponential-rates-lower-bound} is in \cref{proof:nearly-exponential-rates-lower-bound}. At its core, the proof relies on the same idea as the proof of Theorem 19 in \cite{hanneke2026theory}. The main difference is the margin condition of \cref{thm:nearly-exponential-rates-lower-bound} due to using squared loss.

\subsection{Universal Learnability does not Imply Uniform Learnability}
\label{subsec:inf-realized}

In this section, we consider tuples $ (P,\cF)$ of distributions and hypothesis spaces such that $ \cF $ is countable and the distribution $ P $ is such that the infimum of the risk with respect to $P$ is attained in $\cF$:
\begin{equation*}
    \Thetainf = \Big\{(P,\cF)\in\Theta_{\NN}: \text{there exists } f\in\cF \text{ such that } \cR_P(f) = \inf_{f'\in\cF}\cR_P(f')\Big\}.
\end{equation*}
We show that this broad class of tuples is not learnable in the minimax sense other than trivially, but is learnable at a nearly exponential universal rate.
We start by showing that the minimax excess risk in expectation for this class of tuples is constant and equal to $ 1/4 $, attained by trivially outputting $ 1/2 $. This is a well-known no-free-lunch result, here stated for squared loss.
\begin{theorem}
\label{thm:countable-inf-realized-not-uniformly-learnable}
There exists a countable hypothesis class $ \cF \subset \cM$ such that for every deterministic learning algorithm $\cA$ and every $n\in\NN$,
\[
    \sup_{P:(P,\cF)\in\Thetainf}\EE_{\rS\sim P^{n}}\brk{\excessRisk(\cA(\rS))}\geq \frac{1}{4},
\]
and there exists a learning algorithm $\cA$ (the trivial algorithm outputting $1/2$) that achieves equality.
\end{theorem}
The proof of \cref{thm:countable-inf-realized-not-uniformly-learnable} is in \cref{proof:countable-inf-realized-not-uniformly-learnable}.
The next theorem shows that the same class of tuples is learnable at a nearly exponential universal rate. This result is similar to results in \cite{pabbaraju2026agnostic,HannekeKMV26} in other settings, and its proof uses their ideas. We prove \cref{thm:countable-inf-realized-exponential-rates} in \cref{proof:countable-inf-realized-exponential-rates}.
\begin{theorem}
\label{thm:countable-inf-realized-exponential-rates}
For any function $\varphi:\NN\to\RR$ with $\varphi(n)\to\infty$ and $\varphi(n)=o(n)$, there exists a learning algorithm $\cA_{\varphi}$ such that for any tuple $(P,\cF)\in\Thetainf$, there exist constants $C,c>0$ depending on $P,\cF$ (for instance through an enumeration of $\cF$) and $\varphi$ such that for any $n\in\NN$,
    \begin{align*}
        \PP_{\rS\sim P^n}\prn{\excessRisk(\cA_{\varphi}(\cF,\rS))>0}\leq Ce^{-c\varphi(n)}.
    \end{align*}
\end{theorem}
Recall that in \cref{thm:arbitrarily-slow-rates-inf-not-realized} the infimum was \emph{not} realized, so there is no contradiction.
This highlights that learnability in the universal sense need not imply learnability in the uniform sense and that there can be a substantial gap between the two notions.
This mirrors the point of Example 2.3 in \cite{bousquet2021theory}, but for squared loss and distributions that need not be realizable.

\subsection{Arbitrarily Slow Rates for Convex and Finite Aggregation}
\label{subsec:arbitrarily-slow-convex-and-finite}
All the methods we have studied for finite hypothesis classes (\cref{tab:minimax-vs-exponential}) output a finite convex combination of functions in the function class. In this section, we show that there exists a problem class $ \Theta $ of function classes and distributions for which any algorithm that outputs either a convex combination or a finite combination of functions in the function class can achieve only arbitrarily slow rates. For the same problem class $ \Theta $, there exists a learning algorithm that simultaneously achieves universal exponential rates and fast uniform rates. This shows that there are problem classes $ \Theta $ for which one cannot obtain either universal or uniform guarantees when restricting to algorithms that output convex or finite combinations of functions in the function class.

The function-class construction used in the following theorem is based on the Cantor-to-regression embedding
trick \cite[Examples 3 and 4]{AttiasHKKV23}; the distribution construction uses
ideas from \cite[Lemma 5.12]{bousquet2021theory}; and the adversarial choice
against convex and finite combinations is based on ideas from
\cite{hogsgaard2026interplay}, which studies the PAC/minimax setting. The upper
bounds follow from an identification argument enabled by the hypothesis class. We prove \cref{thm:arbitrarily-slow-rates-convex-finite-combination} in \cref{proof:arbitrarily-slow-rates-convex-finite-combination}.

\begin{theorem}
\label{thm:arbitrarily-slow-rates-convex-finite-combination}
Let $\cA$ be a learning rule that, for every function class
$\cF\subset\cM$, maps a sample $S\in(\cX\times[0,1])^{*}$ to a predictor
$\cA(S):\cX\to[0,1]$. Assume that $\cA$ satisfies one of the following
two conditions:
\begin{itemize}[leftmargin=*,itemsep=3pt,topsep=2pt]
    \item \emph{Convex combination:} for every reference class
    $\cF\subset\cM$ and every sample $S\in(\cX\times[0,1])^{*}$, there are
    coefficients $\alpha_{f,S}\in[0,1]$, indexed by $f\in\cF$, with
    $\sum_{f\in\cF}\alpha_{f,S}=1$ such that, for every $x\in\cX$,
    \[
        \cA(S)(x)=\sum_{f\in\cF}\alpha_{f,S}f(x).
    \]

    \item \emph{Finite combination with sample-size-dependent width:}
    there is a function
    $k:\NN\to\NN$, depending only on $\cA$, such that for
    every reference class $\cF\subset\cM$ and every sample
    $S\in(\cX\times[0,1])^{n}$, there are functions
    $f_{1,S},\ldots,f_{k(n),S}\in\cF$ such that, for every $x\in\cX$,
    \[
        \cA(S)(x)\in
        \left[
            \min_{i\in[k(n)]} f_{i,S}(x),
            \max_{i\in[k(n)]} f_{i,S}(x)
        \right].
    \]
\end{itemize}
Then there is a universal constant $C>0$ with the following property. For
every function $R:\NN\to(0,\infty)$ with $R(n)\to0$, there exists a function
class $\cF\subset\cM$ and a class of distributions $\cP$
such that, for some distribution $P\in\cP$ and infinitely many $n\in\NN$,
\begin{align*}
    \EE_{\rS\sim P^{n}}\brk{\excessRisk(\cA(\rS))}
    \geq C R(n).
\end{align*}
On the other hand, for the same classes $\cF$ and $\cP$, there exists a learning algorithm
$\cA^\star$ such that
\begin{alignat*}{2}
    \forall P\in\cP,  n\in \NN,  \delta\in (0,1):&& \qquad  \PP_{\rS\sim P^n}
    \prn{\excessRisk(\cA^\star(\rS))>\frac{\log{(1/\delta)}}{n}}
    &\leq \delta, \\
    \text{and} \qquad \forall P\in\cP, \ \exists c>0: && \PP_{\rS\sim P^n}
    \prn{\excessRisk(\cA^\star(\rS))>0}
    &\leq e^{-cn}.
\end{alignat*}
Furthermore, no algorithm can achieve better rates, up to constant factors, in either the minimax or universal sense.\footnote{The minimax rate is $ \Theta(\log(1/\delta)/n) $ when $ n>\log(1/\delta) $, and the best universal rate is $ \Theta(e^{-cn}) $ for some constant $ c>0 $.}
If $\cA$ satisfies the convex-combination condition, $\cF$ can be
chosen independently of $R$.
\end{theorem}

This concludes the structural results. We now turn to whether exponential rates ever come at the cost of uniform guarantees or whether a best-of-both-worlds algorithm generally exists.

\section{Best of Both Worlds for Infinite Hypothesis Classes?}
\label{sec:best-of-both-worlds}

Perhaps the most natural question is whether obtaining exponential rates, \emph{whenever} they are achievable, comes at the cost of uniform guarantees. In this section, we show that it does and prove matching upper and lower bounds that characterize the trade-off in this very agnostic sense.

Consider a family of distributions $\cP$ on $\cX\times [0,1]$. In this section, we make no explicit assumptions about the hypothesis space $\cF$. Instead, we assume the existence of two algorithms: one that learns $\cF$ uniformly over $\cP$ (e.g., at the minimax rate) and one that universally achieves zero excess risk with exponential probability on all distributions in $\cP$. We call this \emph{learnability in both worlds}.
\begin{definition}
\label{def:learnable-in-both-worlds}
For a family of distributions $\cP$ on $\cX\times [0,1]$ and a fixed hypothesis space $\cF\subset \cM$, the space $\Theta= \crl{(P,\cF):P\in \cP}$ is \emph{learnable in both worlds} if it is
\begin{enumerate}[leftmargin=*,itemsep=0pt,topsep=0pt]
     \item universally learnable at an exponential rate: there exists an algorithm $\Aexp$ that attains zero excess risk with exponential probability on $\Theta$ (\cref{def:exponential-rate-probability}), that is, for all $P\in\cP$, there exist constants $c,C>0$ such that for all $n\in \NN$, $\PP_{\rS\sim P^{ n}}\prn{\excessRisk(\Aexp(\rS))>0}\leq C\exp\prn{-c n}$, and \label{item:exp-universal}
    \item uniformly learnable: the minimax rate (\cref{def:minimax-along-tail}), denoted $r_n(\delta):=  \fM(\Theta,\delta,n)$, satisfies $r_n(\delta)\to 0$ as $n\to \infty$ for every fixed $\delta$, and there exist an algorithm $\Amini$ and universal constants $C,c>0$ such that for all $n\in\NN$, $P\in\cP$, and $\delta\in(0,c)$, the algorithm satisfies $\PP_{\rS\sim P^{ n}}\prn{\excessRisk(\Amini(\rS))\leq C  r_n(\delta)}\geq 1-\delta$.
    \label{item:minimax}
\end{enumerate}
\end{definition}
Learnability of $\Theta$ in both worlds is the minimal requirement for investigating whether an algorithm can achieve both uniform and exponential guarantees whenever each is separately attainable.

Given that for finite dictionaries $Q$-aggregation achieves minimax optimality along the tail \emph{and} zero excess risk with exponential probability (\cref{thm:Q-exponential}), one may be tempted to conclude that if a space $\Theta$ is learnable in both worlds, the following algorithm will achieve both guarantees:
Split the data into two equally sized parts $\rS,\rS'$, compute $\fmini = \Amini(\rS')$ and $\fexp = \Aexp(\rS')$, and aggregate the two-element dictionary $\crl{\fmini,\fexp}$ using $Q$-aggregation with a uniform prior on the sample $\rS$.
A simple calculation shows that this procedure inherits the minimax optimality of $\Amini$. It does \emph{not}, however, inherit the exponential-rate guarantee because the dictionary depends on the first split of the sample, so the gap in excess risk between the models may shrink with $n$. This effectively turns the problem from a universal problem into a uniform problem, in which exponential rates may be impossible. That is the main challenge we address in this section. We now formally prove how the infinite setting introduces a trade-off between uniform and universal guarantees, and demonstrate that the algorithm just described, which we call \emph{best-of-both-worlds $Q$-aggregation} (QBOB), is still useful in that it optimally trades off the uniform and universal guarantees.

\subsection{An Impossibility Result}

As already alluded to, in general, there exist problem spaces $\Theta$ that are learnable in both worlds, but no algorithm can achieve both guarantees simultaneously. We formalize this in the following impossibility result.
The proof of \cref{thm:bob-general} is in \cref{proof:bob-general}.

\begin{theorem}\label{thm:bob-general}
Fix $\cF\subseteq \cM$.
Assume there exist constants $\gamma\in (0,1]$, $\eps\in(0,1/2]$, distinct points $x_1,x_2,\ldots,z_1,z_2,\ldots\in\cX$,
values $y_1,y_2,\ldots\in[\eps,1-\eps]$, and functions $f_1,f_2,\ldots\in\cF$
such that
\begin{align}
\forall i\in\NN, \forall j<i,& \qquad  f_i(x_j) = y_j, \label{eq:bob-tail-prefix}\\
\forall i\in\NN,& \qquad |f_i(x_i)-y_i| \ge \gamma. \label{eq:bob-tail-gap}
\end{align}
Then there exists a family of distributions $\cP$ on $\cX\times[0,1]$ such that $\Theta=\crl{(P,\cF):P\in\cP}$ is learnable in both worlds (\cref{def:learnable-in-both-worlds}) with uniform rate $r_n(\delta)\asymp \min\{1,\log(1/\delta)/n\}$, but not at faster rates.\footnote{Here, we mean that zero excess risk with exponential probability is the best one can achieve and that the rate $r_n(\delta)$ is optimal up to constant factors: the lower bound is of order $ \log{(1/(2\delta))}/n $ and requires that $ \delta < 1/2 $ and $ n\geq \log{(1/(2\delta))} $.} Furthermore, writing
\(
c_\eps:=2+8\log\prn{1/\eps},
\)
for every function $\varphi:\NN\to[4,\infty)$ and every learning algorithm $\cA$, there exists a sequence $\crl{n_k}_{k=1}^\infty$ with $n_k\to\infty$ for which at least one of the following holds:
\begin{enumerate}
    \item There exists a distribution $P\in\cP$ such that for all $k\in\NN$,
    \[
        \PP_{\rS\sim P^{n_k}}\prn{\excessRiskPar{P}{\cF}(\cA(\rS))\ge \frac{\gamma^{2}}{16\varphi(n_k)}}
        \ge
        \frac{1}{4}\exp\prn{-c_\eps\frac{n_k}{\varphi(n_k)}}.
    \]
    \item There exists a sequence $\crl{P_k}_{k=1}^\infty\subseteq \cP$ such that for all $k\in\NN$,
    \[
        \PP_{\rS\sim P_k^{n_k}}\prn{\excessRiskPar{P_k}{\cF}(\cA(\rS))\ge \frac{\gamma^{2}}{16\varphi(n_k)}}
        \ge
        \frac{9}{20}.
    \]
\end{enumerate}
\end{theorem}

Before discussing the proof idea, let us make a remark about the assumptions of \cref{thm:bob-general}.

\begin{remark}
The condition on the function class in \cref{thm:bob-general} is exactly the same as in \cref{thm:nearly-exponential-rates-lower-bound}. The reason that $\Aexp$ can achieve zero excess risk with exponential probability in \cref{thm:bob-general} is that the class of distributions $\cP$ excludes those distributions in \cref{thm:nearly-exponential-rates-lower-bound} that are used to prove the \emph{nearly} exponential rate lower bound. Hence, there is no contradiction between \cref{thm:nearly-exponential-rates-lower-bound,thm:bob-general}.
\end{remark}

\begin{example}
\label{rem:BOB-lower-bound-instance}
An instance of the conditions of \cref{thm:bob-general} is when we take $\cX=\ZZ$ and, for all $i\in\mathbb{N}$
\begin{equation*}
    y_i=\frac{1}{2} \qquad \text{and} \qquad f_i(x)=
    \begin{cases}
        1, & x=i,\\
        \frac{1}{2}, & x\neq i,
    \end{cases}
\end{equation*}
and choose $x_i=i$, $z_i=-i$.
We then obtain that $f_i(x_j)=\frac{1}{2}=y_j$ for $j<i$, and $\abs{f_i(x_i)-y_i}=\frac{1}{2}$.
Hence the assumptions are nonvacuous with the constants $\eps = 1/2$ and $\gamma = 1/2$.
\end{example}

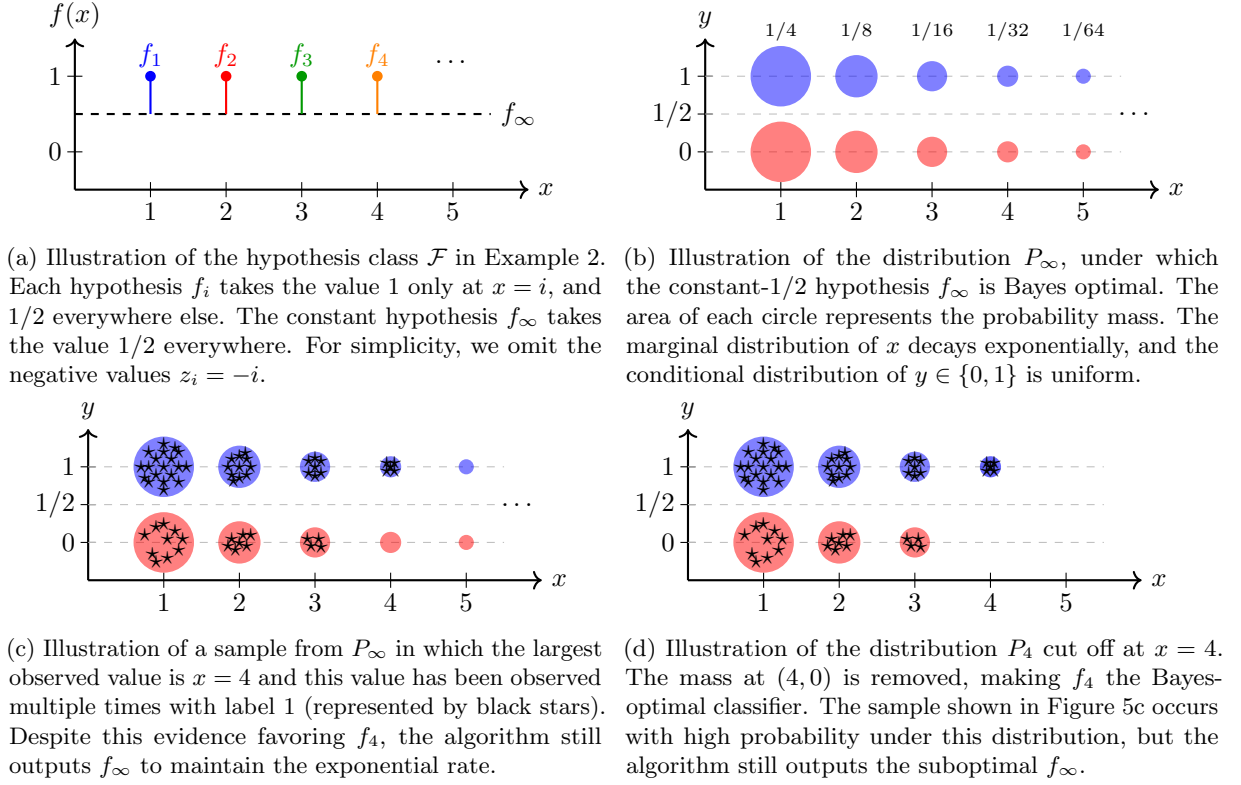
\begin{figure}[h]
    \centering
    \begin{subfigure}[t]{0.49\textwidth}
        \centering
        \begin{tikzpicture}[scale=1]

        \draw[->, thick] (0,-0.5) -- (6,-0.5) node[right] {$x$};
        \draw[->, thick] (0,-0.5) -- (0,1.5) node[above] {$f(x)$};

        \draw (-0.1,0) -- (0.1,0) node[left=0.15cm] {$0$};

        \draw (-0.1,1) -- (0.1,1) node[left=0.15cm] {$1$};

        \foreach \x in {1,2,3,4,5} {
            \draw (\x,-0.6) -- (\x,-0.4) node[below=0.15cm] {$\x$};
        }

        \draw[dashed, thick, black] (0,0.5) -- (5.5,0.5) node[right, text=black] {$f_\infty$};

        \draw[blue, thick] (1,0.5) -- (1,1);
        \fill[blue] (1,1) circle (2pt) node[above] {$f_1$};

        \draw[red, thick] (2,0.5) -- (2,1);
        \fill[red] (2,1) circle (2pt) node[above] {$f_2$};

        \draw[green!60!black, thick] (3,0.5) -- (3,1);
        \fill[green!60!black] (3,1) circle (2pt) node[above] {$f_3$};

        \draw[orange, thick] (4,0.5) -- (4,1);
        \fill[orange] (4,1) circle (2pt) node[above] {$f_4$};

        \node at (5, 1.2) {$\dots$};

    \end{tikzpicture}
        \caption{Illustration of the hypothesis class $\cF$ in \cref{rem:BOB-lower-bound-instance}. Each hypothesis $f_i$ takes the value $1$ only at $x=i$, and $1/2$ everywhere else. The constant hypothesis $f_\infty$ takes the value $1/2$ everywhere. For simplicity, we omit the negative values $z_i=-i$.}
        \label{fig:hypothesis_class}
    \end{subfigure}
    \hfill
    \begin{subfigure}[t]{0.49\textwidth}
        \centering
        \begin{tikzpicture}[scale=1]

        \draw[->, thick] (0,-0.5) -- (6,-0.5) node[right] {$x$};
        \draw[->, thick] (0,-0.5) -- (0,1.5) node[above] {$y$};

        \draw (-0.1,0) -- (0.1,0) node[left=0.15cm] {$0$};
        \draw (-0.1,0.5) -- (0.1,0.5) node[left=0.15cm] {$1/2$};
        \draw (-0.1,1) -- (0.1,1) node[left=0.15cm] {$1$};

        \draw[dashed, gray!50] (0,0) -- (5.5,0);
        \draw[dashed, gray!50] (0,0.5) -- (5.5,0.5);
        \draw[dashed, gray!50] (0,1) -- (5.5,1);

        \foreach \x in {1,2,3,4,5} {
            \draw (\x,-0.6) -- (\x,-0.4) node[below=0.15cm] {$\x$};
        }

        \foreach \x/\p/\r in {1/4/0.4, 2/8/0.28, 3/16/0.2, 4/32/0.14, 5/64/0.1} {

            \fill[blue, opacity=0.5] (\x, 1) circle (\r);
            \node[above=0.35cm] at (\x, 1) {\scriptsize $1/\p$};
            \fill[red, opacity=0.5] (\x, 0) circle (\r);
        }

        \node at (5.7, 0.5) {$\dots$};

    \end{tikzpicture}
        \caption{Illustration of the distribution $P_\infty$, under which the constant-$1/2$ hypothesis $f_\infty$ is Bayes optimal. The area of each circle represents the probability mass. The marginal distribution of $x$ decays exponentially, and the conditional distribution of $y \in \{0, 1\}$ is uniform.}
        \label{fig:distribution_uniform}
    \end{subfigure}
    \begin{subfigure}[t]{0.49\textwidth}
        \centering
        \begin{tikzpicture}[scale=1]

        \draw[->, thick] (0,-0.5) -- (6,-0.5) node[right] {$x$};
        \draw[->, thick] (0,-0.5) -- (0,1.5) node[above] {$y$};

        \draw (-0.1,0) -- (0.1,0) node[left=0.15cm] {$0$};
        \draw (-0.1,0.5) -- (0.1,0.5) node[left=0.15cm] {$1/2$};
        \draw (-0.1,1) -- (0.1,1) node[left=0.15cm] {$1$};

        \draw[dashed, gray!50] (0,0) -- (5.5,0);
        \draw[dashed, gray!50] (0,0.5) -- (5.5,0.5);
        \draw[dashed, gray!50] (0,1) -- (5.5,1);

        \foreach \x in {1,2,3,4,5} {
            \draw (\x,-0.6) -- (\x,-0.4) node[below=0.15cm] {$\x$};
        }

        \foreach \x/\p/\r in {1/4/0.4, 2/8/0.28, 3/16/0.2, 4/32/0.14, 5/64/0.1} {

            \fill[blue, opacity=0.5] (\x, 1) circle (\r);

            \fill[red, opacity=0.5] (\x, 0) circle (\r);
        }

        \foreach \dx/\dy in {0/0, 0.05/0.05, -0.05/-0.05, -0.05/0.05, 0.05/-0.05} {
            \node at (4 + \dx, 1 + \dy) {$\star$};
        }

        \foreach \dx/\dy in {0/0, 0.08/0.08, -0.08/-0.08, -0.08/0.08, 0.08/-0.08, 0/0.12, 0/-0.12} {
            \node at (3 + \dx, 1 + \dy) {$\star$};
        }
        \foreach \dx/\dy in {0.05/0.05, -0.05/-0.05, -0.1/0.05, 0.08/-0.06} {
            \node at (3 + \dx, 0 + \dy) {$\star$};
        }

        \foreach \dx/\dy in {0/0, 0.1/0.1, -0.1/-0.1, 0.1/-0.1, -0.1/0.1, 0.15/0, -0.15/0, 0/0.15, 0/-0.15, 0.08/0.18, -0.08/-0.18} {
            \node at (2 + \dx, 1 + \dy) {$\star$};
        }
        \foreach \dx/\dy in {0/0, 0.1/-0.05, -0.1/0.05, -0.05/-0.1, 0.05/0.1, 0.15/0.1, -0.15/-0.05} {
            \node at (2 + \dx, 0 + \dy) {$\star$};
        }

        \foreach \dx/\dy in {0/0, 0.1/0.1, -0.1/-0.1, 0.1/-0.1, -0.1/0.1, 0.2/0, -0.2/0, 0/0.2, 0/-0.2, 0.2/0.2, -0.2/-0.2, 0.2/-0.2, -0.2/0.2, 0.3/0, -0.3/0, 0/0.3, 0/-0.3, 0.15/0.25, -0.25/-0.15} {
            \node at (1 + \dx, 1 + \dy) {$\star$};
        }
        \foreach \dx/\dy in {0.05/0.05, 0.15/0.15, -0.15/-0.15, 0.2/-0.1, -0.1/0.2, 0.05/-0.2, -0.25/0.1, 0.25/0.05, -0.1/-0.25, 0.0/0.25} {
            \node at (1 + \dx, 0 + \dy) {$\star$};
        }

        \node at (5.7, 0.5) {$\dots$};

    \end{tikzpicture}
        \caption{Illustration of a sample from $P_\infty$ in which the largest observed value is $x=4$ and this value has been observed multiple times with label $1$ (represented by black stars). Despite this evidence favoring $f_4$, the algorithm still outputs $f_\infty$ to maintain the exponential rate.}
    \label{fig:scenario_empirical}
    \end{subfigure}
    \hfill
    \begin{subfigure}[t]{0.49\textwidth}
         \begin{tikzpicture}[scale=1]

        \draw[->, thick] (0,-0.5) -- (6,-0.5) node[right] {$x$};
        \draw[->, thick] (0,-0.5) -- (0,1.5) node[above] {$y$};

        \draw (-0.1,0) -- (0.1,0) node[left=0.15cm] {$0$};
        \draw (-0.1,0.5) -- (0.1,0.5) node[left=0.15cm] {$1/2$};
        \draw (-0.1,1) -- (0.1,1) node[left=0.15cm] {$1$};

        \draw[dashed, gray!50] (0,0) -- (5.5,0);
        \draw[dashed, gray!50] (0,0.5) -- (5.5,0.5);
        \draw[dashed, gray!50] (0,1) -- (5.5,1);

        \foreach \x in {1,2,3,4,5} {
            \draw (\x,-0.6) -- (\x,-0.4) node[below=0.15cm] {$\x$};
        }

        \foreach \x/\p/\r in {1/4/0.4, 2/8/0.28, 3/16/0.2} {

            \fill[blue, opacity=0.5] (\x, 1) circle (\r);

            \fill[red, opacity=0.5] (\x, 0) circle (\r);
        }

        \fill[blue, opacity=0.6] (4, 1) circle (0.14);

        \foreach \dx/\dy in {0/0, 0.05/0.05, -0.05/-0.05, -0.05/0.05, 0.05/-0.05} {
            \node at (4 + \dx, 1 + \dy) {$\star$};
        }

        \foreach \dx/\dy in {0/0, 0.08/0.08, -0.08/-0.08, -0.08/0.08, 0.08/-0.08, 0/0.12, 0/-0.12} {
            \node at (3 + \dx, 1 + \dy) {$\star$};
        }
        \foreach \dx/\dy in {0.05/0.05, -0.05/-0.05, -0.1/0.05, 0.08/-0.06} {
            \node at (3 + \dx, 0 + \dy) {$\star$};
        }

        \foreach \dx/\dy in {0/0, 0.1/0.1, -0.1/-0.1, 0.1/-0.1, -0.1/0.1, 0.15/0, -0.15/0, 0/0.15, 0/-0.15, 0.08/0.18, -0.08/-0.18} {
            \node at (2 + \dx, 1 + \dy) {$\star$};
        }
        \foreach \dx/\dy in {0/0, 0.1/-0.05, -0.1/0.05, -0.05/-0.1, 0.05/0.1, 0.15/0.1, -0.15/-0.05} {
            \node at (2 + \dx, 0 + \dy) {$\star$};
        }

        \foreach \dx/\dy in {0/0, 0.1/0.1, -0.1/-0.1, 0.1/-0.1, -0.1/0.1, 0.2/0, -0.2/0, 0/0.2, 0/-0.2, 0.2/0.2, -0.2/-0.2, 0.2/-0.2, -0.2/0.2, 0.3/0, -0.3/0, 0/0.3, 0/-0.3, 0.15/0.25, -0.25/-0.15} {
            \node at (1 + \dx, 1 + \dy) {$\star$};
        }
        \foreach \dx/\dy in {0.05/0.05, 0.15/0.15, -0.15/-0.15, 0.2/-0.1, -0.1/0.2, 0.05/-0.2, -0.25/0.1, 0.25/0.05, -0.1/-0.25, 0.0/0.25} {
            \node at (1 + \dx, 0 + \dy) {$\star$};
        }

    \end{tikzpicture}
        \caption{Illustration of the distribution $P_4$ cut off at $x=4$. The mass at $(4,0)$ is removed, making $f_4$ the Bayes-optimal classifier. The sample shown in \cref{fig:scenario_empirical} occurs with high probability under this distribution, but the algorithm still outputs the suboptimal $f_\infty$.}
    \label{fig:distribution_cutoff}
    \end{subfigure}
    \caption{Illustration of the construction in the best-of-both-worlds impossibility result \cref{thm:bob-general}.}
\end{figure}

\paragraph{Proof Idea.}
We explain the proof using a simplified version of the instance described in \cref{rem:BOB-lower-bound-instance}.
We construct a sequence of hypotheses $ f_{1}, f_{2}, \ldots $, each taking the value $ 1/2 $ everywhere except at the point $ i $, where $f_i(i) = 1 $. We also include the constant function $f_\infty(x) = 1/2$ for all $x$. This structure is illustrated in \cref{fig:hypothesis_class}.
We construct a family of distributions.
First, we consider a distribution $P_\infty$ under which the constant-$ 1/2 $ hypothesis $ f_\infty $ is Bayes optimal. Under this distribution, the conditional distribution of $y$ given $x \in \NN$ is uniform over $ \{0,1\} $, and the marginal distribution of $ x $ decays exponentially. This base distribution is illustrated in \cref{fig:distribution_uniform}.
We use this to show that any algorithm achieving an exponential learning rate must output $ f_\infty $ with high probability, even when the sample provides evidence suggesting otherwise. For example, if $ x $ is the largest point observed so far and every observation at $x$ has the label $ 1 $, as visualized in \cref{fig:scenario_empirical}, this sample might appear more consistent with $ f_{x} $ being optimal. Yet to maintain an exponential rate on the base distribution, the algorithm must ignore this evidence and still output $ f_\infty $. Since this holds for infinitely many such points $ x $, we can selectively ``cut off'' the base distribution at $ x $ and remove the mass on $ (x,0) $ to construct a new distribution for which $f_x$ is actually optimal. With high probability, a sample from this truncated distribution will be identical to the deceptive sample from the base distribution (\cref{fig:distribution_cutoff}). Because the algorithm still outputs the suboptimal $ f_\infty $ on these samples, it incurs a large enough error and fails to achieve the minimax rate.
The proof handles details omitted from this simplified account and uses the additional points $z_i=-i$ to achieve zero excess risk with exponential probability for the distributions other than $P_\infty$.

\subsection{Best-of-Both-Worlds $Q$-aggregation}
In this section, we show that, for any $\Theta$ that is learnable in both worlds (\cref{def:learnable-in-both-worlds}), the $Q$-aggregation algorithm achieves the best possible best-of-both-worlds trade-off permitted by \cref{thm:bob-general}. Its universal and uniform guarantees are parameterized by the prior and temperature (corresponding to the function $\varphi$ in \cref{thm:bob-general}). It does so by splitting the data and aggregating the outputs of the two algorithms $\Amini,\Aexp$ as explained in the beginning of this section. The algorithm is stated formally in \cref{alg:QBOB}.
Recall that $\Rhat_\rS$ from \eqref{eqn:empirical-risk-def} denotes the risk over the sample $\rS=(X_i,Y_i)_{i=1}^n$ and, in particular, not over a different sample $\rS'$. The proof of \cref{thm:QBOB} is in \cref{proof:QBOB}.

\begin{algorithm}
\caption{Best-of-Both-Worlds $Q$-aggregation (QBOB)}\label{alg:QBOB}
\begin{algorithmic}[1]
\REQUIRE samples $\rS$ and $\rS'$, prior distribution $\pi$, temperature $\beta > 0$, access to $\Amini$ and $\Aexp$.
\STATE Compute $\fmini = \Amini(\rS')$ and $\fexp = \Aexp(\rS')$.
\STATE Compute the minimizer $\rhohatQ$ over $[0,1]$ (break ties in favor of smaller $\rho$) of the empirical objective
\begin{align*}
    \Psihat_{Q}(\rho)
    =& \frac{1}{2} \Rhat_\rS \prn{(1-\rho) \fexp + \rho \fmini}
    + \frac{1}{2}\prn{(1-\rho)\Rhat_\rS(\fexp)
    + \rho \Rhat_\rS(\fmini)}\\
    &\qquad + \frac{\beta(1-\rho)}{n}\log\prn{\tfrac{1}{\piexp}}
    + \frac{\beta\rho}{n}\log\prn{\tfrac{1}{\pimini}}.
\end{align*}
\STATE Return $\Qbob= f_{\rhohatQ}=(1-\rhohatQ)\fexp+\rhohatQ \fmini$.
\end{algorithmic}
\end{algorithm}
\begin{theorem}\label{thm:QBOB}
Let $ c_1,c_2>0 $ be universal constants, let $\cP$ be a class of distributions over $\cX\times [0,1]$, and let $ \cF \subseteq \cM$ be a class of functions mapping $\cX$ to $[0,1]$. Assume that $\Theta=\crl{(P,\cF):P\in\cP}$ is learnable in both worlds with uniform rate $r_n(\delta)$ (\cref{def:learnable-in-both-worlds}).
Let $0< \pimini \leq \piexp < 1$ with $\pimini+\piexp=1$, possibly dependent on $n$, and let $\rS,\rS'\sim P^{n}$ be independent.
Then, given $\rS,\rS'$, access to $\Amini$ and $\Aexp$, the prior $\pi=(\pimini,\piexp)$,
and $ \beta \geq c_2 $, the best-of-both-worlds $Q$-aggregation algorithm (see \cref{alg:QBOB}) outputs $\Qbob$ with the following guarantees:
\begin{enumerate}
    \item For every distribution $P \in \cP$, there exist $C,c>0$ such that for all $n\in\NN$,
    \begin{align*}
    \PP_{(\rS',\rS)\sim P^{2n}}\prn{
    \excessRisk(\Qbob)>0
    }
    &\leq \exp\left( - c_1\beta\log\prn{\frac{\piexp}{\pimini}}  \right)+C\exp\prn{-c n} .
    \end{align*}
    \item For all $n\in\NN$, $P\in \cP$, and $\delta\in(0,c)$, where $c$ is the constant from \cref{def:learnable-in-both-worlds},
    \begin{align*}
    \PP_{(\rS',\rS)\sim P^{2n}}\prn{
    \excessRisk(\Qbob)>
    r_n(\delta/2) + \frac{\beta\log\prn{1/\pimini}}{n} + \frac{2\beta\log\prn{2/\delta}}{n}}\leq \delta.
    \end{align*}
\end{enumerate}
\end{theorem}
\Cref{thm:QBOB} provides a parameterized trade-off between the universal and uniform rates through the prior $\pi=(\pimini,\piexp)$ and temperature $\beta$, both of which may depend on $n$. For example, consider a fixed temperature and the two extreme cases of choosing the prior as a function of $n$: either $\piexp/\pimini =\Theta(e^n)$ or $\piexp=\pimini$. In the first case, \cref{alg:QBOB} achieves zero excess risk with exponential probability, but the uniform guarantee becomes vacuous. In the second case, the uniform bound is minimax optimal (as $r_n=\Omega(1/n)$), but the universal bound has constant failure probability. In the next section, we discuss this trade-off in more detail and explain how it demonstrates the tightness of \cref{thm:bob-general}.

\begin{remark}\Cref{thm:QBOB} can also be extended to the setting in which $\Aexp$ achieves zero excess risk at a rate governed by a rate function $R$. In this case, the guarantee in part~1 of the theorem would contain the term $C R(c n)$ in place of $C\exp(-c n)$.
\end{remark}

\subsection{Tightness of \texorpdfstring{\cref{thm:bob-general,thm:QBOB}}{Theorems \ref{thm:bob-general} and \ref{thm:QBOB}}}
We now explain how \cref{thm:bob-general,thm:QBOB} together establish tightness and characterize the exact trade-off between universal exponential rates and minimax rates.

Fix $\delta\in(0,9/20)$. Let $\cF$ satisfy the assumptions of
\cref{thm:bob-general}, and let $\cP$, $\Amini$, and $\Aexp$ be, respectively, the family of distributions, the minimax algorithm, and the universal-rate algorithm provided
by \cref{thm:bob-general}. \Cref{thm:QBOB} is stated for a sample size of $2n$. To account for this, we set $\tilde{n}=\left\lfloor n/2 \right \rfloor$. When the learner is given $n$ samples, it splits them into two samples of size $\tilde{n}$ and runs $\Qbob$ on the resulting $2\tilde{n}$ samples.

We now choose a target uniform rate and show that the universal rate achieved by \cref{alg:QBOB} matches the lower bound.
Let $\varphitilde(n)\equiv\varphitilde(n,\delta)$ be such that $1/\varphitilde(n)$ is a target uniform rate that is strictly decreasing
to zero and satisfies
\[
    \frac{1}{\varphitilde(n)}
    \geq c_2\frac{\log(2/\delta)}{\tilde{n} }
    \qquad\text{for all }n\in\NN,
\]
where $c_2$ is the universal lower bound on $\beta$ in \cref{thm:QBOB},
enlarged if necessary so that $c_2\geq1$.
Choose the following parameters in \cref{thm:QBOB}:
$\pimini=\delta/2$, $\piexp=1-\delta/2$, $\beta_n=\tilde{n}/(\varphitilde(n)\log(1/\pimini))=\tilde{n}/(\varphitilde(n)\log(2/\delta))$,
and choose the rate function $\varphi(n)=\max\crl{4,\gamma^2\varphitilde(n)/(16K)}$ in \cref{thm:bob-general}
for $K>27$. This yields the following upper and lower bounds; their derivations follow the discussion.

The universal and uniform \emph{upper} bounds from \cref{thm:QBOB} are given by
\begin{equation*}
    \PP_{\rS\sim P^n}\prn{\excessRiskPar{P}{\cF}(\Qbob)>0}\leq \tilde{C} \exp\prn{-\tilde{c} \frac{n}{\varphitilde(n)}} \quad \text{and} \quad \PP_{\rS\sim P^n}\prn{ \excessRiskPar{P}{\cF}(\Qbob)>\frac{27}{\varphitilde(n)} } \leq\delta,
\end{equation*}
where $\tilde{C},\tilde{c}>0$ are constants depending only on $P$ and $\delta$.
The universal and uniform \emph{lower} bounds from \cref{thm:bob-general} are given by
\begin{equation*}
    \PP_{\rS\sim P^{n_k}}\prn{\excessRiskPar{P}{\cF}(\cA(\rS))>0}\geq \check{C}\exp\prn{-\check{c}\frac{n_k}{\varphitilde(n_k)}} \quad \text{or} \quad
    \PP_{\rS\sim P_k^{n_k}}\prn{\excessRiskPar{P_k}{\cF}(\cA(\rS))\geq \frac{K}{\varphitilde(n_k)}}\geq\frac{9}{20},
\end{equation*}
where $\check{C},\check{c}>0$ are constants depending only on $\eps$ and $\gamma$, and $\crl{n_k}_{k=1}^\infty$ is a strictly increasing sequence.
Comparing the upper and lower bounds, we see that for $\Qbob$, the second alternative in the lower bound is incompatible with the uniform upper bound. Because $K>27$, its event is contained in the event that the excess risk is
greater than $27/\varphitilde(n_k)$, whose probability is at most
$\delta<9/20$. Hence, the first alternative (the lower bound on the universal rate) must hold and matches the upper bound for $\Qbob$. For any
prescribed uniform rate $ 1/\varphitilde(n)$, $\Qbob$ attains the optimal universal
exponent $n /\varphitilde(n)$, up to constants and along the subsequence in the
lower bound. In this sense, $\Qbob$ is \emph{Pareto-optimal}: improving the universal rate would come at the cost of the uniform rate. The resulting trade-off
is illustrated in \cref{fig:trade-off}.

To verify these bounds, first note that $\beta_n\geq c_2$, where $c_2$ is chosen to be at least as large as the constant from \cref{thm:QBOB}, so \cref{thm:QBOB} applies. By
\cref{thm:QBOB}, with probability at least $1-\delta$, the excess risk of $\Qbob$ is at most
\(
    24\log(2/\delta)/\tilde{n}+3/\varphitilde(n)
    \leq27/\varphitilde(n),
\)
where the factor $24$ comes from the minimax guarantee of $ \Amini $ in \cref{thm:bob-general}.
This gives, uniformly over $P\in\cP$,
\[
    \PP_{\rS\sim P^n}\prn{
        \excessRiskPar{P}{\cF}(\Qbob)>\frac{27}{\varphitilde(n)}
    }
    \leq\delta,
\]
as claimed.
The first guarantee in \cref{thm:QBOB} yields constants
$\tilde{C},\tilde{c}>0$ depending on $P$ and $\delta$ such that
\[
    \PP_{\rS\sim P^n}\prn{
        \excessRiskPar{P}{\cF}(\Qbob)>0
    }
    \leq
    \tilde{C}\exp\prn{-\tilde{c}\frac{\tilde{n}}{\varphitilde(n)}}
\]
which gives the claimed universal upper bound.
For the lower bounds, invoke \cref{thm:bob-general} with the function
$
    \varphi(n)
    =\max\crl{4,\gamma^2\varphitilde(n)/(16K)}.
$
Here $\varphi$ may be extended arbitrarily from the positive integers to
$\RR_{>0}$ while retaining the lower bound $\varphi\geq4$. Since
$1/\varphitilde(n)\to0$, after discarding finitely many terms from the resulting
sequence $\crl{n_k}_{k=1}^\infty$, we have
$\gamma^2/(16\varphi(n_k))=K/\varphitilde(n_k)$ and $n_k/\varphi(n_k)=16K n_k/(\varphitilde(n_k)\gamma^2).$
Applying \cref{thm:bob-general}, for every learning algorithm $\cA$, at least one of the
following alternatives holds. Either there exists $P\in\cP$ such that, for
every $k$,
\[
    \PP_{\rS\sim P^{n_k}}\prn{
        \excessRiskPar{P}{\cF}(\cA(\rS))>0
    }
    \geq
    \frac{1}{4}\exp\prn{
        -\frac{16Kc_\eps}{\gamma^2}
        \frac{n_k}{\varphitilde(n_k)}
    },
\]
or there exists a sequence $\crl{P_k}_{k=1}^\infty\subseteq\cP$ such that,
for every $k$,
\[
    \PP_{\rS\sim P_k^{n_k}}\prn{
        \excessRiskPar{P_k}{\cF}(\cA(\rS))
        \geq \frac{K}{\varphitilde(n_k)}
    }
    \geq\frac{9}{20}
\]
as claimed. This completes the discussion.

\section{Conclusion}

We study the compatibility of two learning paradigms for regression with squared loss: the standard \emph{uniform} (or PAC/minimax) viewpoint, where for each sample size the worst-case distribution may be chosen by nature, and the more recently introduced \emph{universal} viewpoint, where the distribution is fixed a priori and the sample-size dependence may become exponential.
In particular, we study the compatibility from an algorithmic perspective and ask whether achieving exponential rates in the universal setting must come at the cost of the stronger uniform guarantees.
By revisiting the statistical model selection aggregation setting, we answer this question in the negative when the hypothesis class is finite and show that $Q$-aggregation is optimal under both notions.
We also prove that many other common aggregation strategies, such as pruning with a threshold, online-to-batch conversions via averaging, and star estimation, fail to be optimal for at least one of the universal and uniform notions.
For (countably) infinite hypothesis classes, however, the conclusion changes: the hypothesis space and family of distributions may be such that exponential rates and uniform learning are both possible, but no algorithm can achieve both.
The resulting trade-off can always be achieved by using $Q$-aggregation to combine the minimax-optimal and exponential-rate algorithms (which are assumed to exist).
This characterizes the exact trade-off between universal exponential rates and minimax guarantees.

Several questions remain open and suggest directions for future work.
For one, we do not characterize which hypothesis spaces are learnable at which universal rates for squared loss in the agnostic setting (beyond \cref{sec:countably-infinite-hyp-spaces}). Much of the existing literature on universal learning revolves around analogous questions in different settings, so extending such characterizations to agnostic regression with squared loss is a natural direction.
We also do not provide an \emph{exact} characterization of which pairs consisting of a hypothesis space and a family of distributions exhibit a trade-off between universal exponential rates and uniform rates. \cref{thm:bob-general} already provides a class of such instances, but whether there are examples outside this class remains open. Our focus on exponential rates is motivated by the finite case, in which exponential rates are always possible but the problem is already nontrivial. It remains open to investigate the cost of general universal rates for uniform guarantees. And finally, similar investigations into the trade-off between universal and uniform rates beyond regression with squared loss, for example for classification and other loss functions, could generalize this work.

One high-level motivation is to adopt an algorithmic perspective on different notions of optimality: when must an algorithm designer choose between them\textemdash and when not? Here \emph{admissibility} in the decision-theoretic sense may be an alternative perspective to take on the aggregation problem.

\stopcontents[main]

\section*{Acknowledgements}
Tobias Wegel was supported by SNSF Grant 204439. Mikael M{\o}ller H{\o}gsgaard was supported by a Carlsberg Internationalisation Fellowship.
Patrick Rebeschini was funded by UK Research and Innovation (UKRI) under the UK government’s Horizon Europe funding guarantee [grant number EP/Y028333/1].
\paragraph{LLM usage.} The authors acknowledge the use of LLMs for improving the exposition and exploring some of the proof ideas. The authors take full responsibility for the contents and correctness of this work.

\newpage
\bibliographystyle{plain}
{
\small
\bibliography{bibliography}
}

\newpage
\appendix
\crefalias{section}{appendix}
\crefalias{subsection}{appendix}
\crefalias{subsubsection}{appendix}
\crefname{appendix}{Appendix}{Appendices}
\Crefname{appendix}{Appendix}{Appendices}

\startcontents[appendix]
\printcontents[appendix]{}{1}{\section*{Appendix Contents}}

\newpage

\newcolumntype{L}{>{\raggedright\arraybackslash}X}
\newcolumntype{C}{>{\centering\arraybackslash}m{2.8cm}}
\begin{table}[H]
    \centering
    \caption{Table of notation.}
    \label{tab:notation}
    \renewcommand{\arraystretch}{1.15}
    \setlength{\tabcolsep}{8pt}
    \begin{tabularx}{0.95\linewidth}{C L}
        \toprule
        \textbf{Symbol} & \textbf{Meaning} \\
        \midrule
        $\cX$ & Abstract covariate space \\
        $P$     & Distribution on $\cX\times [0,1]$                     \\
        $\cF$     & Hypothesis class of functions $\cX\to [0,1]$                                 \\
        $\cR_P(f)$ & Risk of $f$ with respect to $P$, $\EE_{(X,Y)\sim P}(f(X)-Y)^2$  \\
        $\excessRisk$       & Excess risk $\cR_P(f)-\inf_{f'\in\cF} \cR_P(f')$                              \\
        $\rS$ & Sample $(X_i,Y_i)_{i=1}^n$ of size $n$ i.i.d.\ from $P$  \\
        $\Rhat_\rS$ & Empirical risk on the sample $\rS$, $\frac{1}{n}\sum_{i=1}^n (f(X_i)-Y_i)^2$  \\
        $\cFstar$ & Set of optimal models in $\cF$ w.r.t.\ distribution $P$ \\
        $\fstar$ & Some arbitrary function in $\cFstar$ \\
        $\Deltamin$ & Smallest positive excess risk, $\min_{f\in\cF\setminus\cFstar}\excessRisk(f)$ \\
        $\threshold$ & Threshold of pruning-based estimators (\cref{def:pruning-based})  \\
        $\nrm{\cdot}_\rS$ & Empirical $L^2(\rS)$-norm on sample $\rS$  \\
        $\rhohat$ & Weights or distribution of an estimator on $[M]$ \\
        $f_\rho$ & Aggregate with distribution $\rho$, $f_\rho(\cdot)=\EE_{f\sim \rho} f(\cdot)$ \\
        $\triangle_M$ & Simplex in $\RR^M$: $\{\rho\in [0,1]^{M}:\sum_{k=1}^K \rho_k = 1\}$  \\
        $\blacktriangle_{M-1}$ & Simplex in $\RR^{M-1}$: $\{\rho\in [0,1]^{M-1}:\sum_{k=1}^K \rho_k \leq 1\}$ \\
        $\Esep(\rS)$ & Event where empirical risk separates optimal and suboptimal models $\{\forall f\in\cF\setminus\cFstar: \ \Rhat_{\rS}(f)-\Rhat_{\rS}(\fstar) > \frac{1}{2}\excessRisk(f)\}$ (\cref{lem:Hoeffding-separation}) \\
        $\Amini,\Aexp$ & Algorithms achieving the minimax and exponential rate (\cref{def:learnable-in-both-worlds}) \\
        $\delta_x$ & Dirac delta distribution on $x$ \\
        \bottomrule
    \end{tabularx}
\end{table}

\section{Proofs of the Preliminary Lemmata}
\label{sec:proofs-preliminaries}

\begin{proof}[Proof of \cref{lem:minimax-along-tail-implies-expectation}]
    Let $\fhat=\cA(\cF,\rS)$ and let the expectations and probabilities be with respect to $\rS$. By the assumption of minimax optimality along the tail, there exist constants $C,c>0$ such that, for all $M$, $n$, and $(P,\cF)\in\Theta_M$,
    \begin{equation*}
        \forall\delta\in(0,c):\qquad \PP\prn{\excessRisk(\fhat)>C\frac{\log(M/\delta)}{n}} \leq \delta.
    \end{equation*}
    Because this bound holds simultaneously for all $\delta\in(0,c)$, integrating the tail bound with $\delta=M\exp(-xn/C)\leq c/2$ for $x\geq C\log(2M/c)/n$ and using the layer-cake representation gives
    \begin{align*}
    \EE\brk{\excessRisk(\fhat)} &\leq \int_0^1 \PP\prn{\excessRisk(\fhat) > x} \ \d x \\
    &\leq C\frac{\log(2M/c)}{n} + \int_{C\log(2M/c)/n}^1 \PP\prn{\excessRisk(\fhat) > C\frac{\log\prn{\frac{M}{M\exp(-xn/C)}}}{n}}  \ \d x \\
    &\leq C\frac{\log(2M/c)}{n}+ \int_{C\log(2M/c)/n}^1 M\exp\prn{-\frac{xn}{C}} \ \d x \\
    &\leq C\frac{\log(2M/c)}{n} + \frac{MC}{n}\exp\prn{-\frac{(C\log(2M/c)/n)n}{C}} \\
    &\leq C\frac{\log(2M/c)}{n} + \frac{cC}{n} \leq  C'\frac{\log(M)}{n}.
    \end{align*}
    This proves minimax optimality in expectation with the constant $C'=2C(1+c+\abs{\log(c)})$.
\end{proof}

\begin{proof}[Proof of \cref{lem:exponential-implications}]
    Let $\fhat=\cA(\cF,\rS)$ and let the expectations and probabilities be with respect to $\rS$.
    The first implication is immediate. The second implication follows because, for any $\delta>0$, 
    \begin{equation*}
        \EE\brk{\excessRisk(\fhat)} \leq \int_0^\infty \PP(\excessRisk(\fhat)>\eps)\d\eps \leq \delta + \int_\delta^1 \PP(\excessRisk(\fhat)>\eps)\d\eps\leq \delta+ \PP(\excessRisk(\fhat)>\delta),
    \end{equation*}
    where the first inequality follows by applying the layer-cake representation to the positive part of the excess risk, and the second follows from $\excessRisk(\fhat)\leq 1$.
    Choosing $\delta = Ce^{-cn}$ and invoking the assumption that $\fhat$ achieves an exponential rate with exponential probability yields the bound $2Ce^{-cn}$.
\end{proof}

\begin{proof}[Proof of \cref{lem:exponential-is-optimal}]
    Let $y=f_1(x)=f_2(x)$, $y_i =f_i(x')$, and $\Delta = \abs{y_1-y_2}$. For $i\in \crl{1,2}$, let $P_i$ be a distribution such that $P_i(\crl{(x,y)})=1/2$ and $P_i(\crl{(x',y_i)})=1/2$. Then $f_i$ has zero risk under $P_i$.
    Let $E_n=\crl{X_1=\dots = X_n=x}$, which has probability $2^{-n}$ under both $P_1$ and $P_2$. On $E_n$, $a_n=\cA(\rS)(x')$ is some deterministic value, regardless of whether $\rS$ was drawn from $P_1$ or $P_2$. For every $n$, with probability at least $2^{-n}$ under at least one of $P_1$ and $P_2$, the algorithm has strictly positive risk. Thus either $P_1$ or $P_2$ satisfy this property for infinitely many $n$, proving the first claim with $C=1$ and $c=\log2$. The claim in expectation follows because for all $n\in\NN$,
    \begin{align*}
        \EE_{I\sim \uniform{1,2}} \EE_{\rS\sim P_I^n}\brk{\excessRiskPar{P_I}{\cF}(\cA(\rS))} &\geq \frac{1}{2}\PP(E_n) \prn{ \EE_{\rS\sim P_1^n}\brk{\excessRiskPar{P_1}{\cF}(\cA(\rS))\mid E_n} +\EE_{\rS\sim P_2^n}\brk{\excessRiskPar{P_2}{\cF}(\cA(\rS)) \mid E_n}} \\
        &\geq 2^{-(n+2)} \prn{\prn{a_n-y_1}^2 +\prn{a_n-y_2}^2 } \\
        &\geq 2^{-(n+3)} \Delta^2.
    \end{align*}
    Therefore, at least one of $P_1$ and $P_2$ satisfies $\EE_{\rS\sim P_i^n}\brk{\excessRiskPar{P_i}{\cF}(\cA(\rS))} \geq (\Delta^2/8) e^{-n \log 2}$ for infinitely many $n \in\NN$. The result holds for $C=\Delta^2/8$ and $c= \log 2$.
\end{proof}

\section{Proofs for Pruning with a Threshold}
\label{sec:proofs-pruning}

\subsection{Pruned-convex ERM and Midpoint Estimators}\label{sec:pruned-erm-midpoint}

\begin{algorithm}\caption{Pruned-convex ERM \cite{lecue2009aggregation}}\label{alg:pruned-convex-ERM}
\begin{algorithmic}[1]
\STATE \textbf{Input:} Dataset $\rS$ of size $n$, finite dictionary $\cF$, confidence parameter $\delta\in(0,1)$.
\STATE Split $\rS$ data-independently into parts $\rS_1,\rS_2$ with
$\abs{\rS_1}=\abs{\rS_2}=\floor{n/2}$.
\STATE On the first half $\rS_1$, compute $\erm=\operatorname{ERM}(\rS_1,\cF)$ and the pruned set
    \begin{equation*}
        \cFhat(\rS_1) = \crl{f\in\cF: \Rhat_{\rS_1}(f)\leq \Rhat_{\rS_1}(\erm)+C_1 \max\crl{\alpha\nrm{\erm-f}_{\rS_1},\alpha^2}} \ \text{with} \ \alpha = \sqrt{\tfrac{\log (2M/\delta)}{\floor{n/2}}}.
    \end{equation*}
\STATE On the second half $\rS_2$, compute $\prunedERM \in \argmin_{f\in \conv(\cFhat(\rS_1))} \Rhat_{\rS_2}(f)$.
\STATE \textbf{Return:} $\prunedERM$.
\end{algorithmic}
\end{algorithm}

\begin{algorithm}\caption{Midpoint estimator \cite{puchkin2021exponential,kanade2024exponential}}\label{alg:midpoint}
\begin{algorithmic}[1]
\STATE \textbf{Input:} Dataset $\rS$ of size $n$, finite dictionary $\cF$, confidence parameter $\delta\in(0,1)$.
\STATE Let $a>0$ be some large universal constant and compute the set
\begin{align*}
    \cFhat(\rS)&=\crl{f\in\cF:\, \Rhat_{\rS}(f)\leq \Rhat_{\rS}(\erm)+ a d_{\delta}(\erm,f)} \\
    &\qquad \text{with} \quad d_\delta(f,g)= \sqrt{\frac{\nrm{f-g}_{\rS}^2\cdot \log(2M/\delta)}{n}}+\frac{\log(2M/\delta)}{n}.
\end{align*}
\STATE  Compute $\tilde f \in \argmin_{f\in\cFhat(\rS)} \Rhat_{\rS}\prn{\frac{\erm+f}{2}}$ and set $\midpoint=(\erm+\tilde f)/2$.
\STATE \textbf{Return:} $\midpoint$.
\end{algorithmic}
\end{algorithm}

\begin{lemma}
\label{lem:pruned-convex-ERM-is-pruning-based}
    For $n\geq 2$, \cref{alg:pruned-convex-ERM} satisfies \cref{def:pruning-based} with a threshold $\threshold$ such that, for all $M\in\NN$ and $\delta\in(0,1)$, there exists $m_0\in\NN$ such that for all $m\geq m_0$,
    \begin{equation*}
        \threshold\leq \frac{C_1\sqrt{\log(2M/\delta)}}{\sqrt{m}}.
    \end{equation*}
\end{lemma}
\begin{proof}
    Let $m=\floor{n/2}$. Since $n\geq 2$, we have $m\geq n/4$, so we may take $\alpha=1/4$. Define
    \begin{equation*}
        \threshold:=C_1 \max\crl{\alpha,\alpha^2}=C_1\max\crl{ \sqrt{\frac{\log(2M/\delta)}{m}}, \frac{\log(2M/\delta)}{m} }.
    \end{equation*}
    Since $\erm$ and $f$ take values in $[0,1]$, we have $\nrm{\erm-f}_{\rS_1}\leq 1$, so
    \begin{equation*}
        C_1 \max\crl{\alpha\nrm{\erm-f}_{\rS_1},\alpha^2}
        \leq
        \threshold .
    \end{equation*}
    This gives $\cFhat(\rS_1)
        \subseteq
        \{f\in\cF: \Rhat_{\rS_1}(f)\leq \Rhat_{\rS_1}(\erm)+\threshold\}$.
    By construction, $\prunedERM\in\conv(\cFhat(\rS_1))$, so \cref{alg:pruned-convex-ERM} satisfies, for each fixed $M$ and $\delta$, \cref{def:pruning-based} with a threshold satisfying $\limsup_{m\to\infty}\threshold=0$. Specifically, for large $m$ we have $$\max\crl{ \sqrt{\frac{\log(2M/\delta)}{m}}, \frac{\log(2M/\delta)}{m} } = \sqrt{\frac{\log(2M/\delta)}{m}},$$ and the lemma follows.
\end{proof}
\begin{lemma}
    \label{lem:midpoint-is-pruning-based}
    \cref{alg:midpoint} satisfies \cref{def:pruning-based} with a threshold $\threshold$ such that, for all $M\in\NN$ and $\delta\in(0,1)$, there exists $m_0\in\NN$ such that for all $m\geq m_0$,
    \begin{equation*}
        \threshold\leq \frac{2a\sqrt{\log(2M/\delta)}}{\sqrt{m}}.
    \end{equation*}
\end{lemma}
\begin{proof}
    Let $m=n$ and $\rS_1=\rS$, so we may take $\alpha=1$. Define
    \begin{equation*}
        \threshold
        :=
        a\prn{\sqrt{\frac{\log(2M/\delta)}{m}}+\frac{\log(2M/\delta)}{m}} .
    \end{equation*}
    Since $f,g\in\cF$ take values in $[0,1]$, we have $\nrm{f-g}_\rS\leq 1$, so
    \begin{equation*}
        a d_\delta(f,g)
        =
        a\prn{\sqrt{\frac{\nrm{f-g}_\rS^2\cdot \log(2M/\delta)}{n}}+\frac{\log(2M/\delta)}{n}}
        \leq
        \threshold .
    \end{equation*}
    This gives $\cFhat{\cF}(\rS_1)
        \subseteq
        \{f\in\cF:\Rhat_{\rS_1}(f)\leq \Rhat_{\rS_1}(\erm)+\threshold\}.$
    By construction, $\erm\in\cFhat(\rS)$ and $\tilde f\in\cFhat(\rS)$, so
    $\midpoint=(\erm+\tilde f)/2\in\conv(\cFhat(\rS_1))$.
    Thus \cref{alg:midpoint} satisfies \cref{def:pruning-based} with a threshold satisfying $\limsup_{m\to\infty}\threshold=0$ for any fixed $M$ and $\delta$. Moreover, for sufficiently large $n=m$, we have
    \begin{equation*}
        \threshold\leq 2a\sqrt{\frac{\log(2M/\delta)}{m}},
    \end{equation*}
    which concludes the proof.
\end{proof}

In particular, for pruned-convex ERM and the midpoint estimator, $\limsup_{m\to\infty}\threshold=0$ for any fixed $M$ and $\delta$, and we may apply \cref{thm:pruning-exponential,thm:pruning-subexp-tail-strong,thm:pruning-no-along-the-tail} with $B_M(\delta)=C_1\sqrt{\log(2M/\delta)}$ and $B_M(\delta)=2a\sqrt{\log(2M/\delta)}$, respectively; see \cref{cor:pruned-erm-midpoint-exponential,cor:pc-ERM-midpoint-no-along-the-tail}.

\subsection{Proof of \cref{thm:pruning-subexp-tail-strong}}
\label{proof:pruning-subexp-tail-strong}

The proof of \cref{thm:pruning-subexp-tail-strong} uses the following lemma, which we prove after the main proof.
\begin{lemma}[Binomial upper-tail lower bound]
\label{lem:binomial-tail-lb}
There exist constants $\kappa_1,\kappa_2>0$ such that the following holds.
Let $S_m\sim \Binomial(m,p)$ and suppose $p,q\in(1/4,3/4)$ with $q>p$.
Then
\begin{equation*}
    \PP\prn{\frac{S_m}{m}> q}\ge \frac{\kappa_1}{1+\sqrt m\,(q-p)} \exp\!\prn{-\kappa_2\,m(q-p)^2}.
\end{equation*}
\end{lemma}
\begin{proof}[Proof of \cref{thm:pruning-subexp-tail-strong}]
Since $ m
\geq \alpha n $, any claim requiring $m$ to be sufficiently large can be ensured by taking $n$ sufficiently large. We use ``$m$ sufficiently large'' to mean that $n$ is sufficiently large. We now construct, for each $n$, a distribution $P$ satisfying the claim of the theorem.
Take $\cF=\{f_1,f_2\}$, with $f_1\equiv 0$ and $f_2\equiv 1$, let $X$ be deterministic and $Y\sim \Bernoulli\prn{\frac12-\mu}$ with $\mu$ chosen below.
Writing $p:=\frac12-\mu$, we get $\cR_P(f_1)=p$, $\cR_P(f_2)=1-p$, $\excessRisk(f_1)=0$, and $\excessRisk(f_2)=2\mu$,
and, on the pruning subsample $\rS_1$,
\begin{equation*}
    \Rhat_{\rS_1}(f_1)=\Ybar_1,\qquad\Rhat_{\rS_1}(f_2)=1-\Ybar_1,
\end{equation*}
where $\Ybar_1:=\frac1m\sum_{(X_i,Y_i)\in \rS_1} Y_i$.
Let $ \kappa_{2} $ be the universal constant from \cref{lem:binomial-tail-lb}, and define
\begin{equation*}
    \lambda_m:=\sqrt{\frac{\log m}{m}},\qquad d_m:=\max\crl{\threshold, \frac{\lambda_m}{2\sqrt{\kappa_{2}}}}, \qquad \text{and}\quad \mu:=d_m-\frac{\threshold}{2} \geq 0.
\end{equation*}
Since $\threshold\to0$ by assumption and $\lambda_m\to0$ as $m\to\infty$, we have $\mu\in(0,1/4]$ for all sufficiently large $m$.
Consider the event
\begin{equation*}
E:=\crl{\Ybar_1>\frac12+\frac{\threshold}{2}}.
\end{equation*}
On $E$, we have $\Rhat_{\rS_1}(f_1)-\Rhat_{\rS_1}(f_2)=2\Ybar_1-1>\threshold$, so $f_1\notin \cFhat(\rS_1)$.
Since $\cF=\{f_1,f_2\}$, this implies $\cFhat(\rS_1)=\{f_2\}$ and since $ \fhat\in\conv(\cFhat(\rS_1)) $, we have $\fhat=f_2$.
Therefore, on $E$,
\begin{equation*}
    \excessRisk(\fhat)=2\mu=2d_m-\threshold\ge d_m.
\end{equation*}

We now lower bound $\PP(E)$.
Let $S_m:=m\Ybar_1\sim \Binomial(m,p)$ and $q:=\frac12+\frac{\threshold}{2}$.
Substituting the definitions gives
\begin{equation*}
    q-p=\mu+\frac{\threshold}{2}=d_m
\end{equation*}
where the latter equality holds by choice of $\mu$.
Since $p,q\to 1/2$, we have $p,q\in(1/4,3/4)$ for all sufficiently large $m$.
Applying \cref{lem:binomial-tail-lb}, we obtain
\begin{equation*}
    \PP(E)=\PP\prn{\frac{S_m}{m}>q}\ge \frac{\kappa_1}{1+\sqrt m\,d_m}\exp\!\prn{-\kappa_2 m d_m^2},
\end{equation*}
where $\kappa_1,\kappa_2>0$ are universal constants.
We distinguish the two possible values of $d_m$.

\paragraph{Case 1: $d_m=\threshold$.}
Since $\threshold\le 1$ for all sufficiently large $m$, $1+\sqrt m\,\threshold\le 2 \sqrt m$,
so for sufficiently large $m$,
\begin{equation*}
\PP(E) \geq \frac{\kappa_1}{1+\sqrt m\,\threshold}\exp\prn{-\kappa_2 m\threshold^2}\ge c_{2} m^{-1/2}\exp\prn{-\kappa_2 m\threshold^2}=\eta,
\end{equation*}
with $ c_{2}=\min\{\kappa_1/2,1\} $.
Moreover, in this case we have that $\threshold\ge \sqrt{\log m/(4m\kappa_{2})}$, so $\log m\le 4\kappa_{2}m\threshold^2$.
Hence for $ m $ large enough such that $1\leq \log m\le 4\kappa_{2}m\threshold^2 $, we have
\begin{equation*}
    \log\frac1{\eta} =-\log (c_{2})+\frac12\log m + \kappa_2 m\threshold^2 \le (1-\log(c_{2}))4\kappa_{2}m\threshold^2.
\end{equation*}
Letting $ C_A:=4\kappa_2(1-\log(c_{2})) $, this implies, using $m\le n$, that for any $c_1\leq C_A^{-1/2}$,
\begin{equation*}
    c_1\sqrt{\frac{\log(1/\eta)}{n}} \leq \threshold = d_m \le \excessRisk(\fhat),
\end{equation*}
where the last inequality holds on $E$.

\paragraph{Case 2: $d_m=\sqrt{\log m/(4m\kappa_{2})}>\threshold$.}
In this case, we can bound
\begin{equation*}
\PP(E)\ge \frac{\kappa_1}{1+\sqrt{\log m/(4\kappa_{2})}}\,\exp\prn{-\log m/4} \ge\frac{\kappa_1}{1+\sqrt{\log m/(4\kappa_{2})}}\,m^{-1/4}.
\end{equation*}
Then, for all sufficiently large $m$, $\frac{\kappa_1}{1+\sqrt{\log m}/(2\sqrt{\kappa_{2}})}\,m^{-1/4}
\ge m^{-1/2}$.
Since $\exp\prn{-\kappa_2 m\threshold^2}\le 1$ and $ 0< c_{2} \leq 1 $, it holds for all sufficiently large $m$ that
\begin{equation*}
\PP(E)\ge c_2\,m^{-1/2}\exp\prn{-\kappa_2 m\threshold^2}=\eta.
\end{equation*}
In this case, $\threshold< \sqrt{\log m/(4m\kappa_{2})}$, so $m\threshold^2\le\log m/(4\kappa_{2})$.
Therefore, for sufficiently large $m$ with $ \log(m)\geq 1 $,
\begin{equation*}
\log\frac1{\eta}=-\log c_2+\frac12\log m + \kappa_2 m\threshold^2\le \prn{\frac{3}{4}-\log(c_{2})}\log m\leq C_B\log m,
\end{equation*}
where $ C_B:=1-\log(c_{2}) $.
Using $m\leq n$, we obtain again that for any $c_1\leq (4\kappa_2 C_B)^{-1/2}$, on $E$,
\begin{equation*}
    c_1\sqrt{\frac{\log(1/\eta)}{n}}\le \frac{1}{\sqrt{4\kappa_{2}}}\sqrt{\frac{\log m}{n}}\le \sqrt{\frac{\log m}{4m\kappa_{2}}}= d_m \le\excessRisk(\fhat).
\end{equation*}

Finally, letting $ c_{1}=\min\{C_A^{-1/2},(4\kappa_{2}C_B)^{-1/2}\} $ and combining the two cases, we obtain for all sufficiently large $m$,
\begin{equation*}
    \PP_{\rS\sim P^n}\prn{\excessRisk(\fhat)\ge c_1\sqrt{\frac{\log(1/\eta)}{n}}} \ge \eta.
\end{equation*}
Taking the supremum over $P$ proves the theorem with $c_3=\kappa_2$.
\end{proof}

It remains to prove \cref{lem:binomial-tail-lb}.
To that end, we use the following tail bound for binomial distributions.
\begin{lemma}[Simplified Theorem 2 by McKay \cite{mckay1989littlewood}]
\label{lem:McKay}
    Let $p\in(0,1)$, $m\geq 1$, and $mp< k\leq m$. Define $x=(k-mp)/\sigma$, where $\sigma^2 = mp(1-p)$. Then, for $S_m\sim \Binomial(m,p)$,
    \begin{equation*}
        \PP\prn{S_m\geq k} \geq \sigma \binom{m-1}{k-1}p^{k-1}(1-p)^{m-k} \frac{1-\Phi(x)}{\phi(x)},
    \end{equation*}
    where $\Phi$ and $\phi$ denote the standard normal CDF and PDF.
\end{lemma}
\begin{proof}[Proof of \cref{lem:binomial-tail-lb}]
For $ m\in\crl{1,\ldots,6 }$, the claim follows for $ \kappa_{1}\leq 1/4^6 $ because $ p\in(1/4,3/4) $. We may therefore assume that $ m\geq 7 $.
We first prove the claim with a nonstrict inequality in the probability and then show how to obtain the strict inequality.
Fix $p,q\in[1/4,3/4]$ with $q>p$, and set $k:=\ceil{mq}$.
Then $k\ge mq>mp$, so $x:=(k-mp)/\sigma>0$, and McKay's theorem applies.
As shown in \citep[bottom of page 15]{gasull2014approximating}, it holds that
$\frac{1-\Phi(x)}{\phi(x)} \geq \frac{1}{1+x}$ for $x> 0$.
Thus, by \cref{lem:McKay},
\begin{equation}
\label{eqn:McKay-applied}
    \PP\prn{\frac{S_m}{m}\ge q}
\ge \frac{\sigma}{1+x} \binom{m-1}{k-1}p^{k-1}(1-p)^{m-k}.
\end{equation}
As we assumed that $p\in[1/4,3/4]$ and $\sigma=\sqrt{mp(1-p)}$, we have
\begin{equation*}
    x=\frac{k-mp}{\sigma} = \frac{\ceil{mq}-mp}{\sqrt{mp(1-p)}} \le \frac{m(q-p)+1}{\sqrt{3m/16}} \le 3\prn{1+\sqrt m\,(q-p)}.
\end{equation*}
This gives
\begin{equation}
\label{eqn:lower-bound-1/(1+x)}
    \frac{1}{1+x}\ge \frac{1}{4+3\sqrt m\,(q-p)}.
\end{equation}
Because $m\geq 7$ and $p,q\in[1/4,3/4]$, we have $k=\ceil{mq}\leq \ceil{3m/4}\leq m-1$.
Therefore, by Proposition 5.4, Equation (111), in \cite{zhu2022nearly}, gives, uniformly for $p,q\in[1/4,3/4]$,
\begin{align}
    \binom{m-1}{k-1}p^{k-1}(1-p)^{m-k}&\ge \frac{1}{\sqrt{2(m-1)}}\exp\prn{-(m-1)\kl\prn{\frac{k-1}{m-1},p}} \nonumber \\
    &\geq c\,m^{-1/2}\exp\!\prn{-m\,\kl(q,p)}, \label{eqn:lower-bound-binompdf}
\end{align}
where
$
\kl(q,p):=q\log\frac{q}{p}+(1-q)\log\frac{1-q}{1-p},
$
and the last inequality is justified below in \cref{eqn:binomial-tail-lb} by noting that for $ m \geq 7 $, $(m-1)\kl\prn{\frac{k-1}{m-1},p}\le m\kl(q,p)+C$ for some constant $C>0$.
Since $\sigma= \sqrt{mp(1-p)}$, plugging \cref{eqn:lower-bound-1/(1+x),eqn:lower-bound-binompdf} into \cref{eqn:McKay-applied} gives
\begin{align*}
    \PP\prn{\frac{S_m}{m}\ge q} &\geq c\frac{\sqrt{mp(1-p)}}{4+3\sqrt m\,(q-p)} \frac{1}{\sqrt{m}}\exp\!\prn{-m\,\kl(q,p)} \\
    &\geq \frac{\kappa_1}{1+\sqrt{m}(q-p)}\exp\!\prn{-m\,\kl(q,p)}.
\end{align*}
Since $(p,q)\mapsto \kl(q,p)$ is $C^2$ on $[1/4,3/4]^2$ and vanishes on the diagonal, there exists $\kappa_2>0$ such that $\kl(q,p)\le \kappa_2(q-p)^2$ for all $p,q\in[1/4,3/4]$.
For any $p,q\in(1/4,3/4)$ with $q>p$, this yields
\begin{align*}
    \PP\prn{\frac{S_m}{m}\ge q} &\geq \frac{\kappa_1}{1+\sqrt m\,(q-p)} \exp\!\prn{-\kappa_2\,m(q-p)^2}.
\end{align*}
We can replace the nonstrict inequality by a strict inequality because we can always choose $ \eps > 0 $, such that $ \eps \leq 1/m $ and $ q+\eps<3/4 $ where by invoking the above bound with $ q+\eps $ in place of $ q $ gives
\begin{align*}
    \PP\prn{\frac{S_m}{m}> q} &\geq \PP\prn{\frac{S_m}{m}\geq q+\eps}\geq \frac{\kappa_1}{1+\sqrt m\,(q+\eps-p)} \exp\!\prn{-\kappa_2\,m(q+\eps-p)^2}
    \\
    &\geq \frac{\kappa_1}{2(1+\sqrt m\,(q-p))} \exp\!\prn{-4\kappa_2\,m((q-p)^2+\eps^2)} \tag{By $ (a+b)^{2}\leq 4(a^2+b^2) $ and $ \eps \leq 1/m $}
    \\
    &\geq \frac{\kappa_1}{2(1+\sqrt m\,(q-p))} \exp{(-4\kappa_{2} )}  \exp\!\prn{-4\kappa_2\,m(q-p)^2} \tag{By $ \eps \leq 1/m $ }.
\end{align*}
The claim follows by absorbing constants, i.e., replacing $\kappa_1$ by $\kappa_1 e^{-4\kappa_2}/2$ and $\kappa_2$ by $4\kappa_2$.

It remains to justify that, for $m\geq 7$, we have
\begin{equation}
\label{eqn:binomial-tail-lb}
(m-1)\kl\prn{\frac{k-1}{m-1},p}\le m\kl(q,p)+C
\end{equation}
for a universal constant $C>0$. Let $r:=(k-1)/(m-1)=(\ceil{mq}-1)/(m-1)$. Since
\begin{equation*}
    r-q=\frac{\ceil{mq}-mq+q-1}{m-1},
\end{equation*}
and $0\leq \ceil{mq}-mq<1$, we have
\begin{equation*}
    \abs{r-q}\leq \frac{\max\{q,1-q\}}{m-1}\leq \frac{3}{4(m-1)}.
\end{equation*}
Moreover, for $m\geq 7$ this implies $r\in[1/8,7/8]$. For $a\in[1/8,7/8]$ and $p\in[1/4,3/4]$,
\begin{equation*}
    \abs{\frac{\d}{\d a}\kl(a,p)}
    =
    \abs{\log\frac{a(1-p)}{p(1-a)}}
    \leq \log(21).
\end{equation*}
Hence, we get that
\begin{align*}
    |\kl(r,p)- \kl(q,p)| &= \abs{\int_{q}^{r}\frac{\d}{\d a}\kl(a,p)\d a}\leq\log(21)\abs{r-q}\leq \frac{3\log(21)}{4(m-1)},
    \\
    \implies \kl\prn{\frac{k-1}{m-1},p}&=\kl(r,p) \leq \kl(q,p) + \frac{3\log(21)}{4(m-1)}.
\end{align*}
Multiplying by $m-1$ and using $\kl(q,p)\geq 0$ gives \cref{eqn:binomial-tail-lb} with $C=3\log(21)/4$.
\end{proof}

\subsection{Proof of \cref{thm:pruning-no-along-the-tail}}
\label{proof:pruning-no-along-the-tail}

Fix $\delta\in(0,1)$ and $ M=2 $. By assumption, there exist $B_M(\delta)>0$ and $m_0\in\NN$ such that
\[
\threshold\le \frac{B_M(\delta)}{\sqrt m}
\qquad\text{for all } m\ge m_0.
\]

Define the dictionary $\cF=\crl{f_1,f_2}$ (hence $M=2$) with $f_1\equiv 0$ and $f_2\equiv 1$, and define the family of distributions $P_{\mu}$ via $\mu\in[0,1/4]$, using fixed $X$ and $Y\sim \Bernoulli(\frac{1}{2}-\mu)$.
Let $\rS_1\subseteq \rS$ be the pruning subsample of size $m\in[\alpha n, n]$, and let $\Ybar_1$ denotethe sample mean over $ \rS_{1} $. Then
\begin{equation*}
    \excessRisk(f_1) =0,\quad\excessRisk(f_2) = 2\mu,
    \quad  \Rhat_{\rS_1}(f_1) = \Ybar_1, \quad \Rhat_{\rS_1}(f_2) = 1-\Ybar_1.
\end{equation*}

For $m\ge m_0$, define $\mu:=\frac{B_M(\delta)}{4\sqrt m}$, and work under $P_{\mu}$. Consider the event
\begin{equation*}
    E:=\crl{\Ybar_1>\frac12+\frac{\threshold}{2}}.
\end{equation*}
On $E$, we have $\Rhat_{\rS_1}(f_1)-\Rhat_{\rS_1}(f_2)=2\Ybar_1-1>\threshold$, so $\erm(\rS_1)=f_2$ and $\Rhat_{\rS_1}(f_1)>\Rhat_{\rS_1}(\erm(\rS_1))+\threshold$.
This implies $f_1\notin \cFhat(\rS_1)$ and it follows that $\cFhat(\rS_1)=\{f_2\}=\conv(\cFhat(\rS_1))$ and since $ \fhat\in\conv(\cFhat(\rS_1)) $  we have $\fhat=f_2$.
Therefore, on $E$, since $m\le n$,
\begin{equation*}
    \excessRisk(\fhat)=2\mu=\frac{B_M(\delta)}{2\sqrt m}\geq \frac{B_M(\delta)}{2\sqrt n}.
\end{equation*}

It remains to lower bound $\PP(E)$. Let $p=\frac12-\mu$ and $\sigma^2:=p(1-p)$.
For sufficiently large $m$, we have $\mu\in[0,1/4]$, so $\sigma\in[\sqrt3/4,1/2]$. Also, because $\mu=B_M(\delta)/(4\sqrt m)$ and $\threshold\le B_M(\delta)/\sqrt m$,
\begin{equation*}
    E=\crl{\Ybar_1>p+\mu+\frac{\threshold}{2}}
\supseteq
\crl{\Ybar_1>p+\frac{3B_M(\delta)}{4\sqrt m}}.
\end{equation*}
We can now apply a version of the Berry--Esseen Theorem (see \citep{petrov1975sums} for an overview):
\begin{lemma}[Berry--Esseen \cite{shevtsova2011absolute}]
\label{lem:Berry-Esseen}
    Suppose $X_1,\ldots,X_n $ are i.i.d.\ real-valued random variables with mean $\mu$, variance $\sigma^2>0$, and third moment $\rho = \EE[\abs{X-\mu}^3]$. Then
    \begin{equation*}
        \sup_{x\in\RR}\abs{\PP\prn{\frac{\sum_{i=1}^n X_i -n\mu}{\sigma \sqrt{n}}\leq x}-\Phi(x)} \leq \frac{\rho}{2\sigma^3\sqrt{n}},
    \end{equation*}
    where $\Phi$ is the CDF of the standard normal distribution.
\end{lemma}
We apply this lemma with $Y_i \sim \Bernoulli(\frac{1}{2}-\mu)$ in $ \rS_{1} $, which yields for $p=1/2-\mu$ and $\mu\in[0,1/4]$ that
\begin{equation*}
    \frac{\rho}{\sigma^3} = \frac{p^2+(1-p)^2}{\sqrt{p(1-p)}} \leq \frac{5}{2\sqrt{3}}
    \tag{by $ \rho=p(1-p)((1-p)^2 + p^2)$ and $\sigma^2=p(1-p)$}
\end{equation*}
so \cref{lem:Berry-Esseen} yields
\begin{equation*}
    \sup_{x\in\RR} \abs{\PP\prn{ \Ybar_1  \leq p+ \frac{\sigma x}{\sqrt{m}}}-\Phi(x)}=\sup_{x\in\RR} \abs{\PP\prn{\sqrt{m} \frac{\Ybar_1 -p}{\sigma} \leq x}-\Phi(x)} \leq \frac{5}{4\sqrt{3}\sqrt{m}}.
\end{equation*}
Hence we can lower bound $\PP(E)$ as
\begin{equation*}
    \PP(E)\geq 1-\Phi\prn{\frac{3B_M(\delta)}{4\sigma}}-\frac{5}{4\sqrt3\sqrt m}\ge 1-\Phi(\sqrt3 B_M(\delta))-\frac{5}{4\sqrt3 \sqrt m} \geq \frac12\prn{1-\Phi(\sqrt3 B_M(\delta))}>0.
\end{equation*}
The last inequality holds for all sufficiently large $m$.

Combining the above estimates, for all sufficiently large $n$, $m\geq \alpha n$ is also sufficiently large, so
\[
\sup_{P} \PP_{\rS\sim P^n}\prn{\excessRisk(\fhat)\geq \frac{B_M(\delta)}{2\sqrt n}}
\ge \frac12\prn{1-\Phi(\sqrt3 B_M(\delta))}.
\]
This concludes the proof.

\section{Proof of \cref{thm:Q-exponential}}
\label{proof:Q-exponential}

\subsection{Summary of Relevant Quantities}
\label{subsec:summary-quantities}
For the remainder of this section, let $(P,\cF)\in\Theta_M$ be fixed. To streamline the analysis, we assume without loss of generality that the hypotheses are ordered according to their risks, so that $ \cR_P(f_1) \geq \cR_P(f_2) \geq \cdots \geq \cR_P(f_M) $. We use the following geometric quantities relative to the optimal hypothesis $f_M$:
\begin{itemize}
    \item $g_j := f_j - f_M$,
    \item $\Delta_j := \cR_P(f_j) - \cR_P(f_M)$,
    \item $d_{jM}^2 := \EE_{X\sim P_X}[(f_j(X)-f_M(X))^2]$,
    \item $G_{ij} := \EE_{X\sim P_X}[g_i(X)g_j(X)]$, forming the positive semidefinite Gram matrix $G \in \mathbb{R}^{(M-1)\times(M-1)}$.
\end{itemize}

Given an i.i.d.\ sample $\rS = \prn{(X_1, Y_1), \dots, (X_n, Y_n)}$, we define the corresponding empirical quantities and denote the empirical risk by $\Rhat_{\rS}(f)$:
\begin{itemize}
    \item $\hat{\Delta}_j := \Rhat_{\rS}(f_j) - \Rhat_{\rS}(f_M)$,
    \item $\hat{d}_{jM}^2 := \frac{1}{n}\sum_{m=1}^n (f_j(X_m)-f_M(X_m))^2$,
    \item $\hat{G}_{ij} := \frac{1}{n}\sum_{m=1}^n g_i(X_m)g_j(X_m)$, forming the empirical positive semidefinite Gram matrix $\hat{G} \in \mathbb{R}^{(M-1)\times(M-1)}$.
\end{itemize}
Let $ \triangle_M=\bigl\{  \rho\in[0,1]^{M}: \sum_{i=1}^M \rho_i = 1 \bigr\}  $ and $\blacktriangle_{M-1}=\bigl\{  \rho\in[0,1]^{M-1}: \sum_{i=1}^{M-1} \rho_i \le 1 \bigr\}  $ denote the standard probability simplex and the simplex with an inequality constraint, respectively.

\subsection{Rewriting the Minimization Problems}
\label{subsec:rewriting}
The $Q$-aggregation estimator $\rhohatQ$ from \cref{eqn:Q-aggregation-definition} (with $\phi\equiv 1$ and a flat prior $\pi$, implying that the regularization term can be omitted) is defined as a solution to the following minimization problem over the simplex $\triangle_M$:
\begin{align*}
\text{minimize over $\rho\in\mathbb{R}^{M}$}&:\quad  \frac{1}{2} \Rhat_{\rS} \!\left(\sum_{i=1}^{M} \rho_i f_i\right) + \frac{1}{2}\sum_{i=1}^{M} \rho_i \Rhat_{\rS}(f_i)
\\
\text{subject to}&:\quad \rho_{1}\geq 0, \ldots, \rho_{M}\geq 0, \quad \sum_{i=1}^{M} \rho_i = 1.
\end{align*}
In the analysis below, it will be convenient to consider the following minimization problem over the inequality-constrained simplex $\blacktriangle_{M-1}$:
\begin{align*}
    \text{minimize over $\rho\in\mathbb{R}^{M-1}$}&:\quad  \frac{1}{2} \Rhat_{\rS} \!\left(f_M + \sum_{i=1}^{M-1} \rho_i g_i\right) + \frac{1}{2} \Rhat_{\rS}(f_M) + \frac{1}{2}\sum_{j=1}^{M-1} \rho_j \hat{\Delta}_j
\\
\text{subject to}&:\quad  \rho_{1}\geq 0, \ldots, \rho_{M-1}\geq 0, \quad \sum_{i=1}^{M-1} \rho_i \leq 1.
\end{align*}

The two problems are equivalent in the following sense: any solution $\rhohatQ=(\rhohat_{Q,1}, \dots, \rhohat_{Q,M})\in \triangle_M$ to the first problem can be mapped to a valid vector $\rhohatQ' = (\rhohat_{Q,1}, \dots, \rhohat_{Q,M-1})\in \blacktriangle_{M-1}$ in the second problem. The objective values agree under this mapping, so the optimal objective value of the second problem is at most that of the first. Conversely, any solution $\rhohatQ'=(\rhohat_{Q,1}, \dots, \rhohat_{Q,M-1})\in \blacktriangle_{M-1}$ to the second problem can be mapped to a valid vector $\rhohatQ = (\rhohat_{Q,1}, \dots, \rhohat_{Q,M-1}, 1-\sum_{i=1}^{M-1} \rhohat_{Q,i})\in \triangle_M$ in the first problem. The objective values again agree under this mapping, so the optimal objective value of the first problem is at most that of the second. Therefore, the two problems share the same optimal objective value, and their optimal solutions can be mapped to one another in a one-to-one manner. From this point forward, we work primarily with the second problem and identify the solution $\rhohatQ$ found by the $Q$-aggregation algorithm with the solution to the second problem via the mapping $ (\rhohat_{Q,1}, \dots, \rhohat_{Q,M-1}, 1-\sum_{i=1}^{M-1} \rhohat_{Q,i}) $. For any $ \rho\in \blacktriangle_{M-1} $, let $f_\rho = f_M + \sum_{j=1}^{M-1} \rho_j g_j$ denote the corresponding convex combination of the base functions. Then the second objective function above can be written compactly as $\hat\Psi(\rho)=\frac{1}{2}\Rhat_{\rS}(f_\rho)+\frac{1}{2}\Rhat_{\rS}(f_M)+\frac{1}{2}\sum_{j=1}^{M-1}\rho_j\hat\Delta_j$.

We record the following properties of the two problems:

\paragraph{The Population Problem (Simplex-Constrained).}

    Define the functional $\Psi:\mathbb{R}^{M-1}\to \mathbb{R}$ as
\begin{align*}
        \Psi(\rho)
        &:=  \frac{1}{2} \cR_P \!\left(f_\rho\right) + \frac{1}{2} \cR_P(f_M) + \frac{1}{2}\sum_{j=1}^{M-1} \rho_j \Delta_j,
\end{align*}
    and define the constraint functions $ h_{i} : \mathbb{R}^{M-1}\to \mathbb{R} $ by $h_i(\rho) = -\rho_i$ for $i=1,\dots,M-1$, and $h_M(\rho) = \sum_{i=1}^{M-1} \rho_i - 1$.
    Consider the following minimization problem:
    \begin{align}
    \text{minimize:} \quad & \Psi(\rho) \nonumber\\
    \text{subject to:} \quad & \rho_{i} \geq 0 \text{ for all } i=1,\dots,M-1 \text{ and } \sum_{i=1}^{M-1} \rho_i \leq 1, \label{eq:Qgeneral:population_problem}
    \end{align}
    which is equivalently expressed as
    \begin{align*}
    \text{minimize:} \quad & \Psi(\rho) \\
    \text{subject to:} \quad & h_i(\rho) \leq 0 \text{ for all } i=1,\dots,M.
    \end{align*}
    The objective $\Psi$ is convex in $\rho$ because it is a sum of convex functions. Moreover, $\Psi$ is differentiable. Thus, the formulation constitutes a convex minimization problem with inequality constraints. The $i$-th constraint function $ h_i $ is convex and differentiable for all $ i=1,\dots,M $. Furthermore, the point $ \rho = (\frac{1}{M},\dots,\frac{1}{M}) $ is strictly feasible and lies in the relative interior of the domains of $\Psi$ and the $h_i$, namely $ \mathbb{R}^{M-1} $. This verifies Slater's condition and ensures that strong duality holds (see, e.g., \cite[page 226]{convexoptimizationbook}). We can therefore characterize any solution using the Karush--Kuhn--Tucker (KKT) conditions (see, e.g., \cite[page 244]{convexoptimizationbook}).

    The first KKT condition below is derived by differentiating the Lagrangian. By an expansion similar to that in \cref{lem:Qgeneral:mixture_risk}, for $\rho\in\blacktriangle_{M-1}$, we have
    \begin{align*}
        \Psi(\rho)=\cR_P(f_M)+\sum_{i=1}^{M-1}\rho_i\left(\Delta_i-\frac{1}{2}d_{iM}^2\right)+\frac{1}{2}\rho^\top G\rho.
    \end{align*}
    By the symmetry of $G$, the gradient takes the form
    \begin{align*}
        \frac{\partial \Psi}{\partial \rho_j}(\rho)
        =
        (G\rho)_j+\Delta_j-\frac{1}{2}d_{jM}^2.
    \end{align*}
    Since the constraint functions satisfy $\nabla h_j(\rho)=-e_j$ for $j=1,\dots,M-1$ and $\nabla h_M(\rho)=\mathbf{1}$, setting the gradient of the Lagrangian to zero yields the first KKT condition. Specifically, $ \rho_{Q}\in \blacktriangle_{M-1} $ is a solution to the minimization problem if and only if there exist Lagrange multipliers $\mu_j \ge 0$ for $j=1,\dots,M$ such that
        \begin{align}
        (G\rho_Q)_j + \Delta_j - \frac{1}{2}d_{jM}^2 + \mu_M - \mu_j &= 0, \quad j=1,\dots,M-1, \label{eq:Qgeneral:rhonegKL_KKT} \\
        \mu_j \rho_{Q,j} &= 0, \quad j=1,\dots,M-1, \nonumber \\
        \mu_M \Bigg(\sum_{i=1}^{M-1} \rho_{Q,i}-1 \Bigg) &= 0. \nonumber
        \end{align}

\paragraph{The Empirical Problem (Simplex-Constrained).}
            Define the empirical functional $\hat{\Psi}:\mathbb{R}^{M-1}\to \mathbb{R}$ as
            \begin{align*}
            \hat{\Psi}(\rho)
            &:=  \frac{1}{2} \Rhat_{\rS} \!\left(f_\rho\right) + \frac{1}{2} \Rhat_{\rS}(f_M) + \frac{1}{2}\sum_{j=1}^{M-1} \rho_j \hat{\Delta}_j.
            \end{align*}
            Consider the following minimization problem:
                \begin{align*}
    \text{minimize:} \quad & \hat{\Psi}(\rho) \\
    \text{subject to:} \quad & h_i(\rho) \leq 0 \text{ for all } i=1,\dots,M.
    \end{align*}
            By an argument formally identical to the one used for the population problem, we can characterize any solution $\rhohatQ\in \blacktriangle_{M-1}$ to the empirical problem using KKT conditions. A vector $\rhohatQ\in\blacktriangle_{M-1}$ is a solution to the empirical optimization problem if and only if there exist Lagrange multipliers $\hat{\mu}_j \ge 0$ for $j=1,\dots,M$ such that
        \begin{equation}
        \label{eq:Qgeneral:empirical-KKT}
            \begin{aligned}
                (\hat{G}\rhohatQ)_j + \hat{\Delta}_j - \frac{1}{2}\hat{d}_{jM}^2 + \hat{\mu}_M - \hat{\mu}_j &= 0, \quad j=1,\dots,M-1, \\
        \hat{\mu}_j \rhohat_{Q,j} &= 0, \quad j=1,\dots,M-1, \\
        \hat{\mu}_M \Bigg(\sum_{i=1}^{M-1} \rhohat_{Q,i}- 1\Bigg) &= 0.
            \end{aligned}
        \end{equation}
        The $Q$-aggregation estimator $\rhohatQ$ is defined as a solution to the empirical problem and therefore satisfies these empirical KKT conditions.

\begin{figure}
    \centering
    \begin{tikzpicture}[
  >=Latex,
  node distance=10mm and 10mm,
  every node/.style={
    draw,
    rounded corners,
    align=center,
    minimum width=22mm,
    minimum height=8mm
  }
]

\node (obs2) {\cref{lem:Qgeneral:obs2}};
\node (convergence) [right=of obs2] {\cref{lem:Qgeneral:convergence}: \\ exponential-probability event};
\node (mixture) [right=of convergence] {\cref{lem:Qgeneral:mixture_risk}};
\node (obs4) [right=of mixture] {\cref{lem:Qgeneral:obs4}};
\node (case1) [below=of $(obs2)!0.5!(convergence)$] {\cref{lem:Qgeneral:case1}: \\ case $\Psi(\rho_Q)<\cR_P(\fstar)$};
\node (case2) [right=28mm of case1] {\cref{lem:Qgeneral:case2}:\\ case $\Psi(\rho)\geq \cR_P(\fstar)$};
\node (Thm) [below=of $(case1)!0.5!(case2)$] {\cref{thm:Q-exponential}};

\draw[->] (obs2) -- (case1);
\draw[->] (convergence) -- (case1);
\draw[->] (convergence) -- (case2);
\draw[->] (mixture) -- (case2);
\draw[->] (obs4) -- (case2);
\draw[->] (case1) -- (Thm);
\draw[->] (case2) -- (Thm);

\end{tikzpicture}
    \caption{Proof structure of \cref{thm:Q-exponential}. We distinguish between two cases, each treated in \cref{lem:Qgeneral:case1,lem:Qgeneral:case2}, respectively. The same concentration result, \cref{lem:Qgeneral:convergence}, yields the result in both cases.}
    \label{fig:Q-proof}
\end{figure}
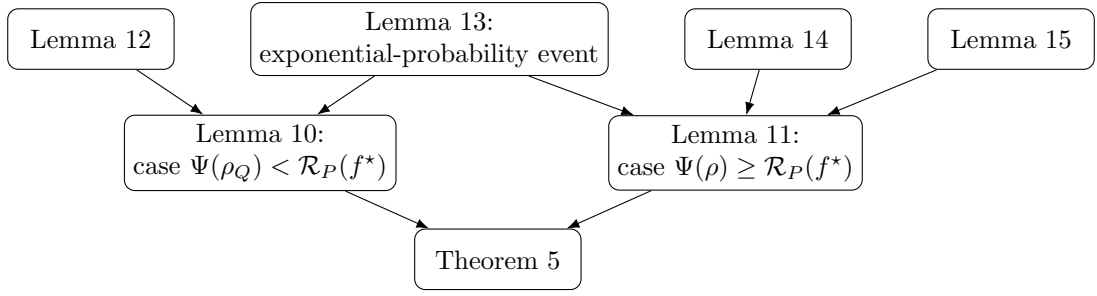

\subsection{Proof of \cref{thm:Q-exponential}}
In the sequel, we use the fact that $G$ is a symmetric positive semidefinite matrix. Consequently, it admits a decomposition $G=\sqrt{G}\sqrt{G}$, where $\sqrt{G}$ is the unique symmetric positive semidefinite square root of $G$ (see, e.g., \cite{Horn_Johnson_2012}).
We also use the established characterizations of the solutions to both the empirical and population problems via the KKT conditions. Because the domain $\blacktriangle_{M-1}$ is compact and the objectives are continuous, the infima of these problems are attained. We begin by stating two central lemmas from which \cref{thm:Q-exponential} directly follows. We provide their proofs in the following subsections. The overall proof structure is illustrated in \cref{fig:Q-proof}.
\begin{lemma}\label{lem:Qgeneral:case1}
If there exists a solution $ \rho_Q \in \blacktriangle_{M-1} $  to the population problem $ \Psi $ from \Cref{eq:Qgeneral:population_problem} such that $\Psi(\rho_Q) < \cR_P(f_M)$, then with probability at least $1-  CM^{2}\exp{(-c\frac{n}{M^{2}} )}  $,  $\cR_P(f_{\rhohatQ}) < \cR_P(f_M)$ for all sufficiently large $n$, where $c,C>0$ depend only on $\cF$, $P$, and an arbitrary but fixed choice of $\rho_Q \in \blacktriangle_{M-1} $ satisfying the condition $\Psi(\rho_Q) < \cR_P(f_M)$.
\end{lemma}

\begin{lemma}\label{lem:Qgeneral:case2}
If there does not exist a solution $ \rho_Q \in \blacktriangle_{M-1} $  to the population problem $ \Psi $ from \Cref{eq:Qgeneral:population_problem} such that $\Psi(\rho_Q) < \cR_P(f_M)$, then with probability at least $1-CM^{2}\exp{(-c\frac{n}{M^{2}} )}  $,  $\cR_P(f_{\rhohatQ}) \le \cR_P(f_M)$ for all sufficiently large $n$, where $c,C>0$ depend only on $\cF$ and $P$.
\end{lemma}
\begin{proof}[Proof of \cref{thm:Q-exponential}]
By \cref{lem:Qgeneral:case1,lem:Qgeneral:case2}, with probability at least $1- CM^{2}\exp{(-c\frac{n}{M^{2}})}  $, we have $\cR_P(f_{\rhohatQ}) \le \cR_P(f_M)$ for all sufficiently large $n$. Since $\cR_P(f_M) = \min_{j} \cR_P(f_j)$, this gives $\cR_P(f_{\rhohatQ}) \le \min_{j} \cR_P(f_j)$, which concludes the proof.
\end{proof}
We now prove the lemmas in the following two subsections.

\subsection{Proof of \cref{lem:Qgeneral:case1}}
To prove \cref{lem:Qgeneral:case1}, we use the following two lemmas, which we state here and prove after showing how they imply the lemma. Notice that in the following, the constants $c,C>0$ depend on $M$ and $\rho_Q$.
\begin{lemma}\label{lem:Qgeneral:obs2}
Let $ \rho_Q \in \blacktriangle_{M-1}$ be any solution to the population problem $ \Psi $ in \Cref{eq:Qgeneral:population_problem}.
If $\Psi(\rho_Q) < \cR_P(f_M)$, then $\cR_P(f_{\rho_Q}) < \cR_P(f_M)$.
\end{lemma}

\begin{lemma}\label{lem:Qgeneral:convergence}
There exist universal constants $c,C>0$ such that the following holds.
Let $ \rho_Q \in \blacktriangle_{M-1}$ be any solution to the population problem $ \Psi $ in \Cref{eq:Qgeneral:population_problem}. With probability at least $1- CM^{2}\exp{(-c\frac{\eps^{2}}{M^{2}}n )}  $, the following hold for all $j=1,\dots,M-1$:
\begin{align*}
 |(\hat{G}\rhohatQ)_j + \hat{\Delta}_j - \frac{1}{2}\hat{d}_{jM}^2- ((G\rhohatQ)_j + \Delta_j - \frac{1}{2}d_{jM}^2 )| \le &\eps, \\
\text{and} \quad \|\sqrt{G}( \rhohatQ - \rho_Q)\|^2_2=\ee_{X\sim P_X}\left[ \left( f_{\rhohatQ}(X) - f_{\rho_Q}(X) \right)^2 \right]
\leq &\eps.
\end{align*}
\end{lemma}

\begin{proof}[Proof of \cref{lem:Qgeneral:case1}] Let $\rho_Q$ be a solution to the population problem satisfying
$\Psi(\rho_Q)<\cR_P(f_M)$.
By \cref{lem:Qgeneral:obs2}, $\Psi(\rho_Q) < \cR_P(f_M)$ implies $\cR_P(f_{\rho_Q}) < \cR_P(f_M)$. Let $ c_{\mathrm{gap}}:= \cR_P(f_M) - \cR_P(f_{\rho_Q}) > 0 $. Using the Cauchy--Schwarz inequality in the final step, we bound the risk of the empirical minimizer $ \rhohatQ $ as
\begin{align*}
 \cR_P(f_{\rhohatQ})
 &= \ee_{(X,Y)\sim P}\left[ \left( f_{\rhohatQ}(X)  -Y \right)^2 \right] \tag{By definition of risk} \\
 &= \ee_{(X,Y)\sim P}\left[ \left( \left(f_{\rhohatQ}(X) - f_{\rho_Q}(X)\right) + \left(f_{\rho_Q}(X) - Y\right) \right)^2 \right] \tag{Adding and subtracting $f_{\rho_Q}$} \\
 &= \ee_{X\sim P_X}\left[ \left( f_{\rhohatQ}(X) - f_{\rho_Q}(X) \right)^2 \right] + \cR_P(f_{\rho_Q}) + 2\ee_{(X,Y)\sim P}\left[ \left( f_{\rhohatQ}(X) - f_{\rho_Q}(X) \right)(f_{\rho_Q}(X) - Y) \right] \tag{Expanding the square} \\
 &\leq \ee_{X\sim P_X}\left[ \left( f_{\rhohatQ}(X) - f_{\rho_Q}(X) \right)^2 \right] + \cR_P(f_{\rho_Q}) + 2\sqrt{\ee_{X\sim P_X}\left[ \left( f_{\rhohatQ}(X) - f_{\rho_Q}(X) \right)^2 \right] \cR_P(f_{\rho_Q})}. \tag{By Cauchy--Schwarz}
\end{align*}
Apply \cref{lem:Qgeneral:convergence} with $ \eps=c_{\mathrm{gap}}^{2}/16 $. For all sufficiently large $n$, with probability at least $ 1-CM^{2}\exp{(-c\frac{n}{M^{2}} )}  $ (where $c$ absorbs the constant $ \eps $, which depends on $ P $, $ \cF$, and the choice of $ \rho_{Q} $), the distance term satisfies $ \ee_{X\sim P_X}\left[ \left( f_{\rhohatQ}(X) - f_{\rho_Q}(X) \right)^2 \right] \le \eps $. Thus, on this event,
\begin{align*}
 \cR_P(f_{\rhohatQ})
 &\leq \eps + \cR_P(f_{\rho_Q}) + 2\sqrt{\eps \cR_P(f_{\rho_Q})} \tag{By \cref{lem:Qgeneral:convergence}} \\
 &\leq c_{\mathrm{gap}}^{2}/16 + \cR_P(f_{\rho_Q}) + 2\sqrt{\left(c_{\mathrm{gap}}^{2}/16\right) \cR_P(f_{\rho_Q})} \tag{By substituting $ \eps=c_{\mathrm{gap}}^{2}/16 $} \\
 &\leq \cR_P(f_{\rho_Q}) + c_{\mathrm{gap}}^{2}/16 + c_{\mathrm{gap}}/2 \tag{By $f_j,Y\in[0,1]$, implying $\cR_P(f_{\rho_Q})\leq 1$} \\
 &< \cR_P(f_M),
\end{align*}
where the last inequality again uses $f_j,Y\in[0,1]$, implying $c_{\mathrm{gap}} \le 1$, so $c_{\mathrm{gap}}^{2}/16 + c_{\mathrm{gap}}/2 < c_{\mathrm{gap}}$ and $ \cR_P(f_{\rho_Q}) = \cR_P(f_M) - c_{\mathrm{gap}} $.
\end{proof}
\begin{proof}[Proof of \cref{lem:Qgeneral:obs2}]
Recall the definition of the population objective $ \Psi $  evaluated at $\rho_Q$ (\cref{eq:Qgeneral:population_problem}):
\[
\Psi(\rho_Q) = \frac{1}{2} \cR_P(f_{\rho_Q}) + \frac{1}{2} \cR_P(f_M) + \frac{1}{2}\sum_{j=1}^{M-1} (\rho_Q)_j \Delta_j.
\]
Under the assumption $\Psi(\rho_Q) < \cR_P(f_M)$, substituting the definition gives
\[
\frac{1}{2} \cR_P(f_{\rho_Q}) + \frac{1}{2} \cR_P(f_M) + \frac{1}{2}\sum_{j=1}^{M-1} (\rho_Q)_j \Delta_j < \cR_P(f_M).
\]
Multiplying by $2$ and subtracting $\cR_P(f_M)$ from both sides yields
\[
\cR_P(f_{\rho_Q}) + \sum_{j=1}^{M-1} (\rho_Q)_j \Delta_j < \cR_P(f_M).
\]
Because $f_M$ is optimal, the excess risk $\Delta_j = \cR_P(f_j) - \cR_P(f_M) \ge 0$ is nonnegative for all $j=1, \dots, M-1$. Because $\rho_Q$ is simplex-constrained, all components satisfy $(\rho_Q)_j \ge 0$. The penalty term is therefore nonnegative: $\sum_{j=1}^{M-1} (\rho_Q)_j \Delta_j \ge 0$.

Using this nonnegativity gives
\[
\cR_P(f_{\rho_Q}) \le \cR_P(f_{\rho_Q}) + \sum_{j=1}^{M-1} (\rho_Q)_j \Delta_j < \cR_P(f_M).
\]
This proves that $\cR_P(f_{\rho_Q}) < \cR_P(f_M)$.
\end{proof}

\begin{proof}[Proof of \cref{lem:Qgeneral:convergence}]
  Consider the following events bounding the empirical quantities:
  \begin{align*}
  E_{1}&=\{ \forall j\in \{1,\dots,M-1\} : |\hat{\Delta}_{j}-\Delta_j| \le \eps/5 \} \\
  E_{2}&=\{ \forall i,j\in \{1,\dots,M-1\} : |\hat{G}_{i,j} - G_{i,j}| \le \frac{\eps}{5M} \} \\
  E_{3}&=\{ \forall j\in \{1,\dots,M-1\} : |\hat{d}_{jM}^2 - d_{jM}^2| \le \eps/5 \}
  \end{align*}
  Hoeffding's inequality and a union bound, using $f_j,Y\in[0,1]$, give
  \begin{align*}
    \PP(E_1\cap E_2\cap E_3)\ge 1- 2((M-1)^2+2(M-1))\exp\prn{-\frac{n\eps^2}{50M^2}}
\ge 1- 6M^2\exp\prn{-\frac{n\eps^2}{50M^2}}.
  \end{align*}
  We condition on the intersection of these events from now on.

  For any $j=1,\dots,M-1$:
  \begin{align}\label{eq:Qgeneral:convergence2}
  \left|(\hat{G}\rhohatQ)_j + \hat{\Delta}_j - \frac{1}{2}\hat{d}_{jM}^2- \left((G\rhohatQ)_j + \Delta_j - \frac{1}{2}d_{jM}^2 \right)\right| &\le \sum_{i=1}^{M-1} |\hat{G}_{ji} - G_{ji}| \rhohat_{Q,i} + |\hat{\Delta}_j - \Delta_j| + \frac{1}{2}|\hat{d}_{jM}^2 - d_{jM}^2| \nonumber \\
  &\le \frac{\eps}{5M}(M-1) + \frac{\eps}{5} + \frac{\eps}{10} < \frac{\eps}{2} < \eps,
  \end{align}
  which proves the first claim of the lemma.

 For convenience, define
  \[
  \rho_{Q,M} := 1-\sum_{j=1}^{M-1} (\rho_Q)_j,
  \qquad
  \rhohat_{Q,M} := 1-\sum_{j=1}^{M-1} \rhohat_{Q,j}.
  \]
  By the KKT optimality conditions, we have $ \mu_{M}\rho_{Q,M}=0 $ and $ \hat{\mu}_M\rhohat_{Q,M}=0 $. Since $ \rho_{Q},\rhohatQ\in \blacktriangle_{M-1} $, we also have $ \rho_{Q,M},\rhohat_{Q,M}\geq 0 $.
  The KKT optimality conditions (\cref{eq:Qgeneral:rhonegKL_KKT}) give:
  \begin{align*}
      ( G\rho_Q )_j + \Delta_j - \frac{1}{2}d_{jM}^2 &= \mu_j - \mu_M,
  \end{align*}
 and
  \begin{align*}
      ( \hat{G}\rhohatQ )_j + \hat{\Delta}_j - \frac{1}{2}\hat{d}_{jM}^2 &= \hat{\mu}_j - \hat{\mu}_M.
  \end{align*}
  Subtracting the two equations yields
  \begin{align*}
      ( \hat{G}\rhohatQ )_j - ( G\rho_Q )_j + \hat{\Delta}_j - \Delta_j - \frac{1}{2}(\hat{d}_{jM}^2 - d_{jM}^2) = -(\hat{\mu}_M-\mu_M) + \hat{\mu}_j - \mu_j.
  \end{align*}
  Multiplying both sides by $ (\rhohat_{Q,j} - \rho_{Q,j}) $ and summing over $ j=1,\dots,M-1 $, we obtain
  \begin{align*}
&\sum_{j=1}^{M-1} (\rhohat_{Q,j} - \rho_{Q,j}) \left[ ( \hat{G}\rhohatQ )_j - ( G\rho_Q )_j + \hat{\Delta}_j - \Delta_j - \frac{1}{2}(\hat{d}_{jM}^2 - d_{jM}^2) \right]
\\
    &=\sum_{j=1}^{M-1} (\rhohat_{Q,j} - \rho_{Q,j}) \left[ -(\hat{\mu}_M-\mu_M) + \hat{\mu}_j - \mu_j \right]
    \\
      &= -(\hat{\mu}_M-\mu_M)\sum_{j=1}^{M-1} (\rhohat_{Q,j} - \rho_{Q,j}) + \sum_{j=1}^{M-1} (\rhohat_{Q,j} - \rho_{Q,j})(\hat{\mu}_j - \mu_j)
      \\
      &= (\hat{\mu}_M-\mu_M)\sum_{j=1}^{M-1} (\rho_{Q,j}-\rhohat_{Q,j} ) + \sum_{j=1}^{M-1} (\rhohat_{Q,j} \hat{\mu}_j - \rhohat_{Q,j} \mu_j - \rho_{Q,j} \hat{\mu}_j + \rho_{Q,j} \mu_j)
      \\
      &= (\hat{\mu}_M-\mu_M)(\rhohat_{Q,M} - \rho_{Q,M}) + \sum_{j=1}^{M-1} (\rhohat_{Q,j} \hat{\mu}_j - \rhohat_{Q,j} \mu_j - \rho_{Q,j} \hat{\mu}_j + \rho_{Q,j} \mu_j)
      \tag{By $\sum_{j=1}^{M-1} \rho_{Q,j} = 1 - \rho_{Q,M}$ and $\sum_{j=1}^{M-1} \rhohat_{Q,j} = 1 - \rhohat_{Q,M}$}
      \\
      &= -\hat{\mu}_M\rho_{Q,M} - \mu_M\rhohat_{Q,M} - \sum_{j=1}^{M-1} (\rhohat_{Q,j} \mu_j + \rho_{Q,j} \hat{\mu}_j)
      \tag{By $ \hat{\mu}_M\rhohat_{Q,M} = 0 $, $ \mu_M\rho_{Q,M} = 0 $, $ \rhohat_{Q,j}\hat{\mu}_j = 0 $, and $ \rho_{Q,j}\mu_j = 0 $}
      \\
      &\le 0,
      \tag{By $ \hat{\mu}_M,\rhohat_{Q,j},\mu_M,\rho_{Q,j},\hat{\mu}_j ,\mu_j \geq 0 $}
  \end{align*}
  where the second to last equality uses the KKT conditions again.
  Using $( \hat{G}\rhohatQ )_j - ( G\rho_Q )_j = (G(\rhohatQ-\rho_Q))_j + ((\hat{G}-G)\rhohatQ)_j$ and the preceding inequality, we obtain
  \begin{align*}
      \sum_{j=1}^{M-1} (\rhohat_{Q,j} - \rho_{Q,j})& (G(\rhohatQ-\rho_Q))_j + \sum_{j=1}^{M-1} (\rhohat_{Q,j} - \rho_{Q,j}) \left[ ((\hat{G}-G)\rhohatQ)_j + \hat{\Delta}_j - \Delta_j - \frac{1}{2}(\hat{d}_{jM}^2 - d_{jM}^2) \right]
      \\
       &=\sum_{j=1}^{M-1} (\rhohat_{Q,j} - \rho_{Q,j}) \left[ ( \hat{G}\rhohatQ )_j - ( G\rho_Q )_j + \hat{\Delta}_j - \Delta_j - \frac{1}{2}(\hat{d}_{jM}^2 - d_{jM}^2) \right]\leq0
       \\
      \implies \|\sqrt{G}(\rhohatQ - \rho_Q)\|^2_2 &\le \sum_{j=1}^{M-1} (\rho_{Q,j} - \rhohat_{Q,j}) \left[ (\hat{G}\rhohatQ)_j - (G\rhohatQ)_j + \hat{\Delta}_j - \Delta_j - \frac{1}{2}(\hat{d}_{jM}^2 - d_{jM}^2) \right] \\
      &\le \sum_{j=1}^{M-1} |\rho_{Q,j} - \rhohat_{Q,j}| \cdot \left| (\hat{G}\rhohatQ)_j - (G\rhohatQ)_j + \hat{\Delta}_j - \Delta_j - \frac{1}{2}(\hat{d}_{jM}^2 - d_{jM}^2) \right|.
  \end{align*}
  The absolute value in the sum is the empirical deviation bounded in \cref{eq:Qgeneral:convergence2}. Thus,
  \begin{align*}
      \|\sqrt{G}(\rhohatQ - \rho_Q)\|^2_2 &\le \sum_{j=1}^{M-1} |\rho_{Q,j} - \rhohat_{Q,j}| \cdot \frac{\eps}{2} \\
      &\le \frac{\eps}{2} \sum_{j=1}^{M-1} (\rho_{Q,j} + \rhohat_{Q,j}) \le \eps,
  \end{align*}
  where the final inequality uses $ \rhohatQ, \rho_Q \in \blacktriangle_{M-1} $, which ensures that the sum of their coordinates is at most $ 2 $. This completes the proof of the lemma.
\end{proof}

\subsection{Proof of \cref{lem:Qgeneral:case2}}
To prove \cref{lem:Qgeneral:case2}, we use \cref{lem:Qgeneral:convergence} from the previous section and the following two lemmas, which we state here and prove after showing how they imply \cref{lem:Qgeneral:case2}.

\begin{lemma}\label{lem:Qgeneral:mixture_risk}
For any $ \rho\in \blacktriangle_{M-1} $, we have
\begin{align*}
\cR_P(f_\rho) = \cR_P(f_M) + \sum_{i =1}^{M-1} \rho_i (\Delta_i - d_{iM}^2) + \sum_{i =1}^{M-1} \sum_{j=1}^{M-1} \rho_i \rho_j G_{ij}.
\end{align*}
\end{lemma}

\begin{lemma}\label{lem:Qgeneral:obs4}
    There exists $ \rho\in \blacktriangle_{M-1} $ such that $\sum_{j=1}^{M-1} \rho_j d_{jM}^2 > 2 \sum_{j=1}^{M-1} \rho_j \Delta_j$ if and only if there exists a solution $ \rho_Q \in \blacktriangle_{M-1} $ to the population problem $ \Psi $ in \cref{eq:Qgeneral:population_problem} such that $\Psi(\rho_Q) < \cR_P(f_M)$.
\end{lemma}

\begin{proof}[Proof of \cref{lem:Qgeneral:case2}]
Since there is no solution $ \rho_Q \in\blacktriangle_{M-1}$ to the population problem $ \Psi $ such that $\Psi(\rho_Q) < \cR_P(f_M)$, and $ \Psi(0)=\cR_P(f_M) $ (where $0 \in \mathbb{R}^{M-1}$ denotes the zero vector, meaning all weight is on the optimal function $f_M$), the zero vector $0 \in \mathbb{R}^{M-1}$ is a solution to the population problem. For the remainder of the proof, let $ \rho_Q=0 $. Define
\begin{align*}
I_{>} &= \left\{j \in \{1, \ldots, M-1\} : \Delta_j - \frac{1}{2}d_{jM}^2 > 0\right\}, \\
I_{=,=} &= \left\{j \in \{1, \ldots, M-1\} : \Delta_j - \frac{1}{2}d_{jM}^2 = 0, \Delta_j = 0\right\}, \\
I_{=,\not=} &= \left\{j \in \{1, \ldots, M-1\} : \Delta_j - \frac{1}{2}d_{jM}^2 = 0, \Delta_j \neq 0\right\}.
\end{align*}

We condition on the high-probability event of \cref{lem:Qgeneral:convergence} with $ \rho_Q=0 $ and any fixed $\eps\in(0,1]$ satisfying the following inequalities (if the sets $I_{>}$ and $I_{=,\not=}$ are empty, we can choose any $\eps\in(0,1]$):
\begin{alignat}{2}
 \eps+\sqrt{\eps} &\leq (\Delta_j - \frac{1}{2}d_{jM}^2)/4 \quad &&\text{ for all } j\in I_{>},\label{eq:Qgeneral:case2eps1}
 \\
    \sqrt{\eps} &\leq \Delta_j/2 &&\text{ for all } j\in I_{=,\not=}.\label{eq:Qgeneral:case2eps2}
\end{alignat}
Such an $\eps$ exists because $I_{>}$ and $I_{=,\not=}$ are finite, and the margins $\Delta_j - \frac{1}{2}d_{jM}^2$ for $j\in I_{>}$ and $\Delta_j$ for $j\in I_{=,\not=}$ are strictly positive by definition of these sets. For the latter, $ \Delta_{j}\neq0 $ implies $\Delta_j > 0$ because $f_M$ attains the minimal risk. This choice of $ \eps $ depends only on $ \cF $ and $P$.
Thus, the event from \cref{lem:Qgeneral:convergence} occurs with probability at least $ 1-CM^{2}\exp{(-c\frac{n}{M^{2}} )}  $ (where $c$ absorbs the dependence on $ \eps $, which depends only on $ \cF $ and $ P $): For all $j=1,\dots,M-1$,
\begin{align}
 |(\hat{G}\rhohatQ)_j + \hat{\Delta}_j - \frac{1}{2}\hat{d}_{jM}^2- ((G\rhohatQ)_j + \Delta_j - \frac{1}{2}d_{jM}^2 )| &\le \eps \label{eq:Qgeneral:case21} \\
    \text{and} \quad \|\sqrt{G} \rhohatQ\|^{2}_2=\|\sqrt{G}( \rhohatQ - \rho_Q)\|^2_2=\ee_{X}\left[ \left( f_{\rhohatQ}(X) - f_M(X) \right)^2 \right]
&\leq  \eps \label{eq:Qgeneral:case22}
\end{align}
because $\rho_Q=0$.

\paragraph{Bounding the coordinates of $ G\rhohatQ $.}
The squared Euclidean norm of the $j$-th row of $\sqrt{G}$ (denoted $\sqrt{G}_{j,\cdot}$) satisfies $\|\sqrt{G}_{j,\cdot}\|_2^2 = \sum_{i} (\sqrt{G})_{ji}(\sqrt{G})_{ij} = (\sqrt{G}\sqrt{G})_{jj} = G_{jj} = d_{jM}^2 \leq 1$. Cauchy--Schwarz inequality gives
\begin{align}
 &|(G\rhohatQ)_j|\label{eq:Qgeneral:case23}
 \\
 &= |(\sqrt{G}_{j,\cdot})^{\top} (\sqrt{G}\rhohatQ)| \tag{Since $ G $ is positive semidefinite} \\
 &\leq \|\sqrt{G}_{j,\cdot}\|_2 \|\sqrt{G} \rhohatQ\|_2 \tag{By the Cauchy--Schwarz inequality} \\
 &\leq \|\sqrt{G} \rhohatQ\|_2 \tag{Since $\|\sqrt{G}_{j,\cdot}\|_2 \le 1$} \\
 &\leq \sqrt{\eps}. \tag{By \cref{eq:Qgeneral:case22}}
\end{align}

\paragraph{Coordinates in $ I_{>} $ vanish for $ \eps $ sufficiently small.}
By the KKT characterization \eqref{eq:Qgeneral:empirical-KKT} of the empirical problem, there exist $\hat{\mu}_{j}\geq 0$ for $j=1,\dots,M$ such that
\begin{align*}
        (\hat{G}\rhohatQ)_j + \hat{\Delta}_j - \frac{1}{2}\hat{d}_{jM}^2 + \hat{\mu}_M - \hat{\mu}_j &= 0, \qquad \text{for all }j\in\crl{1,\dots,M-1}, \\
        \hat{\mu}_j \rhohat_{Q,j} &= 0, \qquad \text{for all } j\in\crl{1,\dots,M-1}.
\end{align*}
Fix $j\in I_{>}$. By \cref{eq:Qgeneral:case21} and the first KKT equation above, we can replace the empirical quantities by their population counterparts with an additive $ \eps $ error:
\begin{align*}
  (G\rhohatQ)_j + \Delta_j - \frac{1}{2}d_{jM}^2 + \hat{\mu}_M - \hat{\mu}_j \in(-\eps,\eps). \tag{Substituting empirical quantities}
\end{align*}
By \cref{eq:Qgeneral:case23}, $ (G\rhohatQ)_j $ lies in $ (-\sqrt{\eps},\sqrt{\eps}) $. We obtain
\begin{align}\label{eq:Qgeneral:case24}
      \Delta_j - \frac{1}{2}d_{jM}^2 + \hat{\mu}_M - \hat{\mu}_j \in(-\eps-\sqrt{\eps},\eps + \sqrt{\eps}).
\end{align}
Since $ \hat{\mu}_M\geq 0 $ by the KKT characterization \eqref{eq:Qgeneral:empirical-KKT},
\begin{align*}
      \Delta_j - \frac{1}{2}d_{jM}^2 - \hat{\mu}_j \leq \eps+\sqrt{\eps}.
\end{align*}
Because $j\in I_{>}$ so $\Delta_j - \frac{1}{2}d_{jM}^2>0$ and $ \eps $ was chosen so that $ \eps+\sqrt{\eps} \leq (\Delta_j - \frac{1}{2}d_{jM}^2)/4 $ (\cref{eq:Qgeneral:case2eps1}), we obtain
\[
\hat{\mu}_j \geq \frac{3}{4}\left(\Delta_j - \frac{1}{2}d_{jM}^2\right) > 0.
\]
The second condition in \eqref{eq:Qgeneral:empirical-KKT} now gives $ \rhohat_{Q,j} = 0 $.

\paragraph{Conclusion.}
By the hypothesis of this case and the equivalence in \cref{lem:Qgeneral:obs4}, every $ \rho\in \blacktriangle_{M-1} $ satisfies $\sum_{j=1}^{M-1} \rho_j d_{jM}^2 \leq 2 \sum_{j=1}^{M-1} \rho_j \Delta_j$. Setting $\rho = e_j$ (the unit vector for the $j$-th component) gives $d_{jM}^2 \leq 2 \Delta_j$, so $ \Delta_{j}-\frac{1}{2}d_{jM}^2 \geq 0 $ for all $ j\in\{  1,\ldots,M-1\}$. It follows that $ I_{>}, I_{=,=} $, and $ I_{=,\not=} $ form a partition of $ \{  1,\ldots,M-1\} $.
By \cref{lem:Qgeneral:mixture_risk}, we can write the risk of the empirical estimator $f_{\rhohatQ}$ as
\begin{align*}
\cR_P(f_{\rhohatQ}) &= \cR_P(f_M) + \sum_{i =1}^{M-1} \rhohat_{Q,i} (\Delta_i - d_{iM}^2) + \sum_{i =1}^{M-1} \sum_{j=1}^{M-1} \rhohat_{Q,i} \rhohat_{Q,j} G_{ij} \tag{By \cref{lem:Qgeneral:mixture_risk}} \\
&= \cR_P(f_M) + \sum_{i =1}^{M-1} \rhohat_{Q,i} (\Delta_i - d_{iM}^2) + \sum_{i =1}^{M-1} \rhohat_{Q,i} (G\rhohatQ)_{i} \tag{By definition of the Gram matrix $ G $} \\
&= \cR_P(f_M) + \sum_{i =1}^{M-1} \rhohat_{Q,i} \left(\Delta_i - d_{iM}^2+ (G\rhohatQ)_{i} \right) \tag{Factoring out $\rhohat_{Q,i}$} \\
&= \cR_P(f_M) + \sum_{i\in I_{>} } \rhohat_{Q,i} \left(\Delta_i - d_{iM}^2+ (G\rhohatQ)_{i} \right) \\
&\quad + \sum_{i\in I_{=,=} } \rhohat_{Q,i} \left(\Delta_i - d_{iM}^2+ (G\rhohatQ)_{i} \right) + \sum_{i\in I_{=,\not=} } \rhohat_{Q,i} \left(\Delta_i - d_{iM}^2+ (G\rhohatQ)_{i} \right) \tag{Splitting the sum}
\end{align*}
We now bound each sum above by $0$; this implies $\cR_P(f_{\rhohatQ}) \leq \cR_P(f_M)$
and concludes the proof.
For the first sum over $ I_{>} $, we proved above that $ \rhohat_{Q,i}=0  $, so the entire sum equals zero:
\begin{align*}
 \sum_{i\in I_{>} } \rhohat_{Q,i} (\Delta_i - d_{iM}^2+ (G\rhohatQ)_{i} ) = 0.
\end{align*}
For the second sum over $ I_{=,=} $, we have $ \Delta_i - \frac{1}{2}d_{iM}^2 = 0 $ and $ \Delta_i = 0 $, which imply $ d_{iM}^{2}=0 $. The latter implies $ f_i = f_M $ $P$-almost everywhere. This gives $ g_i = 0 $ $P$-almost everywhere and $ G_{ij}=\ee_{X\sim P_X}[g_{i}g_{j}]=0 $ for every $ j\in\{  1,\ldots,M-1\} $. This further implies $ (G\rhohatQ)_{i}=\sum_{j=1}^{M-1} G_{ij} \rhohat_{Q,j} = 0 $, so the second sum equals $ 0 $:
\begin{align*}
 \sum_{i\in I_{=,=} } \rhohat_{Q,i} (\Delta_i - d_{iM}^2+ (G\rhohatQ)_{i} ) = 0.
\end{align*}
Finally, for the sum over $ I_{=,\not=} $, we have $ \Delta_i - \frac{1}{2}d_{iM}^2 = 0 $ and $ \Delta_i \neq 0 $, implying that $ \Delta_i-d_{iM}^2 = -\Delta_{i} < 0 $. By \cref{eq:Qgeneral:case23}, $ (G\rhohatQ)_{i}\leq \sqrt{\eps} $. By the choice of $\eps$ (\cref{eq:Qgeneral:case2eps2}), we have $ (G\rhohatQ)_{i}\leq \Delta_{i}/2 $. Since $ \rhohat_{Q,i} \geq 0 $ and $ \Delta_{i}>0 $,
\begin{align*}
 \sum_{i\in I_{=,\not=} } \rhohat_{Q,i} (\Delta_i - d_{iM}^2+ (G\rhohatQ)_{i} ) &= \sum_{i\in I_{=,\not=} } \rhohat_{Q,i} (-\Delta_i + (G\rhohatQ)_{i} ) \\
 &\leq \sum_{i\in I_{=,\not=} } \rhohat_{Q,i} (-\Delta_i/2) \\
 &\leq 0.
\end{align*}
This concludes the proof.
\end{proof}

\begin{proof}[Proof of \cref{lem:Qgeneral:mixture_risk}]
 Let $f_\rho = f_M + \sum_{i =1}^{M-1} \rho_i g_i$. The risk of this combination expands as
\begin{align*}
    \cR_P(f_\rho) &= \ee_{(X,Y)\sim P}\left[\left(f_M(X) - Y + \sum_{i =1}^{M-1} \rho_i g_i(X)\right)^2\right] \tag{By definition of the expected risk $\cR_P(\cdot)$} \\
    &= \ee_{(X,Y)\sim P}\left[\left(f_M(X) - Y\right)^2\right] + 2 \sum_{i =1}^{M-1} \rho_i \ee_{(X,Y)\sim P}\left[\left(f_M(X) - Y\right)g_i(X)\right] + \ee_{X\sim P_X}\left[\left(\sum_{i =1}^{M-1} \rho_i g_i(X)\right)^2\right] \tag{By expanding the square} \\
    &= \cR_P(f_M) + 2 \sum_{i =1}^{M-1} \rho_i \ee_{(X,Y)\sim P}\left[\left(f_M(X) - Y\right)g_i(X)\right] + \sum_{i =1}^{M-1} \sum_{j=1}^{M-1} \rho_i \rho_j G_{ij}. \tag{By definition of $\cR_P(f_M)$ and the Gram matrix $G$}
\end{align*}
To evaluate the cross-term, we analyze the risk of an individual hypothesis $f_i$ for $i \in \{1,\dots,M-1\}$:
\begin{align*}
    \cR_P(f_i) &= \ee_{(X,Y)\sim P}\left[\left(f_M(X) + g_i(X) - Y\right)^2\right] \tag{By substituting $f_i = f_M + g_i$} \\
    &= \cR_P(f_M) + 2\ee_{(X,Y)\sim P}\left[\left(f_M(X) - Y\right)g_i(X)\right] + \ee_{X\sim P_X}\left[g_i(X)^2\right] \tag{By expanding the square} \\
    &= \cR_P(f_M) + 2\ee_{(X,Y)\sim P}\left[\left(f_M(X) - Y\right)g_i(X)\right] + d_{iM}^2. \tag{Using $d_{iM}^2 = \ee_{X\sim P_X}[g_i(X)^2]$}
\end{align*}
Rearranging terms yields an expression for the expected cross-term:
\begin{align*}
    2\ee_{(X,Y)\sim P}\left[\left(f_M(X) - Y\right)g_i(X)\right] &= \cR_P(f_i) - \cR_P(f_M) - d_{iM}^2 \tag{By isolating the cross-term} \\
    &= \Delta_i - d_{iM}^2. \tag{By substituting the excess risk $\Delta_i = \cR_P(f_i) - \cR_P(f_M)$}
\end{align*}
Substituting this back into the expression for $\cR_P(f_\rho)$ gives
\begin{align*}
    \cR_P(f_\rho) &= \cR_P(f_M) + \sum_{i =1}^{M-1} \rho_i (\Delta_i - d_{iM}^2) + \sum_{i =1}^{M-1} \sum_{j=1}^{M-1} \rho_i \rho_j G_{ij}
\end{align*}
as claimed.
\end{proof}

\begin{proof}[Proof of \cref{lem:Qgeneral:obs4}]
    ($\Rightarrow$) Assume there exists a $\rho\in \blacktriangle_{M-1}$ such that $\sum_{j=1}^{M-1} \rho_j \left(\Delta_j - \frac{1}{2}d_{jM}^2\right) < 0$. Consider the convex combination moving from the optimal function $f_M$ towards the mixture $f_{\rho}$. Let $\alpha \in (0, 1]$ and define the weight vector $ \alpha \rho$. Since $\rho \in \blacktriangle_{M-1}$ and $\alpha \le 1$, we have $\alpha \rho \in \blacktriangle_{M-1}$.

    Evaluating the objective $\Psi$ of the population problem at this mixture yields
    \begin{align*}
        \Psi(\alpha \rho)
        &= \frac{1}{2} \cR_P\left(f_M + \alpha \sum_{j=1}^{M-1} \rho_j g_j\right) + \frac{1}{2} \cR_P(f_M) + \frac{\alpha}{2} \sum_{j=1}^{M-1} \rho_j \Delta_j \tag{By definition of $\Psi$} \\
        &= \frac{1}{2} \left[ \cR_P(f_M) + \alpha \sum_{j=1}^{M-1} \rho_j (\Delta_j - d_{jM}^2) + \alpha^2 \rho^\top G \rho \right] + \frac{1}{2} \cR_P(f_M) + \frac{\alpha}{2} \sum_{j=1}^{M-1} \rho_j \Delta_j \tag{By \cref{lem:Qgeneral:mixture_risk}} \\
        &= \cR_P(f_M) + \alpha \sum_{j=1}^{M-1} \rho_j \left(\Delta_j - \frac{1}{2}d_{jM}^2\right) + \frac{1}{2}\alpha^2 \rho^\top G \rho. \tag{Simplifying terms}
    \end{align*}
    We view this as a quadratic function in $\alpha$ with linear coefficient $\sum_{j=1}^{M-1} \rho_j \left(\Delta_j - \frac{1}{2}d_{jM}^2\right)$ and  quadratic coefficient $\frac{1}{2} \rho^\top G \rho$.
    By assumption, the linear coefficient is strictly negative:
    \begin{align*}
        \sum_{j=1}^{M-1} \rho_j \left(\Delta_j - \frac{1}{2}d_{jM}^2\right) < 0.
    \end{align*}
    Let this strictly negative constant be $-a$ (where $a > 0$), and let $b = \frac{1}{2} \rho^\top G \rho \ge 0$. The objective becomes
    \begin{align*}
        \Psi(\alpha \rho) = \cR_P(f_M) - \alpha a + \alpha^2 b.
    \end{align*}
    If $b=0$, then $\Psi(\alpha \rho)=\cR_P(f_M)-\alpha a<\cR_P(f_M)$ for every $\alpha\in(0,1]$. If $b>0$, then for any sufficiently small $\alpha \in \left(0, \min\crl{1, \frac{a}{b}}\right)$, we have $-\alpha a + \alpha^2 b< 0$. In either case, there exists $\alpha\in(0,1]$ such that
    \begin{align*}
        \Psi(\alpha \rho) < \cR_P(f_M).
    \end{align*}
    Let $\rho_Q$ be a minimizer of $\Psi$ over the simplex $\blacktriangle_{M-1}$. Then $\Psi(\rho_Q) \le \Psi(\alpha \rho) < \cR_P(f_M)$.

    ($\Leftarrow$) Conversely, assume $\Psi(\rho_Q) < \cR_P(f_M)$. Evaluating the objective $\Psi$ at $\rho_Q$ gives
    \begin{align*}
        \Psi(\rho_Q) &= \cR_P(f_M) + \sum_{j=1}^{M-1} (\rho_Q)_j \left(\Delta_j - \frac{1}{2}d_{jM}^2\right) + \frac{1}{2} \rho_Q^\top G \rho_Q < \cR_P(f_M).
    \end{align*}
    Subtracting $\cR_P(f_M)$ from both sides yields
    \begin{align*}
        \sum_{j=1}^{M-1} (\rho_Q)_j \left(\Delta_j - \frac{1}{2}d_{jM}^2\right) + \frac{1}{2} \rho_Q^\top G \rho_Q < 0.
    \end{align*}
    Since $G$ is a positive semidefinite Gram matrix, the quadratic term is nonnegative ($\frac{1}{2} \rho_Q^\top G \rho_Q \ge 0$). For the entire expression to be strictly negative, the linear term must be strictly negative:
    \begin{align*}
        \sum_{j=1}^{M-1} (\rho_Q)_j \left(\Delta_j - \frac{1}{2}d_{jM}^2\right) < 0.
    \end{align*}
    Rearranging this inequality gives
    \begin{align*}
        \sum_{j=1}^{M-1} (\rho_Q)_j d_{jM}^2 > 2 \sum_{j=1}^{M-1} (\rho_Q)_j \Delta_j.
    \end{align*}
    Since $\rho_Q \in \blacktriangle_{M-1}$, we have found a valid weight vector $\rho$ that satisfies the condition, completing the proof.
\end{proof}

\subsection{Proof of \texorpdfstring{\cref{prop:Q-aggregation-sparsity}}{Proposition \ref{prop:Q-aggregation-sparsity}}}
\label{proof:Q-aggregation-sparsity}

Again, without loss of generality, assume that $\cR_P(f_M)\leq \ldots\leq \cR_P(f_1)$. Recall the notation and reduction to $\blacktriangle_{M-1}$ from \cref{subsec:summary-quantities,subsec:rewriting}.

On $\blacktriangle_{M-1}$, the KKT conditions of the empirical objective, as derived in \cref{eq:Qgeneral:empirical-KKT}, imply the existence of multipliers $\hat\mu_j\geq 0$, $j=1,\ldots,M$, such that
\begin{equation*}
        \begin{aligned}
            (\hat{G}\rhohatQ)_j + \hat{\Delta}_j - \frac{1}{2}\hat{d}_{jM}^2 + \hat{\mu}_M - \hat{\mu}_j &= 0, \quad j=1,\dots,M-1, \\
    \hat{\mu}_j \rhohat_{Q,j} &= 0, \quad j=1,\dots,M-1, \\
    \hat{\mu}_M \Bigg(\sum_{i=1}^{M-1} \rhohat_{Q,i}- 1\Bigg) &= 0.
        \end{aligned}
    \end{equation*}
Combining the conditions gives
\begin{equation*}
    (\hat{G}\rhohatQ)_j + \hat{\Delta}_j - \frac{1}{2}\hat{d}_{jM}^2 >0 \implies \hat{\mu}_j = (\hat{G}\rhohatQ)_j + \hat{\Delta}_j - \frac{1}{2}\hat{d}_{jM}^2 + \hat{\mu}_M>0 \implies \rhohat_{Q,j}=0.
\end{equation*}
By \cref{lem:Qgeneral:convergence}, with probability at least $1-CM^2\exp(-c\frac{\eps^2n}{M^2})$, the following inequalities hold:
\begin{align*}
 |(\hat{G}\rhohatQ)_j + \hat{\Delta}_j - \frac{1}{2}\hat{d}_{jM}^2- ((G\rhohatQ)_j + \Delta_j - \frac{1}{2}d_{jM}^2 )| \le &\eps, \\
\text{and} \quad \|\sqrt{G}( \rhohatQ - \rho_Q)\|^2_2=\ee_{X\sim P_X}\left[ \left( f_{\rhohatQ}(X) - f_{\rho_Q}(X) \right)^2 \right]
\leq &\eps.
\end{align*}
Therefore, on the same event, we have that
\begin{align*}
    (\hat{G}\rhohatQ)_j + \hat{\Delta}_j - \frac{1}{2}\hat{d}_{jM}^2 &\geq (G\rhohatQ)_j + \Delta_j - \frac{1}{2}d_{jM}^2 -\eps \\
    &= (G\rho_Q)_j+ \Delta_j - \frac{1}{2}d_{jM}^2 +(G(\rhohatQ-\rho_Q))_j-\eps \\
    &\geq (G\rho_Q)_j+ \Delta_j - \frac{1}{2}d_{jM}^2-\sqrt{\eps}-\eps
\end{align*}
where the last inequality follows from the same argument as in \cref{eq:Qgeneral:case23}, which gives $|(G(\rhohatQ-\rho_Q))_j|\leq \sqrt{\eps}$. Therefore, if $(G\rho_Q)_j+ \Delta_j - \frac{1}{2}d_{jM}^2>\sqrt{\eps}+\eps$, then $\rhohat_{Q,j}=0$.

Finally, recall that $M$ was chosen arbitrarily as the index of one of the best functions. Fix any $\fstar\in\cFstar$, and recall from \cref{subsec:summary-quantities} that $\Delta_j=\cR_P(f_j)-\cR_P(\fstar)$, $d_{jM}^2=\EE_{X\sim P_X}[(f_j(X)-\fstar(X))^2]$, and $(G\rho_Q)_j=\EE_{X\sim P_X}[(f_j(X)-\fstar(X))(f_{\rho_Q}(X)-\fstar(X))]$.
The polarization identity gives
\begin{equation*}
    (G\rho_Q)_j = \frac{1}{2} \EE_{X\sim P_X}\brk{(f_j(X)-\fstar(X))^2+(f_{\rho_Q}(X)-\fstar(X))^2-(f_{\rho_Q}(X)-f_j(X))^2}.
\end{equation*}
Since the reduction to $\blacktriangle_{M-1}$ identifies the first $M-1$ coordinates with the corresponding full-simplex coordinates, this implies that on $\triangle_M$ we have that on the event from above,
\begin{align*}
    &\cR_P(f_j)-\cR_P(\fstar)+\frac{1}{2}\EE_{X\sim P_X}\brk{(f_{\rho_Q}(X)-\fstar(X))^2-(f_j(X)-f_{\rho_Q}(X))^2}\\
    &= (G\rho_Q)_j+ \Delta_j - \frac{1}{2}d_{jM}^2 >\sqrt{\eps}+\eps
\end{align*}
implies $\rhohat_{Q,j}=0$. This concludes the proof of \cref{prop:Q-aggregation-sparsity}.

\section{Proofs for Countably Infinite Hypothesis Spaces}

\subsection{Proof of \cref{thm:arbitrarily-slow-rates-inf-not-realized}}
\label{proof:arbitrarily-slow-rates-inf-not-realized}

Without loss of generality, let $\NN\subseteq \cX$.
Define, for any $i\in\NN$ and $I\in\{0,1\}^{i}$, the hypothesis
\begin{align*}
 f_I(x)=
 \begin{cases}
    I_x, & \text{if } x\in[i],\\
    0, & \text{otherwise}.
 \end{cases}
\end{align*}
Let the hypothesis space $\cF$ be the collection of all such $f_I$:
\begin{align*}
    \cF=\{f_I:I\in\{0,1\}^i \text{ for some } i\in\NN\} = \bigcup_{i\in\NN}\{f_I:I\in\{0,1\}^i\}.
\end{align*}
This class is countable because it is a countable union of finite sets.

We next show that $\cF$ admits an infinite VCL tree. Recall the formal
definition from \cite{bousquet2021theory}.
\begin{definition}[VCL tree, Definition~1.8 in \cite{bousquet2021theory}]
Let $\cH\subseteq\{0,1\}^{\cX}$ and let
$d\in\NN\cup\{\infty\}$. A VC-Littlestone tree, or \emph{VCL tree}, of depth $d$ for
$\cH$ is a collection
\[
    \Big\{x_u=(x_u^0,\ldots,x_u^k)\in\cX^{k+1}:
    0\le k<d,\
    u\in\prod_{\ell=1}^{k}\{0,1\}^{\ell}
    \Big\},
\]
where the empty product for $k=0$ is interpreted as a singleton, such that for
every $n<d$ and every
\[
    y=(y_1,\ldots,y_{n+1})\in\prod_{\ell=1}^{n+1}\{0,1\}^{\ell},
    \qquad
    y_\ell=(y_\ell^0,\ldots,y_\ell^{\ell-1})\in\crl{0,1}^\ell,
\]
there exists $h\in\cH$ satisfying $h\big(x_{y_{\le k}}^i\big)=y_{k+1}^i$ for all $0\le k\le n$ and $0\le i\le k$,
where $y_{\le k}=(y_1,\ldots,y_k)$ and $y_{\le0}$ is the empty tuple. The class
$\cH$ admits an infinite VCL tree if this holds with $d=\infty$.
\end{definition}

We now construct such a tree explicitly for $\cF$. For $k\in\NN\cup\{0\}$, set
\[
    \mathcal{U}_k=\prod_{\ell=1}^{k}\{0,1\}^{\ell},
\]
where $\mathcal{U}_0$ consists of the empty tuple. The set
\[
    \mathcal{T}=\{(k,u,r):k\in\NN\cup\{0\},\ u\in\mathcal{U}_k,\ r\in\{0,\ldots,k\}\}
\]
is countable because each $\mathcal{U}_k$ is finite. Fix an arbitrary injective map $\psi:\mathcal{T}\to\NN$.
For every $k\in\NN\cup\{0\}$ and every $u\in\mathcal{U}_k$, define
\[
    x_u=(x_u^0,\ldots,x_u^k)\in\NN^{k+1},
    \qquad
    x_u^r=\psi(k,u,r),\quad r=0,\ldots,k.
\]
This gives a candidate infinite VCL tree. We verify the realization condition.
Fix $n\in\NN\cup\{0\}$ and fix
\[
    y=(y_1,\ldots,y_{n+1})\in\prod_{\ell=1}^{n+1}\{0,1\}^{\ell},
    \qquad
    y_\ell=(y_\ell^0,\ldots,y_\ell^{\ell-1})\in\crl{0,1}^\ell.
\]
Recalling that $y_{\leq k}=(y_1^0,(y_2^0,y_2^1),\ldots,(y_k^0,\ldots,y_k^{k-1}))$, define the set
\[
    S_y=\{x_{y_{\le k}}^r\in \NN:0\le k\le n,\ 0\le r\le k\}\subset\NN.
\]
The map $\psi$ is injective, so the points in $S_y$ are all distinct, and the
following definition is well-defined. Let $m_y=\max S_y$, and choose
$I^{(y)}\in\{0,1\}^{m_y}$ such that
\[
    I^{(y)}_{x_{y_{\le k}}^r}=y_{k+1}^r
    \qquad
    \text{for all }0\le k\le n\text{ and }0\le r\le k.
\]
This is possible because the points $x_{y_{\le k}}^r$ are distinct; on coordinates in $[m_y]\setminus S_y$, set the values arbitrarily, say to zero. The corresponding function
$f_{I^{(y)}}$ belongs to $\cF$ and, by construction,
\[
    f_{I^{(y)}}\big(x_{y_{\le k}}^r\big)
    =
    y_{k+1}^r
    \qquad
    \text{for all }0\le k\le n\text{ and }0\le r\le k.
\]
This is precisely the VCL realization condition for the finite path
$y=(y_1,\ldots,y_{n+1})$. Since $n$ and $y$ were arbitrary, the collection
$\{x_u:u\in\mathcal{U}_k,\ k\in\NN\cup\{0\}\}$ is an infinite VCL tree
for $\cF$.

We now use the existence of the VCL tree to prove the result.
We first observe that if $R(n_0)=0$ for some $n_0\in\NN$, then $R(n)=0$ for all $n\ge n_0$, since $R$ is decreasing and nonnegative. Since the right-hand side is then zero for infinitely many $n$, the claim of the theorem follows once we exhibit a distribution for which the infimum
over $\cF$ is zero, as the excess risk then is bounded below by zero, and the infimum is not attained for this distribution. To produce such a distribution, let $P$ be the distribution on $\NN\times\{0,1\}$ given by
$P(\{(j,1)\})=2^{-j}$ for $j\in\NN$. Then, for every $f_I\in\cF$ with
$I\in\{0,1\}^i$, the function $f_I$ is zero outside $[i]$, while $P$ assigns
positive mass to $(j,1)$ for every $j>i$, so
$\cR_P(f_I)>0$. On the other hand, if $I^{(m)}=(1,\ldots,1)\in\{0,1\}^m$,
then $\cR_P(f_{I^{(m)}})\le P(\{(m+1,1),(m+2,1),\ldots\})\to0$. Thus
$\inf_{f\in\cF}\cR_P(f)=0$, but no element of $\cF$ attains this infimum. We
may assume from now on that $R(n)>0$ for all $n\in\NN$.

Fix a learning algorithm $\cA$ for the squared loss. Given a sample $\rS$, define the thresholded classifier
\begin{align*}
    \bar f_{\rS}(x)=\ind\{\cA(\rS)(x)\ge 1/2\}.
\end{align*}
Since $\cF$ has an infinite VCL tree, the proof of \cite[Theorem~5.11]{bousquet2021theory}
implies that there is a universal constant $c_0>0$ such that, for every rate function $R$ with $R(n)\to0$, there exists a binary distribution $P$ on
$\NN\times\{0,1\}$ satisfying
\begin{align}\label{eq:infimum-zero-not-attained}
    \forall h\in\cF,\ \PP_{(X,Y)\sim P}\big(h(X)\ne Y\big)>0 \quad \text{and} \quad \inf_{h\in\cF} \PP_{(X,Y)\sim P}\big(h(X)\ne Y\big)=0
\end{align}
and
\begin{align}\label{eq:lower-bound-classification-error}
    \EE_{\rS\sim P^n}
    \brk{
    \PP_{(X,Y)\sim P}\big(\bar f_{\rS}(X)\ne Y\big)
    }
    \ge c_0 R(n)
\end{align}
for infinitely many $n\in\NN$. We show below, by inspecting the proof \cite[Theorem~5.11]{bousquet2021theory}, that the
distribution can be chosen to have these properties.
Assume this for the moment.
For such a binary distribution, the squared loss of a binary-valued $h\in\cF$
equals its classification error, so
\[
    \inf_{h\in\cF}\cR_P(h)=0.
\]
For any $a\in[0,1]$ and $y\in\{0,1\}$,
\[
    (a-y)^2
    \ge
    \frac14 \ind\{\ind\{a\ge1/2\}\ne y\}.
\]
Therefore, for the same $P$,
\begin{align*}
    \EE_{\rS\sim P^n}\brk{\excessRisk(\cA(\rS))}
    &=
    \EE_{\rS\sim P^n}\brk{
    \EE_{(X,Y)\sim P}\big[(\cA(\rS)(X)-Y)^2\big]-\inf_{h\in\cF}\cR_P(h)
    }\\
    &=
    \EE_{\rS\sim P^n}\brk{
    \EE_{(X,Y)\sim P}\big[(\cA(\rS)(X)-Y)^2\big]
    }\\
    &\ge\frac14\EE_{\rS\sim P^n}
    \brk{
    \PP_{(X,Y)\sim P}\big(\bar f_{\rS}(X)\ne Y\big)
    }\\
    &\ge\frac{c_0}{4}R(n)
\end{align*}
for infinitely many $n\in\NN$. Setting $c=c_0/4$ gives the desired lower bound.

It remains to verify that the distribution supplied
by the construction in the proof of Theorem~5.11 in \cite{bousquet2021theory} satisfies \cref{eq:infimum-zero-not-attained,eq:lower-bound-classification-error}. Let $ y=(y_{1},y_{2},\ldots) $, where the random vectors $ y_{i}\in\{0,1\}^{i} $ are independent and have i.i.d.\ Bernoulli coordinates with parameter $1/2$. Define the random measure
\begin{align*}
P_{y}\big((x^r_{y_{\le k-1}},y_k^r)\big)=\frac{p_k}{k},\qquad \text{for } 0\leq r\leq k-1 \text{ and } k\geq 1,
\end{align*}
where $(p_k)_{k\ge1}$ is a probability distribution on $\NN$ that depends only on $R$ and has infinite support contained in $\{n\in\NN: R(n)>0\}$; this is the case considered here.
By the construction in the proof of Theorem~5.11 in \cite{bousquet2021theory}, there exist a universal constant $c_0'>0$, two strictly increasing sequences
$(n_i)_{i=1}^{\infty}$ and $(k_{i})_{i=1}^{\infty}$, and measurable events $E_i^y
=
\left\{
(X,Y)=
\left(x^r_{y_{\le k_i-1}},y_{k_i}^r\right)
\text{ for some }0\leq r<k_i
\right\}$
such that
\[
\EE_{y}\left[\limsup_{i\to\infty}
\frac{1}{R(n_i)}
\EE_{\rS\sim P_y^{n_i}}\left[
\PP_{(X,Y)\sim P_{y}}(\bar f_{\rS}(X)\ne Y, E_i^y)
\right]\right]\ge c_0',
\]
and, for every realization $y$,
\[
\limsup_{i\to\infty}
\frac{1}{R(n_i)}
\EE_{\rS\sim P_y^{n_i}}\left[
\PP_{(X,Y)\sim P_{y}}(\bar f_{\rS}(X)\ne Y, E_i^y)
\right]
\le C,
\]
for some universal constant $C>0$.
For the random sequence $y$, define the event
\[
E=\left\{
\sup\left\{
x_{y_{\le k-1}}^r:
p_k>0,\ 0\le r\le k-1,\ y_k^r=1
\right\}
=\infty
\right\}.
\]
This event means that $y$ assigns label one to infinitely many points with strictly positive mass under the $X$-marginal of $P_y$. We claim that $\PP(E)=1$. For every $k$ with $p_k>0$, the probability that level $k$
contains no label equal to one is $2^{-k}$. Denote this event by $E_k=\{y:  \forall 0\le r\le k-1,  y_k^r=0\}$. Since the support of $(p_k)$
is infinite and $\sum_{k\ge 1, p_k>0} \PP(E_k) \leq \sum_{k=1}^\infty 2^{-k}<\infty$, the Borel--Cantelli lemma implies that $\p(\cap_{n\ge 1} \cup_{k\ge n,p_k>0} E_k) = 0 $, which is equivalent to the claim that
all but finitely many supported levels contain a one. Since there are infinitely many supported levels, $\PP(E)=1$.
Since the random variable inside the following expectation is bounded by $C$ and
$\PP(E)=1$, multiplying it by $\ind\{E\}$ does not change its
expectation:
\begin{align*}
c_0'
&\le
\EE_{y}\left[\limsup_{i\to\infty} \frac{1}{R(n_i)}
\EE_{\rS\sim P_y^{n_i}}
\brk{\PP_{(X,Y)\sim P_{y}}(\bar f_{\rS}(X)\ne Y, E_i^y)}
\right]
\\
&=
\EE_{y}\left[\limsup_{i\to\infty} \frac{1}{R(n_i)}
\EE_{\rS\sim P_y^{n_i}}
\brk{\PP_{(X,Y)\sim P_{y}}(\bar f_{\rS}(X)\ne Y, E_i^y)}
\ind\{E\}
\right]
\\
&\le
\EE_{y}\left[\limsup_{i\to\infty} \frac{1}{R(n_i)}
\EE_{\rS\sim P_y^{n_i}}
\brk{\PP_{(X,Y)\sim P_{y}}(\bar f_{\rS}(X)\ne Y)}
\ind\{E\}
\right].
\end{align*}
Therefore, there exists a realization $y$ for which
\begin{align*}
    \limsup_{i\rightarrow \infty} \frac{1}{R(n_{i})}\EE_{\rS\sim P_{y}^{n_{i}}}
    \brk{
    \PP_{(X,Y)\sim P_{y}}\big(\bar f_{\rS}(X)\ne Y\big)
    }\ind\{E\}
    \geq c_0'.
\end{align*}
Choose such a realization $y$, set $P=P_{y}$, and let $P_X$ denote its
$X$-marginal. Since the displayed quantity is positive and contains the factor
$\ind\{E\}$, this realization belongs to $E$. The same display implies the classification lower bound in \cref{eq:lower-bound-classification-error}, with
any constant $c_0<c_0'$; for definiteness, take $c_0=c_0'/2$. Because $\psi$ is injective, each point in the support of $P_X$ is associated with a unique label under $P$. Thus, under $P$, there are
infinitely many points of strictly positive $P_X$-mass whose label is one. Let $f_I\in\cF$ be
arbitrary, with $I\in\{0,1\}^{i}$. Because $y\in E$, there exists
$j>i$ such that $P_X(\{j\})>0$ and the label corresponding to $j$ is one. Since $j\notin[i]$,
$f_I(j)=0$, by the definition of the class, and the error at this positive-mass point gives
\[
    \cR_P(f_I)
    \ge
    P_X(\{j\})>0.
\]
Thus, every hypothesis in $\cF$ has strictly positive risk, showing the first claim in \cref{eq:infimum-zero-not-attained}.
For each $m\in\NN$, choose $I^{(m)}\in\{0,1\}^m$ to agree
with the unique label assigned by $P$ to each support point in $[m]$, and define it
arbitrarily at points of $[m]$ outside the support. The function
$f_{I^{(m)}}$ can make a mistake only on $\{m+1,m+2,\ldots\}$, so
\[
    \cR_P(f_{I^{(m)}})
    \le
    P_X(\{m+1,m+2,\ldots\})
    \to 0 \quad \text{as } m\to\infty.
\]
Thus $\inf_{f\in\cF}\cR_P(f)=0$, showing the second claim in \cref{eq:infimum-zero-not-attained}.

\subsection{Proof of \cref{thm:nearly-exponential-rates-lower-bound}}
\label{proof:nearly-exponential-rates-lower-bound}
The main ingredient of the proof is that we can construct probability distributions of the following kind.
\begin{lemma}
\label{lem:direct-Q}
Fix $\eps\in(0,1/2]$. Let $y\in[\eps,1-\eps]$, let
$a_1,\dots,a_m\in[0,1]$, and let $\alpha_1,\dots,\alpha_m\ge 0$ satisfy $\sum_{i=1}^m \alpha_i \le \eps$.
Then there exists a probability measure $Q$ on $[0,1]$, supported on $\{a_1,\dots,a_m,0,1\}$, such that $\EE_{z\sim Q}\brk{z}=y$ and $Q(\{a_i\})\ge \alpha_i$ for all $i=1,\dots,m$.
\end{lemma}
\begin{proof}
Let $\beta:=\sum_{i=1}^m \alpha_i$ and $s:=\sum_{i=1}^m \alpha_i a_i$. Since $a_i\in[0,1]$ for all $i$, we have $0\le s\le \beta\le \eps$.
Set $q_1:=y-s$ and $q_0:=1-\beta-y+s$. Then $0\leq q_1\leq 1$ because $s\le \beta\le \eps\le y\leq 1-\eps$.
Since $ y\leq1-\eps $, we also have $0\leq q_0\leq 1$: indeed, $q_0=1-\beta-y+s\ge 1-\beta-y\ge 1-\eps-(1-\eps)=0$, while $q_0=1-\beta-y+s\leq 1-\beta+s\leq 1 $ since $s\leq \beta$.
Using Dirac measure notation, define
\begin{equation*}
Q:=\sum_{i=1}^m \alpha_i \delta_{a_i}+q_0\delta_0+q_1\delta_1.
\end{equation*}
Its total mass is $Q([0,1])=\beta+q_0+q_1=\beta+(1-\beta-y+s)+(y-s)=1$,
so $Q$ is a probability measure on $[0,1]$, supported on $\{a_1,\dots,a_m,0,1\}$.
For each $i$, we have $Q(\{a_i\})\ge \alpha_i$. Its mean is
\begin{equation*}
\int z\,dQ(z)=\sum_{i=1}^m \alpha_i a_i+q_1=s+(y-s)=y.
\end{equation*}
This proves that $Q$ has the required properties.
\end{proof}

\begin{proof}[Proof of \cref{thm:nearly-exponential-rates-lower-bound}]
Fix any learning algorithm $\cA$, and for each $j\in\NN$ let
$
p_j:=2^{-j}.
$
We begin by constructing distributions $P_0$ and $P_i$, $i\in\NN$, using \cref{lem:direct-Q}.
For each $i\in\NN$, let
$
\alpha_i:=\eps\,2^{-i-2},
$
so that
$
\sum_{i\ge 1}\alpha_i=\eps\sum_{i\ge 1}2^{-i-2}=\eps/4\le \eps.
$
For each $j\in\NN$, apply \cref{lem:direct-Q} with
$
m=j,\ y=y_j,\ a_i=f_i(x_j).
$
This yields a probability measure $Q_j$ on $[0,1]$, supported on
$
\{f_1(x_j),\dots,f_j(x_j),0,1\},
$
such that
$
\int z\,dQ_j(z)=y_j
$
and, for every $i\le j$,
$
Q_j\big(\{f_i(x_j)\}\big)\ge \alpha_i.
$
Define distributions $P_0$ and $P_i$, $i\in\NN$, on $\mathcal{X}\times[0,1]$ by
$
P_0(X=x_j)=P_i(X=x_j)=p_j,
$
and
$$
P_0:\quad Y\,|\,X=x_j \sim Q_j \qquad \text{and} \qquad  P_i,\ i>0:\quad Y\,|\,X=x_j \sim
\begin{cases}
Q_j, & j<i,\\
\delta_{f_i(x_j)}, & j\ge i.
\end{cases}
$$

We can now identify the Bayes regressors.
Let $g_0:\mathcal{X}\to[0,1]$ be any measurable function such that
$
g_0(x_j)=y_j$ for all $j\in\NN$.
The conditional mean of $Y$ under $P_0$ is $g_0$ on the support points $x_j$.
Hence, $g_0$ is a Bayes regressor for $P_0$, and for every measurable $f$,
$
\cR_{P_0}(f)-\cR_{P_0}(g_0)=\EE_{(X,Y)\sim P_0}\big[(f(X)-g_0(X))^2\big].
$
For each $i\in\NN$, the conditional mean of $Y$ under $P_i$ is exactly $f_i$: if $j<i$, then
$
\EE_{P_i}[Y\mid X=x_j]=\int z\,dQ_j(z)=y_j=f_i(x_j)
$
by \eqref{eq:prefix-reg}, while if $j\ge i$, then $Y=f_i(x_j)$ almost surely under $P_i$.
Hence, $f_i$ is a Bayes regressor for $P_i$.
Moreover, for every $m\in\NN$,
\begin{equation}
\label{eq:inf-under-P0}0\le \cR_{P_0}(f_{m+1})-\cR_{P_0}(g_0)
=
\sum_{j>m} p_j\,(f_{m+1}(x_j)-y_j)^2
\le
\sum_{j>m}p_j
=
2^{-m},
\end{equation}
where we used \eqref{eq:prefix-reg}. Therefore $\inf_{f\in\cF}\cR_{P_0}(f)=\cR_{P_0}(g_0)$.

Let $(X_1,Y^{(0)}_1),(X_2,Y^{(0)}_2),\dots$ be independent $P_0$-distributed random variables.
For each $i\in\NN$ and $t\in\NN$, define
$$
Y_t^{(i)}:=
\begin{cases}
Y_t^{(0)}, & X_t\in\{x_1,\dots,x_{i-1}\},\\
f_i(X_t), & X_t\in\{x_j:j\ge i\}.
\end{cases}
$$
Then $(X_t,Y_t^{(i)})_{t\ge1}$ is a sequence of independent $P_i$-distributed random variables. For $ i\in\NN\cup \{0\} $, let $\fhat^{(i)}=\cA\prn{(X_t,Y_t^{(i)})_{t=1}^n}$.
For $i\in\NN\cup\{0\}$ and $n\in\NN$, define
$$
\varepsilon_n^{(i)}
:=
\EE\!\left[\cR_{P_i}(\fhat^{(i)})\right]-\inf_{f\in\mathcal{F}}\cR_{P_i}(f).
$$
There are two cases.

\paragraph{Case 1: There exists an $i\in\NN\cup\{0\}$ such that $\limsup_{n\to\infty}\varepsilon_n^{(i)}>0$.}
The claimed lower bound is immediate by taking $P=P_i$ and any bounded function
$\psi(n)$ satisfying $\psi(n)>-\log\!\big(\tfrac12\limsup_n \varepsilon_n^{(i)}\big)$.

\paragraph{Case 2: For every $i\in\NN\cup\crl{0}$, it holds $\limsup_{n\to\infty}\varepsilon_n^{(i)}= 0$.}
We prove the lower bound with $P=P_0$.
Fix $n,i\in\NN$, and define the events
$$
A_{n,i}:=
\left\{
\left|\fhat^{(i)}(x_i)-f_i(x_i)\right|<\frac{\gamma}{2}
\right\}.
$$
Since $f_i$ is the Bayes regressor under $P_i$, we have
$
\cR_{P_i}(\fhat^{(i)})-\cR_{P_i}(f_i)
=
\EE_{(X,Y)\sim P_i}\!\left[\big(\fhat^{(i)}(X)-f_i(X)\big)^2\right].
$
On the event $A_{n,i}^c$, the contribution of the single point $x_i$ to the
excess risk is at least $p_i\gamma^2/4$, so Markov's inequality gives
\begin{equation}
\mathbb{P}(A_{n,i}^c)
\le
\frac{4}{\gamma^2 p_i}\,\varepsilon_n^{(i)}.
\label{eq:reg-markov}
\end{equation}
Under $P_0$, since $g_0(x_i)=y_i$ and
$|f_i(x_i)-y_i|\ge \gamma$ by \eqref{eq:gap-reg}, whenever
$
\left|\fhat^{(0)}(x_i)-f_i(x_i)\right|<\frac{\gamma}{2},
$
we must have
$
\left|\fhat^{(0)}(x_i)-g_0(x_i)\right|\ge \frac{\gamma}{2}.
$
Since $\inf_{f\in\mathcal{F}}\cR_{P_0}(f)=\cR_{P_0}(g_0)$, as shown in \cref{eq:inf-under-P0}, it follows that
\begin{equation}
\EE\!\left[\cR_{P_0}(\fhat^{(0)})\right]-\inf_{f\in\mathcal{F}}\cR_{P_0}(f)
\ge
\frac{\gamma^2}{4}\,p_i\,
\mathbb{P}\!\left(
\left|\fhat^{(0)}(x_i)-f_i(x_i)\right|<\frac{\gamma}{2}
\right).
\label{eq:reg-p0}
\end{equation}

We now relate the probabilities in \cref{eq:reg-markov,eq:reg-p0}.
Define
\[
I_{n,i}^{<}:=\{t\le n:X_t\in\{x_1,\dots,x_{i-1}\}\},
\qquad
I_{n,i}^{\geq}:=\{t\le n:X_t\in\{x_j:j\geq i\}\},
\]
and set $\widehat N_{n,i}:=|I_{n,i}^{\geq}|$.
Let $E_{n,i}$ be the event
$$
E_{n,i}:=
\left\{
Y^{(0)}_t=f_i(X_t)\text{ for every }t\in I_{n,i}^{\geq}
\right\}.
$$
Recall that
$
\fhat^{(i)}=\cA\big((X_t,Y_t^{(i)})_{t=1}^n\big)
$
and that, for $ i\in\NN $ and $ t\in I_{n,i}^{\geq} $, we have $ Y_{t}^{(i)}=f_i(X_t) $. Thus, on $E_{n,i}$ we have
$
(X_{\le n},Y^{(0)}_{\le n})=(X_{\le n},Y^{(i)}_{\le n}),
$
which implies
$
\fhat^{(0)}=\fhat^{(i)}.
$
Conditional on $X_{\le n}$, the labels $\{Y_t^{(0)}:t\in I_{n,i}^{\geq}\}$
are independent, and each satisfies
$
\mathbb{P}\!\left(
Y^{(0)}_t=f_i(X_t)\,\middle|\,X_{\le n}
\right)\ge \alpha_i,
$
for $t\in I_{n,i}^{\geq}$, because if $X_t=x_j$ with $j\geq i$, then
$Q_j(\{f_i(x_j)\})\ge \alpha_i$.
Therefore,
\begin{equation}
\mathbb{P}(E_{n,i}\mid X_{\le n})
\ge \alpha_i^{|I_{n,i}^{\geq}|}
= \alpha_i^{\widehat N_{n,i}}.
\label{eq:reg-E}
\end{equation}

Recall that for $ i\in\mathbb{N} $,
$
\fhat^{(i)}=\cA\big((X_t,Y_t^{(i)})_{t=1}^n\big)
$
and that
$
Y_t^{(i)}=Y_t^{(0)}
$
when $X_t\in\{x_1,\dots,x_{i-1}\}$ (or $ t\in I_{n,i}^{<} $), while
$
Y_t^{(i)}=f_i(X_t)
$
when $X_t\in\{x_j:j\geq i\}$ (or $ t\in I_{n,i}^{\geq} $). Conditionally on $X_{\le n}$, the two index sets
$I_{n,i}^{<}$ and $I_{n,i}^{\geq}$ are fixed. Moreover, the sample
$(X_t,Y_t^{(i)})_{t=1}^n$ is determined by $X_{\le n}$, the prefix labels
$\{Y_t^{(0)}:t\in I_{n,i}^{<}\}$, and the deterministic tail values
$\{f_i(X_t):t\in I_{n,i}^{\geq}\}$. Thus $\fhat^{(i)}$, and hence also $A_{n,i}$,
is measurable with respect to $X_{\le n}$ and the prefix labels
$\{Y_t^{(0)}:t\in I_{n,i}^{<}\}$.

On the other hand, $E_{n,i}$ is the event that
$
Y_t^{(0)}=f_i(X_t)
$
for every $t\in I_{n,i}^{\geq}$, so it is measurable with respect to $X_{\le n}$
and the tail labels $\{Y_t^{(0)}:t\in I_{n,i}^{\geq}\}$. Since the labels
$Y_t^{(0)}$ are conditionally independent given $X_{\le n}$, the prefix labels
and tail labels are conditionally independent given $X_{\le n}$. Hence $A_{n,i}$
and $E_{n,i}$ are conditionally independent given $X_{\le n}$. This conditional independence combined with previous observations gives
\begin{align}
\mathbb{P}\!\left(
\left|\fhat^{(0)}(x_i)-f_i(x_i)\right|<\frac{\gamma}{2}
\right)
&\ge
\mathbb{P}(A_{n,i}\cap E_{n,i}) \notag\\
&=
\EE\!\left[\mathbb{P}(A_{n,i}\cap E_{n,i}\mid X_{\le n})\right]
\notag\\
&=
\EE\!\left[\mathbb{P}(A_{n,i}\mid X_{\le n})\mathbb{P}(E_{n,i}\mid X_{\le n})\right]
\notag\\
&\ge
\EE\!\left[\mathbb{P}(A_{n,i}\mid X_{\le n})\alpha_i^{\widehat N_{n,i}}\right]
\notag\\
&\ge
\EE\!\left[\mathbb{P}(A_{n,i}\mid X_{\le n})\mathbf{1}\{\widehat N_{n,i}\le 4p_i n\}\right]
\alpha_i^{4p_i n}
\notag\\
&\ge
\Big(\mathbb{P}(A_{n,i})-\mathbb{P}(\widehat N_{n,i}>4p_i n)\Big)\alpha_i^{4p_i n}\notag
\\
&\ge
\Big(1-\mathbb{P}(A_{n,i}^{c})-\mathbb{P}(\widehat N_{n,i}>4p_i n)\Big)\alpha_i^{4p_i n}
\label{eq:reg-couple}
\end{align}

Now
$
\mathbb{P}(X_t\in\{x_j:j\geq i\})=\sum_{j\geq i}p_j=2p_i,
$
so $\widehat N_{n,i}$ is binomial with mean $\mu=2p_i n$. The multiplicative
Chernoff bound with threshold $2\mu=4p_i n$ gives
\begin{equation}
\mathbb{P}(\widehat N_{n,i}>4p_i n)\le e^{-2p_i n/3}.
\label{eq:reg-chernoff}
\end{equation}
Combining \eqref{eq:reg-markov}, \eqref{eq:reg-couple}, and \eqref{eq:reg-chernoff}, we obtain
$
\mathbb{P}\!\left(
\left|\fhat^{(0)}(x_i)-f_i(x_i)\right|<\frac{\gamma}{2}
\right)
\ge
\left(
1-\frac{4}{\gamma^2 p_i}\varepsilon_n^{(i)}-e^{-2p_i n/3}
\right)\alpha_i^{4p_i n}.
$
Substituting this into \eqref{eq:reg-p0} yields, for every $n,i\in\NN$,
\begin{equation}
\EE\!\left[\cR_{P_0}(\fhat^{(0)})\right]-\inf_{f\in\mathcal{F}}\cR_{P_0}(f)
\ge
\frac{\gamma^2}{4}\,p_i
\left(
1-\frac{4}{\gamma^2 p_i}\varepsilon_n^{(i)}-e^{-2p_i n/3}
\right)\alpha_i^{4p_i n}.
\label{eq:reg-main}
\end{equation}
For each $i\in\NN$, define
\[
n_i:=
\min\left\{
m\in\NN:
\frac{4}{\gamma^2 p_i}\varepsilon_n^{(i)}
+e^{-2p_i n/3}
\le \frac12
\text{ for every } n\ge m
\right\}.
\]
This is well-defined because $\varepsilon_n^{(i)}\to0$ and $e^{-2p_i n/3}\to0$
for each fixed $i$.
For $n\ge n_1$, the following is well-defined:
$
i_n:=\max\{i\in\NN: n\ge \max\{n_i,i^2\}\}.
$
Furthermore, $i_n\to_{n \to \infty }\infty$, and $i_n\le \sqrt n$.
Applying \eqref{eq:reg-main} with $n\ge n_1$ and $i=i_n$, gives:
$$
\EE\!\left[\cR_{P_0}(\fhat^{(0)})\right]-\inf_{f\in\mathcal{F}}\cR_{P_0}(f)
\ge
\frac{\gamma^2}{8}\,p_{i_n}\alpha_{i_n}^{4p_{i_n}n}.
$$
This means we can choose the function
$$
\psi(n):=
\log\!\left(\frac{8}{\gamma^2}\right)
+\log\!\left(\frac{1}{p_{i_n}}\right)
+4p_{i_n}n\log\!\left(\frac{1}{\alpha_{i_n}}\right).
$$
Then the right-hand side above is exactly $e^{-\psi(n)}$.
Finally,
$
\log\!\left(\frac{1}{p_{i_n}}\right)=i_n\log 2=o(n),
$
because $i_n\le \sqrt n$, and also
$
4p_{i_n}n\log\!\left(\frac{1}{\alpha_{i_n}}\right)
=
4\cdot 2^{-i_n}n\Big((i_n+2)\log 2+\log(1/\eps)\Big)
=o(n),
$
since $i_n\to\infty$ implies $4\cdot 2^{-i_n}\Big((i_n+2)\log 2+\log(1/\eps)\Big)\to 0$.
Therefore $\psi(n)=o(n)$, and
$
\EE\!\left[\cR_{P_0}(\fhat^{(0)})\right]-\inf_{f\in\mathcal{F}}\cR_{P_0}(f)
\ge e^{-\psi(n)}
$
for all sufficiently large $n$, hence in particular for infinitely many $n$.
This completes the proof.
\end{proof}

\subsection{Proof of \cref{thm:arbitrarily-slow-rates-convex-finite-combination}}
\label{proof:arbitrarily-slow-rates-convex-finite-combination}

To prove the theorem, we use the following lemma.
\begin{lemma}[Lemma 5.12 in \cite{bousquet2021theory}]\label{lem:sequencebousquet}
Let $R(t)\to 0$ be any function. Then there exist probabilities
$p_1,p_2,\ldots \ge 0$ such that $\sum_{i\ge 1} p_i = 1$, two increasing sequences
of integers $(n_j)_{j\ge 1}$ and $(i_j)_{j\ge 1}$, and a constant
$\tfrac12 \le c_0 \le 1$ such that the following hold for all $j>1$:
\begin{enumerate}[label=(\alph*)]
\item $\sum_{i>i_j} p_i \le \frac{1}{n_j}$.
\item $n_j\,p_{i_j} \le i_j$.
\item $p_{i_j} = c_0\,R(n_j)$.
\end{enumerate}
\end{lemma}

\begin{proof}[Proof of \cref{thm:arbitrarily-slow-rates-convex-finite-combination}]
Fix a rate function $R:\NN\to(0,\infty)$ with $R(n)\to0$, and apply
\cref{lem:sequencebousquet} to obtain $p_1,p_2,\ldots$, increasing sequences
$(n_\ell)_{\ell\ge1}$ and $(i_\ell)_{\ell\ge1}$, and a constant
$c_0\in[1/2,1]$.
Let $\cA$ be a learning rule satisfying one of the two conditions in the
theorem statement. We use the convex construction whenever the convex
combination condition is available; in that case, set
\[
    b_i:=2
    \qquad\text{for every }i\in\NN.
\]
Otherwise, $\cA$ satisfies the finite-combination condition. Let
$k:\NN\to\NN$ be the corresponding sample-size-dependent width, and define
\[
    b_{i_\ell}:=2k(n_\ell)
    \qquad\text{for every }\ell\in\NN,
    \qquad
    b_i:=2
    \quad\text{if } i\notin\{i_\ell:\ell\in\NN\}.
\]
This is well-defined because the sequence $(i_\ell)_{\ell\ge1}$ is strictly
increasing. The proof below is identical in the two cases except for the local
averaging argument in \cref{eq:lowerboundtwo2}; there, in the finite case, the
only width used at sample size $n_\ell$ is $k(n_\ell)$.

We construct a class $\cF$ on the input space $\cX=\mathbb{N}^{3}$. For every
$m\in\mathbb{N}$ and every
$I=(I_1,\ldots,I_m)$ with $I_i\in[b_i]^{2i}$ for $i\in[m]$, define
\begin{align*}
    f_I((x_1,x_2,x_3))
    &=
    \begin{cases}
        0
        & \text{if } x_1\in[m],\ x_2\in[2x_1],
        \text{ and } x_3=I_{x_1,x_2},\\
        \gamma_I
        & \text{otherwise},
    \end{cases}
\end{align*}
where the values $\gamma_I\in[1/2,1]$ are chosen to be distinct. Such values
exist because the set of all valid finite indices $I$ given by
\[
    \bigcup_{m=1}^{\infty}
    \prod_{i=1}^{m}[b_i]^{2i},
\]
is countable. The latter follows since it is a countable union of finite sets. Let
\[
    \cF:=
    \left\{
        f_I:
        I\in
        \bigcup_{m=1}^{\infty}
        \prod_{i=1}^{m}[b_i]^{2i}
    \right\}.
\]
In the convex-combination case, $b_i=2$ for every $i$, so this class is fixed
independently of $R$. In the finite-combination case, the class may depend on
$R$ and $\cA$ through the sequence $(b_i)$, which is determined by $k$ and the increasing sequences $(n_\ell)$ and $(i_\ell)$.
For $n\in\NN$, write
\[
    \mathrm{Real}_{n}(\cF):=
    \left\{((x_1,y_1),\ldots,(x_n,y_n))\in(\cX\times \cY)^n:
    \exists f\in\cF\ \forall r\in[n],\ f(x_r)=y_r\right\}.
\]
Define the set of distributions
\[
    \cP_{\cF}:=
    \left\{P\text{ over }\cX\times[0,1]:
    \inf_{f\in\cF}\EE_{(X,Y)\sim P}\brk{(f(X)-Y)^2}=0
    \text{ and } \forall n\in\NN,\ P^n(\mathrm{Real}_{n}(\cF))=1
    \right\}.
\]
We will construct a distribution $P$ such that the excess risk of
$\cA$ is at least a universal constant times $R(n)$ for infinitely many
$n$.
To this end, define, for any $z=(z_1,z_2,\ldots)$ with
$z_i\in[b_i]^{2i}$, a distribution $P_z$ over $\cX\times[0,1]$ by
\begin{align*}
    \PP_{(X,Y)\sim P_z}(X=(i,s,z_{i,s}),Y=0)
    =
    \frac{p_i}{2i},
    \qquad i\in\NN,\ s\in[2i].
\end{align*}
The distribution $P_z$ assigns no mass elsewhere.
By the definition of $P_z$, for every $ m\in\mathbb{N}$, the hypothesis
$f_{(z_1,\ldots,z_m)}\in\cF$ satisfies
\begin{align*}
    \EE_{(X,Y)\sim P_z}
    \brk{
        \left(f_{(z_1,\ldots,z_m)}(X)-Y\right)^2
    }
    &=\sum_{i=1}^{\infty}\sum_{s=1}^{2i}\frac{p_i}{2i}\left(f_{(z_1,\ldots,z_m)}((i,s,z_{i,s}))-0\right)^2
    \\
    &= \sum_{i=1}^{\infty}\sum_{s=1}^{2i} \frac{p_i}{2i} \gamma_{(z_1,\ldots,z_m)}^{2}(1-\ind\{(z_{i,s}  =z_{i,s})\wedge(i\in[m])\wedge(s\in[2i])\})
    \leq
    \sum_{i>m}p_i,
\end{align*}
which goes to zero as $m\to\infty$, since the sequence $p_i$ forms a probability
distribution. This shows that
\[
    \inf_{f\in\cF}
    \EE_{(X,Y)\sim P_z}
    \brk{(f(X)-Y)^2}=0.
\]
We also check the exact finite-sample realizability condition. Fix $n\in\NN$ and draw
\[
    S=((x_1,0),\ldots,(x_n,0))\sim P_z^n.
\]
Almost surely, for every $r\in[n]$, there exist $a_r\in\NN$ and $b_r\in[2a_r]$ such that
\[
    x_r=(a_r,b_r,z_{a_r,b_r}).
\]
Choose $m\geq\max_{r\in[n]}a_r$. For this choice, $f_{(z_1,\ldots,z_m)}\in\cF$ and
\[
    f_{(z_1,\ldots,z_m)}(x_r)=0
    \qquad\text{for every }r\in[n].
\]
Thus, the sample $S$ belongs to $\mathrm{Real}_{n}(\cF)$ almost surely. Since $n$ was arbitrary, $P_z^n(\mathrm{Real}_{n}(\cF))=1$ for every $n$, so $P_z\in\cP_{\cF}$.
Finally, no single $f_I\in\cF$ achieves zero risk under $P_z$. Indeed, if
$I$ has length $m$, choose $\ell$ with $i_\ell>m$; then $P_z$ puts mass
$p_{i_\ell}=c_0R(n_\ell)>0$ on level $i_\ell$, while $f_I$ equals $\gamma_I$
on that entire level. Thus, its risk satisfies
\[
    \EE_{(X,Y)\sim P_z}\brk{(f_I(X)-Y)^2}
    \geq
    p_{i_\ell}\gamma_I^2>0.
\]
Because the infimum risk is zero,
\begin{align}\label{eq:lowerboundtwo1}
 &
     \limsup_{n\to\infty}\left\{\frac{1}{R(n)}\EE_{\rS\sim P_{z}^{n}}\brk{\excessRisk(\cA(\rS))}\right\}
    \nonumber\\
    &=
   \limsup_{n\to\infty} \left\{ \frac{1}{R(n)}  \left(\EE_{\rS\sim P_{z}^{n}}\brk{\EE_{(X,Y)\sim P_{z}}\brk{(\cA(\rS)(X)-Y)^{2}} } - \inf_{f\in\cF}\EE_{(X,Y)\sim P_{z}}\brk{(f(X)-Y)^{2}}\right)\right\}\nonumber
   \\
   &=\limsup_{n\to\infty} \left\{\frac{1}{R(n)}\EE_{\rS\sim P_{z}^{n}}\brk{\EE_{(X,Y)\sim P_{z}}\brk{(\cA(\rS)(X)-Y)^{2}}}\right\},
\end{align}
so it suffices to bound the latter.
We have
\begin{align*}
 \EE_{\rS\sim P_{z}^{n}}\brk{\EE_{(X,Y)\sim P_{z}}\brk{(\cA(\rS)(X)-Y)^{2}}}
 = \EE_{\rS\sim P_{z}^{n}}\brk{\sum_{i=1}^{\infty} \frac{p_{i}}{2i}\sum_{s=1}^{2i} \cA(\rS)( (i,s,z_{i,s}) )^{2}},
\end{align*}
since $\ry=0$ almost surely.
For any $i\in\mathbb{N}$, the right-hand side is bounded below by
\begin{align*}
   \frac{p_i}{2i}\sum_{s=1}^{2i}
   \EE_{\rS\sim P_z^n}
   \left[
       \cA(\rS)((i,s,z_{i,s}))^2
       \ind\{(i,s,z_{i,s})\notin\rS\}
   \right].
\end{align*}
Let $Q$ be the distribution on pairs $(i,s)$ given by
$Q(i,s)=p_i/(2i)$ for $i\in\NN$ and $s\in[2i]$. We may sample
$\rS\sim P_z^n$ by first drawing
$\rX=((X_{1,1},X_{1,2}),\ldots,(X_{n,1},X_{n,2}))\sim Q^n$
and then setting
\[
    \rS
    =
    (((X_{1,1},X_{1,2},z_{X_{1,1},X_{1,2}}),0),\ldots,
    ((X_{n,1},X_{n,2},z_{X_{n,1},X_{n,2}}),0)).
\]
We write this labelled sample as $(\rX,z_{\rX})$. Using this representation, for any
$i,n\in\NN$,
\begin{align*}
 &\EE_{\rS\sim P_z^n}
 \brk{
     \EE_{(X,Y)\sim P_z}
     \brk{
         \left(\cA(\rS)(X)-Y\right)^2
     }
 }\\
 &\geq
 \frac{p_i}{2i}\sum_{s=1}^{2i}
 \EE_{\rX\sim Q^n}
 \brk{
     \cA((\rX,z_{\rX}))((i,s,z_{i,s}))^2
     \ind\{(i,s)\notin\rX\}
 },
\end{align*}
where we used that $(i,s,z_{i,s})\notin\rS$ if and only if $(i,s)\notin\rX$.
In particular, this bound holds for $i=i_\ell$ and $n=n_\ell$. Combining it
with \cref{eq:lowerboundtwo1}, we obtain
\begin{align*}
   &\limsup_{n\to\infty} \left\{ \frac{1}{R(n)}  \left(\EE_{\rS\sim P_{z}^{n}}\brk{\EE_{(X,Y)\sim P_{z}}\brk{(\cA(\rS)(X)-Y)^{2}} }- \inf_{f\in\cF}\EE_{(X,Y)\sim P_{z}}\brk{(f(X)-Y)^{2}}\right)\right\}
   \\
   &\geq
   \limsup_{\ell\to\infty} \left\{ \frac{1}{R(n_{\ell})}  \frac{p_{i_{\ell}}}{2i_{\ell}}\sum_{s=1}^{2i_{\ell}}\EE_{\rX\sim Q^{n_{\ell}}}\brk{ \cA((\rX,z_{\rX}))( (i_{\ell},s,z_{i_{\ell},s}) )^{2}\ind\{ (i_{\ell},s) \not\in \rX \}}\right\}
   \\
   &\geq
   \limsup_{\ell\to\infty} \left\{ \frac{1}{4i_{\ell}}  \sum_{s=1}^{2i_{\ell}}\EE_{\rX\sim Q^{n_{\ell}}}\brk{ \cA((\rX,z_{\rX}))( (i_{\ell},s,z_{i_{\ell},s}) )^{2}\ind\{ (i_{\ell},s) \not\in \rX \}}\right\}.
\end{align*}
The first inequality uses the subsequence $n=n_\ell$, and the last
inequality follows from $p_{i_\ell}=c_0R(n_\ell)$ with $c_0\ge1/2$.
Now let $Z$ be a random sequence with independent coordinates, where each
$Z_i$ has $2i$ independent entries uniformly distributed over $[b_i]$.
Using
\begin{align*}
0\leq  \frac{1}{4i_{\ell}}  \sum_{s=1}^{2i_{\ell}}\EE_{\rX\sim Q^{n_{\ell}}}\brk{ \cA((\rX,Z_{\rX}))( (i_{\ell},s,Z_{i_{\ell},s}) )^{2}\ind\{ (i_{\ell},s) \not\in \rX \}}\leq 1,
\end{align*}
reverse Fatou's lemma, with majorant $1$, gives
\begin{align*}
 &\EE_{Z} \brk{\limsup_{\ell\to\infty} \left\{\frac{1}{4i_{\ell}}  \sum_{s=1}^{2i_{\ell}}\EE_{\rX\sim Q^{n_{\ell}}}\brk{ \cA((\rX,Z_{\rX}))( (i_{\ell},s,Z_{i_{\ell},s}) )^{2}\ind\{ (i_{\ell},s) \not\in \rX \}}\right\} }
 \\
 &\geq \limsup_{\ell\to\infty} \frac{1}{4i_{\ell}}  \sum_{s=1}^{2i_{\ell}} \left\{  \EE_{Z} \brk{\EE_{\rX\sim Q^{n_{\ell}}}\brk{ \cA((\rX,Z_{\rX}))( (i_{\ell},s,Z_{i_{\ell},s}) )^{2}\ind\{ (i_{\ell},s) \not\in \rX \}}} \right\}
 \\
 &=\limsup_{\ell\to\infty} \frac{1}{4i_{\ell}}  \sum_{s=1}^{2i_{\ell}} \left\{  \EE_{\rX\sim Q^{n_{\ell}}}\brk{\EE_{Z}\brk{ \cA((\rX,Z_{\rX}))( (i_{\ell},s,Z_{i_{\ell},s}) )^{2}}\ind\{ (i_{\ell},s) \not\in \rX \}} \right\}
 \\
  &=\limsup_{\ell\to\infty} \frac{1}{4i_{\ell}}  \sum_{s=1}^{2i_{\ell}} \left\{  \EE_{\rX\sim Q^{n_{\ell}}}\brk{\EE_{Z_{\backslash i_{\ell},s}}\brk{ \EE_{Z_{i_{\ell},s}}\brk{\cA((\rX,Z_{\rX}))( (i_{\ell},s,Z_{i_{\ell},s}) )^{2}}}\ind\{ (i_{\ell},s) \not\in \rX \}} \right\},
\end{align*}
where the last equality uses the independence of $Z$ and $\rX$ to change the
order of expectation. The notation $Z_{\backslash i_\ell,s}$ denotes the sequence
$Z$ with the entry $Z_{i_\ell,s}$ removed. We also used that
$\ind\{(i_\ell,s)\notin\rX\}$ depends only on $\rX$, so it may be taken outside the expectation over $Z$. We claim that, for every
realization $(\rx,z_{\backslash i_\ell,s})$ of
$(\rX,Z_{\backslash i_\ell,s})$ with $(i_\ell,s)\notin \rx$ (so $Z_{\rX}$ is
fixed at $z_{\rx}$), the following bound holds under the relevant aggregation
restriction on $\cA$:
\begin{align}\label{eq:lowerboundtwo2}
    \EE_{Z_{i_\ell,s}}
    \brk{
        \cA((\rx,z_\rx))((i_\ell,s,Z_{i_\ell,s}))^2
    }
    \geq
    \frac{1}{16}.
\end{align}

Combining this bound with the preceding inequalities, we obtain
\begin{align*}
    &\EE_{Z} \brk{ \limsup_{n\to\infty} \left\{ \frac{1}{R(n)}  \left(\EE_{\rS\sim P_{Z}^{n}}\brk{\EE_{(X,Y)\sim P_{Z}}\brk{(\cA(\rS)(X)-Y)^{2}} }- \inf_{f\in\cF}\EE_{(X,Y)\sim P_{Z}}\brk{(f(X)-Y)^{2}}\right)\right\}}
    \\
&\geq
 \EE_{Z} \brk{\limsup_{\ell\to\infty} \left\{\frac{1}{4i_{\ell}}  \sum_{s=1}^{2i_{\ell}}\EE_{\rX\sim Q^{n_{\ell}}}\brk{ \cA((\rX,Z_{\rX}))( (i_{\ell},s,Z_{i_{\ell},s}) )^{2}\ind\{ (i_{\ell},s) \not\in \rX \}}\right\} }
 \\
 &\geq \frac{1}{32}
    \limsup_{\ell\to\infty}  \left\{ \frac{1}{2i_{\ell}} \sum_{s=1}^{2i_{\ell}}   \EE_{\rX\sim Q^{n_{\ell}}}\brk{\ind\{ (i_{\ell},s) \not\in \rX \}} \right\}
    \\
    &= \frac{1}{32}
    \limsup_{\ell\to\infty} \frac{1}{2i_{\ell}}  \sum_{s=1}^{2i_{\ell}} \prn{1-\frac{p_{i_{\ell}}}{2i_{\ell}}}^{n_{\ell}}
    \\
    &\geq \frac{1}{32}
    \limsup_{\ell\to\infty}  \prn{1-\frac{n_{\ell}p_{i_{\ell}}}{2i_{\ell}}}\tag{Bernoulli's inequality}
    \\
    &\geq \frac{1}{64}. \tag{by \cref{lem:sequencebousquet}, which gives $n_\ell p_{i_\ell}\le i_\ell$}
\end{align*}
Thus, there exists a sequence $z\in\prod_{i=1}^{\infty}[b_i]^{2i}$ such that
\begin{align*}
       \limsup_{n\to\infty} \left\{ \frac{1}{R(n)}  \left(\EE_{\rS\sim P_{z}^{n}}\brk{\EE_{(X,Y)\sim P_{z}}\brk{(\cA(\rS)(X)-Y)^{2}} }- \inf_{f\in\cF}\EE_{(X,Y)\sim P_{z}}\brk{(f(X)-Y)^{2}}\right)\right\}
         \geq \frac{1}{64}.
\end{align*}
This implies the claimed lower-bound statement for any constant $C<\frac{1}{64}$, for instance $C=\frac{1}{65}$.
It remains to prove \cref{eq:lowerboundtwo2}. Fix $\ell$,
$s\in[2i_\ell]$, and a realization $(\rx,z_{\backslash i_\ell,s})$ with
$(i_\ell,s)\notin \rx$. In this case, $Z_{\rx}$ is fixed at $z_\rx$, so
$\cA((\rx,z_\rx))$ is fixed. In particular, the coefficients in the convex
combination, or the functions in the finite combination at sample size
$n_\ell$, are fixed independently of $Z_{i_\ell,s}$.

\paragraph{Convex combination.} In this case, $b_{i_\ell}=2$. For
$z'\in[b_{i_\ell}]$,
\[
    \cA((\rx,z_\rx))((i_\ell,s,z'))
    =
    \sum_{f\in\cF}\alpha_f f((i_\ell,s,z'))
\]
for some fixed $\alpha_f\in[0,1]$, where $\sum_{f\in\cF}\alpha_f=1$. For
$z'\in[b_{i_\ell}]$, write
\[
    \beta_{z'}
    :=
    \sum_{f\in\cF:f((i_\ell,s,z'))=0}\alpha_f.
\]
Since each $f\in\cF$ has at most one value $z'\in[b_{i_\ell}]$ such that
$f((i_\ell,s,z'))=0$, we have $\sum_{z'\in[b_{i_\ell}]}\beta_{z'}\le1$.
Moreover, all non-zero values of hypotheses in $\cF$ are at least $1/2$, so
\[
    \cA((\rx,z_\rx))((i_\ell,s,z'))
    \geq
    \frac{1}{2}(1-\beta_{z'}).
\]
Therefore, by Jensen's inequality, and $ Z_{i_\ell,s} $ being uniform over $ [b_{i_{\ell}}] $
\begin{align*}
     \EE_{Z_{i_\ell,s}}
     \brk{
        \cA((\rx,z_\rx))((i_\ell,s,Z_{i_\ell,s}))^2
     }
     &\geq
     \frac{1}{b_{i_\ell}}\sum_{z'\in[b_{i_\ell}]}\frac{1}{4}(1-\beta_{z'})^2\\
     &\geq
     \frac{1}{4}\left(1-\frac{1}{b_{i_\ell}}\sum_{z'\in[b_{i_\ell}]}\beta_{z'}\right)^2
     \geq
     \frac{1}{16}.
\end{align*}

\paragraph{Finite combination.} Set $K_\ell:=k(n_\ell)$, so
$b_{i_\ell}=2K_\ell$. Since the fixed sample $(\rx,z_\rx)$ has size $n_\ell$,
the finite-combination condition gives functions
$f_1,\ldots,f_{K_\ell}\in\cF$ such that, for every $z'\in[b_{i_\ell}]$,
\begin{align*}
    \cA((\rx,z_\rx))((i_\ell,s,z'))
    \in
    \left[
        \min_{q\in[K_\ell]} f_q((i_\ell,s,z')),
        \max_{q\in[K_\ell]} f_q((i_\ell,s,z'))
    \right].
\end{align*}
Define the set
\[
    B:=\{z'\in[b_{i_\ell}]:\exists q\in[K_\ell]\text{ such that }f_q((i_\ell,s,z'))=0\}.
\]
Since each $f_q\in\cF$ has at most one value $z'\in[b_{i_\ell}]$ such that
$f_q((i_\ell,s,z'))=0$, we have $|B|\le K_\ell$. If $z'\notin B$, then
$f_q((i_\ell,s,z'))\ge1/2$ for every $q\in[K_\ell]$, and the finite-combination condition yields
$\cA((\rx,z_\rx))((i_\ell,s,z'))\ge1/2$. Since $Z_{i_\ell,s}$ is uniform over $[b_{i_\ell}]$,
\begin{align*}
     \EE_{Z_{i_\ell,s}}
     \left[
        \cA((\rx,z_\rx))((i_\ell,s,Z_{i_\ell,s}))^2
     \right]
     \geq
     \frac{2K_\ell-K_\ell}{2K_\ell}\cdot\frac{1}{4}
     =
     \frac{1}{8}.
\end{align*}

\paragraph{Optimal algorithm $\cA^\star$.}
We now construct the learner achieving the upper bound in the theorem
statement. Let $\mathbf{0}:\cX\to[0,1]$ denote the identically zero predictor.
This predictor need not belong to $\cF$.
Given a sample
$S=((x_1,y_1),\ldots,(x_n,y_n))$, the learner $\cA^\star$ is defined as
follows. If all labels in $S$ are zero, it outputs $\mathbf{0}$. Otherwise,
let $r$ be the smallest index such that $y_r\neq0$. If $y_r=\gamma_I$ for a
valid index $I$, the learner outputs $f_I$; if there is no such index, it
outputs $\mathbf{0}$. Since the values $\gamma_I$ are all distinct, this
defines $\cA^\star$ unambiguously on every sample.

Fix $P\in\cP_{\cF}$ and write
$
    q:=\PP_{(X,Y)\sim P}(Y\neq0).
$
We first show that either $q=0$ or $P$ is realized by a single hypothesis
$f_I\in\cF$. Since $P(\mathrm{Real}_1(\cF))=1$, we have
$
    \PP_{(X,Y)\sim P}\bigl(Y\in\{0\}\cup\{\gamma_I:I\text{ is a valid index}\}\bigr)=1.
$
The collection of valid indices is countable. If $q>0$, countability implies that there is an
index $I$ such that
$
   \PP_{(X,Y)\sim P}(Y=\gamma_I)>0.
$
We claim that this $f_I$ realizes $P$. Otherwise,
$\PP_{(X,Y)\sim P}(f_I(X)\neq Y)>0$. For two independent observations
$(X_1,Y_1),(X_2,Y_2)\sim P$, the event
$
    \{Y_1=\gamma_I\}
    \cap
    \{f_I(X_2)\neq Y_2\}
$
would have probability
\[
    \PP_{(X_{1},Y_{1})\sim P}(Y_{1}=\gamma_I)\PP_{(X_{2},Y_{2})\sim P}(f_I(X_{2})\neq Y_{2})=\PP_{(X,Y)\sim P}(Y=\gamma_I)\PP_{(X,Y)\sim P}(f_I(X)\neq Y)>0
\]
by independence. Suppose that $((X_{1},Y_{1}),(X_{2},Y_{2}))$ belongs to the event $ \{Y_1=\gamma_I\} \cap \{f_I(X_2)\neq Y_2\}$ and to
$\mathrm{Real}_2(\cF)$. By definition of $\mathrm{Real}_2(\cF)$, there would exist
$f_{I'}\in\cF$ such that $Y_1=f_{I'}(X_1)$ and $Y_2=f_{I'}(X_2)$.
Since $Y_1=\gamma_I$ on $ \{Y_1=\gamma_I\} \cap \{f_I(X_2)\neq Y_2\}$ and $f_{I'}$ takes values only in $\{0,\gamma_{I'}\}$, the first equality gives $\gamma_I=\gamma_{I'}$. The values $\gamma_{I'}$ are distinct, so $I'=I$.
This gives $Y_2=f_I(X_2)$, contradicting the definition of the event $ \{Y_1=\gamma_I\} \cap \{f_I(X_2)\neq Y_2\}$. Thus, the $ \{Y_1=\gamma_I\} \cap \{f_I(X_2)\neq Y_2\}$ is disjoint from
$\mathrm{Real}_2(\cF)$. Since it has positive $P^2$-probability, this
contradicts $P^2(\mathrm{Real}_2(\cF))=1$. Consequently, we conclude that
$\PP_{(X,Y)\sim P}(f_I(X)=Y)=1$.
In particular, under $P$, every non-zero label equals $\gamma_I$ almost surely.

If $q=0$, then $Y=0$ almost surely, so $\cA^\star$ outputs
$\mathbf{0}$ and has zero excess risk. Now suppose that $q>0$. Whenever the
sample contains a non-zero label, $\cA^\star$ identifies and outputs the
realizing hypothesis $f_I$ and has zero excess risk. The failure probability satisfies
\begin{align*}
    \PP_{\rS\sim P^n}
    \prn{\excessRisk(\cA^\star(\rS))>0}
    &\leq
    (1-q)^n\leq
    \exp(-qn),
\end{align*}
whereas for $q=0$ the excess risk is identically zero. Thus,
$\cA^\star$ achieves zero excess risk with exponential probability on
$\cP_{\cF}$, with exponent constant $ q $ when $ q>0 $; when $q=0$, the bound holds with, for example, exponent constant $1$.

To establish the uniform rate of this algorithm, we consider the two cases $ q\leq \log{(1/\delta)}/n $ and $ q> \log{(1/\delta)}/n $. Since the algorithm outputs either the predictor $ \textbf{0} $ or a function $ f_{I} $ realizing $ P $, its risk is at most $ q=P(\crl{(x,y):y\not=0}) $ in either case. When $ q\leq \log{(1/\delta)}/n$, both its risk and its excess risk are bounded by $q\leq \log{(1/\delta)}/n$. If instead $ q> \log{(1/\delta)}/n $, the probability of positive excess risk is at most $ \exp(-qn) < \delta $. Combining these cases, with probability at least $ 1-\delta $, the excess risk is therefore zero. In both cases, with probability at least $ 1-\delta $, the excess risk is bounded by $ \log{(1/\delta)}/n $.
This completes the proof.

\paragraph{Lower bounds on universal and minimax rates.}
Since $ b_i\ge 2 $ in both cases, the following two hypotheses are always available.
Consider two hypotheses \(f_I\) and \(f_{I'}\) that are both zero at \((1,1,1)\), while at \((1,2,1)\), \(f_I=\gamma_I\) and \(f_{I'}=0\).
The two hypotheses differ by at least \(1/2\) at \((1,2,1)\), since \(\gamma_I\in[1/2,1]\).
Let $ 0\leq q \leq 1 $ be a parameter. Define $ P_{q} $ to assign probability $ q $ to $ ((1,1,1),0) $ and probability $ 1-q$ to $ ((1,2,1),\gamma_{I}) $, and define $ P_{q}' $ to assign probability $ q $ to $ ((1,1,1),0) $ and probability $ 1-q$ to $ ((1,2,1),0) $.
The two distributions are realized by $ f_{I} $ and $ f_{I'} $, respectively, so they belong to $ \cP_{\cF} $.
Under either distribution, the probability that all $n$ samples equal \(((1,1,1),0)\) is \(q^n\).
Now let $ \cA $ be any deterministic algorithm, and define $ a:=\cA(((1,1,1),0)_{i=1}^n)(1,2,1) $, the prediction of the algorithm at $ (1,2,1) $ given a sample consisting of $n$ copies of $ ((1,1,1),0) $. Since $ \gamma_{I}\in[1/2,1] $, either $ |a|=|a-0|\geq 1/4 $ or $ |a-\gamma_{I}|\geq 1/4 $.
In the former case, with probability at least $ q^{n} $ under $ P_{q}' $, the algorithm outputs a hypothesis with excess risk at least $ (1-q)/16 $. In the latter case, with probability at least $ q^{n} $ under $ P_{q} $, the algorithm outputs a hypothesis with excess risk at least $ (1-q)/16 $.
For the universal rate lower bound, choose $ q=1/2 $. For each $ n $  one of the distributions is such that, with probability at least $ 1/2^{n}=\exp(-n\log(2)) $, the algorithm outputs a hypothesis with excess risk at least $ 1/32 $.
Since there are only two distributions, one of them must have the property that the algorithm outputs a hypothesis with excess risk at least $ 1/32 $ with probability at least $\exp(-n\log(2)) $ for infinitely many $ n $, which establishes the claimed universal-rate lower bound.
For the minimax rate, when $ n> \log(1/\delta) $, choose $ q=1-\log(1/\delta)/(2n) $. Since $ \log(1-x)> -2x $ for $ x\in(0,1/2] $, we have $ q^{n}= \exp{(n\log(q) )}>  \exp{(-\log{(1/\delta)} )} =\delta  $.
Thus, with probability at least $ \delta $, the algorithm outputs a hypothesis with excess risk at least $ (1-q)/16\geq \log{(1/\delta)}/(32n)>\log{(1/\delta)}/(64n) $.
This establishes the universal and minimax lower bounds and completes the proof of the theorem.
\end{proof}

\subsection{Proof of \cref{thm:countable-inf-realized-not-uniformly-learnable}}
\label{proof:countable-inf-realized-not-uniformly-learnable}
Let $\cX=\NN$. Write $[m]=\{1,\ldots,m\}$. For $i\in\NN$ and $I=(I_1,\ldots,I_i)\in\{0,1\}^{i}$, define the predictor $f_{I}\in\cM$ by
\begin{align*}
    f_{I}(x) = \begin{cases}
        I_x, & x\in[i],\\
        1, & \text{otherwise}.
    \end{cases}
\end{align*}
Let the function space be
\(
    \cF=\crl{f_{I}:i\in\NN,\ I\in\{0,1\}^{i}}.
\)
This class is countable, since it is a countable union of finite sets.

We first prove the lower bound. Fix a deterministic learning algorithm $\cA$ and a sample size $n\in\NN$. Let $m\in\NN$ and let $z=(z_1,\ldots,z_m)\in\{0,1\}^{m}$. Define $P_{z}$ as the distribution of $(X,Y)$ obtained by drawing $J\sim\uniform{[m]}$ and setting
$
    X=J$ and $  Y=z_J.
$
We have $f_{z}\in\cF$ and $\cR_{P_z}(f_{z})=0$, so $(P_z,\cF)\in\Thetainf$ and $\inf_{f\in\cF}\cR_{P_z}(f)=0$.

For $u=(u_1,\ldots,u_n)\in[m]^n$, define
$
    O(u)=\{u_1,\ldots,u_n\},$ $
    z_u=(z_{u_1},\ldots,z_{u_n}),
$
and define the labeled sample generated by $u$ and $z$ as
$
    S(u,z_u)=\bigl((u_1,z_{u_1}),\ldots,(u_n,z_{u_n})\bigr).
$
If $\rU\sim\uniform{[m]}^{n}$, then $S(\rU,z_{\rU})$ has distribution $P_z^n$. Let $Z$ be uniformly distributed on $\{0,1\}^{m}$, independently of $\rU$. For $B\subseteq[m]$, write $Z_B=(Z_j)_{j\in B}$, in any fixed order. We have
\begin{align*}
    &\ee_{Z}\left[
        \ee_{\rS\sim P_{Z}^{n}}
        \brk{\excessRiskPar{P_Z}{\cF}(\cA(\rS))}
    \right] \\
    &=\ \frac{1}{m}\sum_{j=1}^{m}
    \ee_{\rU\sim\uniform{[m]}^{n}}\left[
        \ee_{Z}\brk{
            \bigl(\cA(S(\rU,Z_{\rU}))(j)-Z_j\bigr)^2
        }
    \right] \tag{because $\cR_{P_Z}(f_Z)=0$} \\
    &=\ \frac{1}{m}\sum_{j=1}^{m}
    \ee_{\rU\sim\uniform{[m]}^{n}}\left[
        \ee_{Z_{O(\rU)}}\left[
            \ee_{Z_{[m]\setminus O(\rU)}}\brk{
                \bigl(\cA(S(\rU,Z_{\rU}))(j)-Z_j\bigr)^2
            }
        \right]
    \right] \\
    &\geq\ \frac{1}{m}
    \ee_{\rU\sim\uniform{[m]}^{n}}\left[
        \sum_{j\in [m]\setminus O(\rU)}
        \ee_{Z_{O(\rU)}}\left[
            \ee_{Z_{[m]\setminus O(\rU)}}\brk{
                \bigl(\cA(S(\rU,Z_{\rU}))(j)-Z_j\bigr)^2
            }
        \right]
    \right]  \\
    &\stackrel{(i)}{=}\ \frac{1}{2m}
    \ee_{\rU\sim\uniform{[m]}^{n}}\left[
        \sum_{j\in [m]\setminus O(\rU)}
        \ee_{Z_{O(\rU)}}\left[
            \bigl(\cA(S(\rU,Z_{\rU}))(j)\bigr)^2
            +\bigl(\cA(S(\rU,Z_{\rU}))(j)-1\bigr)^2
        \right]
    \right]\\
    &\stackrel{(ii)}{\geq}\ \frac{1}{4m}\ee_{\rU\sim\uniform{[m]}^{n}}
    \brk{m-\abs{O(\rU)}} \\
    &\stackrel{(iii)}{\geq}\ \frac{1}{4}\left(1-\frac{n}{m}\right),
\end{align*}
where $(i)$ uses the fact that, on $[m]\setminus O(\rU)$, the prediction of $\cA(S(\rU,Z_\rU))$ is fixed while $Z_j\sim \uniform{\crl{0,1}}$ remains independent; $(ii)$ uses $a^2+(a-1)^2\geq 1/2$ for all $a\in\RR$; and $(iii)$ uses $\abs{O(\rU)}\leq n$.

Since the left-hand side is an average over $z\in\{0,1\}^{m}$, there exists a $z^{(m)}\in\{0,1\}^{m}$ such that
\[
    \ee_{\rS\sim P_{z^{(m)}}^{n}}
    \brk{\excessRiskPar{P_{z^{(m)}}}{\cF}(\cA(\rS))}
    \geq \frac{1}{4}\left(1-\frac{n}{m}\right).
\]
Because $(P_{z^{(m)}},\cF)\in\Thetainf$ and $m$ can be chosen arbitrarily large, it follows that for every $\cA$ and every $n\in\NN$,
\[
    \sup_{P:(P,\cF)\in\Thetainf}
    \ee_{\rS\sim P^n}\brk{\excessRisk(\cA(\rS))}
    \geq \frac{1}{4}.
\]

It remains to note that the constant learning algorithm $\cA_{1/2}$, defined by $\cA_{1/2}(\rS)(x)=1/2$ for all samples $\rS$ and all $x\in\cX$, gives the matching upper bound. Indeed, for every distribution $P$ on $\cX\times[0,1]$ and every function class $\cF\subseteq\cM$,
\[
    \excessRisk(\cA_{1/2}(\rS))
    \leq \ee_{(X,Y)\sim P}\brk{(1/2-Y)^2}
    \leq \frac{1}{4}.
\]
Combining the lower bound with this upper bound proves the claim.

\subsection{Proof of \cref{thm:countable-inf-realized-exponential-rates}}
\label{proof:countable-inf-realized-exponential-rates}

At its core, the proof of this result resembles \cref{thm:pruning-exponential}. In contrast to that proof, however, some care is needed to avoid a union bound over the infinite hypothesis class.

Fix a function $\varphi:\NN\to\RR$ such that $\varphi(n)\to\infty$ and
$\varphi(n)=o(n)$, and write
$
    \psi(n):=\max\{1,\lceil \varphi(n)\rceil\}$ and $\tau_n:=2\sqrt{\frac{\psi(n)}{n}}.
$
For the countable class $\cF$, fix once and for all an enumeration $f_1,f_2,\ldots$ of its elements.

The algorithm $\cA_{\varphi}(\cF,\rS)$ is defined as follows. It scans the first $\psi(n)$ hypotheses, keeping track of an index $\widehat i$. It
starts with $\widehat i=1$, and for $i=2,\ldots,\psi(n)$ replaces $\widehat i$ by
$i$ if for all $j< i$
$
    \tau_{n} \leq \Rhat_{\rS}(f_{j})-\Rhat_{\rS}(f_i) .
$
It returns $\cA_{\varphi}(\cF,\rS)=f_{\widehat i}$.

Now fix $(P,\cF)\in\Thetainf$. Recall the notation
$
    \cFstar=\crl{f\in\cF:\cR_P(f)=\inf_{g\in\cF}\cR_P(g)}
$
and let $i_\star$ be the first index in the chosen enumeration such that
$f_{i_\star}\in\cFstar$. Write $\fstar=f_{i_\star}$ and
$
    \Delta_i:=\cR_P(f_i)-\cR_P(\fstar)\geq 0.
$
Since $i_\star$ is the first optimal index, $\Delta_i>0$ for every
$i<i_\star$. If $i_\star>1$, set
$
    \gamma:=\min_{i<i_\star}\Delta_i>0.
$
If $i_\star=1$, the estimates involving $\gamma$ below are simply omitted.
Because $\psi(n)\to\infty$ and $\tau_n\to0$, for all sufficiently large $n$ we
have $\psi(n)\geq i_\star$ and, when $i_\star>1$, $\tau_n\leq \gamma/2$. From now on, we assume that $n$ is sufficiently large for these conditions to hold.

Consider the events
\begin{align*}
    E_n^{-}
    &=\crl{\forall j<i_\star:\ \tau_n\leq \Rhat_{\rS}(f_j)-\Rhat_{\rS}(\fstar)},\\
    E_n^{+}
    &=\crl{\forall j \text{ with } i_\star< j\leq \psi(n):\
    \tau_n>\Rhat_{\rS}(\fstar)-\Rhat_{\rS}(f_j)}.
\end{align*}
On $E_n^{-}\cap E_n^{+}$, the algorithm returns an element of $\cFstar$, specifically $\fstar$.
Indeed, before the scan reaches $i_\star$, the current candidate is one of
$f_1,\ldots,f_{i_\star-1}$, and $E_n^{-}$ implies that $\fstar$ is empirically better than any of them by a margin of at least $\tau_n$. The scan switches to
$\fstar$ at step $i_\star$. After the scan reaches $i_\star$, the current candidate is $\fstar$, and $E_n^{+}$ ensures that no
$f_j$ replaces it, because the empirical risk of $\fstar$ is never greater than that of any $f_j$ with $i_\star <j\leq \psi(n)$ by at least $ \tau_{n} $. Thus, the final output is $ f_\star $ which is optimal.

It remains to bound the probability of the complement. For any fixed $i$, the
random variable
$
    \xi_i(X,Y):=(f_i(X)-Y)^2-(\fstar(X)-Y)^2
$
takes values in $[-1,1]$ and has expectation $\Delta_i$. Hoeffding's
inequality gives, for all sufficiently large $n$,
\[
    \PP_{\rS\sim P^n}\big((E_n^{-})^c\big)
    \\
    =
    \PP_{\rS\sim P^n}\big(\exists j<i_\star:\ \tau_n> \Rhat_{\rS}(f_j)-\Rhat_{\rS}(\fstar)\big)
    \\
    \leq (i_\star-1)\exp\prn{-\frac{\gamma^2 n}{8}},
\]
where we used the union bound and the assumption that $ n $ is large enough for $\tau_n\leq \gamma/2$. When $i_\star=1$, the right-hand side is understood to be zero.
For every $ j $ with $  i_\star< j\leq \psi(n) $, we have $\Delta_j\geq 0$, so Hoeffding's inequality and a union bound over $j$ give
\begin{align*}
    \PP_{\rS\sim P^n}\big((E_n^{+})^c\big)=\PP_{\rS\sim P^n}\prn{\exists j \text{ with } i_\star< j\leq \psi(n):\
    \tau_n\leq \Rhat_{\rS}(\fstar)-\Rhat_{\rS}(f_j)}
    &\leq \psi(n)\exp\prn{-\frac{n\tau_n^2}{2}} \\
    &=\psi(n)\exp(-2\psi(n))
    \\
    &\leq \exp(-\psi(n)),
\end{align*}
where the last inequality uses $\psi(n)\geq 1$, so $ \ln{(\psi(n))}-2\psi(n) \leq -\psi(n) $.
For all sufficiently large $n$,
\[
    \PP_{\rS\sim P^n}\prn{\excessRisk(\cA_{\varphi}(\cF,\rS))>0}
    \leq
    (i_\star-1)\exp\prn{-\frac{\gamma^2 n}{8}}
    +\exp(-\psi(n))\leq C \exp(-c\psi(n)),
\]
since $\psi(n)=o(n)$ and, for sufficiently large $n$, $ \psi(n)\geq \varphi(n)/2 $. Increasing $C$ and decreasing $c$ if necessary ensures that the bound holds for all $n\in\NN$,
which proves the claim.

\section{Proof of \cref{thm:bob-general}}
\label{proof:bob-general}

The proof first constructs the dictionary and the distributions (\cref{subsec:construction-counter}), then establishes learnability in both worlds (\cref{subsec:learnability-both-worlds-counter}), and finally proves the lower bound for any learner (\cref{subsec:lower-any-counter}).

\subsection{Construction of the Distributions}
\label{subsec:construction-counter}
Let $f_\infty:\cX \to [0,1]$ be any measurable function such that $y_i=f_\infty(x_i)$ for all $i$, with $f_\infty$ defined arbitrarily elsewhere on $\cX$.
Fix an injection $\iota:\NN\times\NN\to\NN$ and write
\(
z_{j,\ell}:=z_{\iota(j,\ell)}.
\)
For each $i,\ell\in\NN$, define
\[
y_i:=f_\infty(x_i),
\qquad
p_i:=2^{-i},
\qquad
\omega_\ell:=\frac{2^{-\ell}}{100}, \quad \text{and}\quad Z_{j,\ell}:=\sum_{i=1}^{j}p_i+\omega_\ell.
\]
We then have $1/2\le Z_{j,\ell}<2$.
For each $i\in\NN$, define the distribution $Q_i$ by
\[
Q_i:=
\begin{cases}
\prn{1-\frac{y_i}{f_i(x_i)}}\delta_0
+\frac{y_i}{f_i(x_i)}\delta_{f_i(x_i)},
& y_i<f_i(x_i),\\[0.5em]
\frac{1-y_i}{1-f_i(x_i)}\delta_{f_i(x_i)}
+\frac{y_i-f_i(x_i)}{1-f_i(x_i)}\delta_1,
& y_i>f_i(x_i),
\end{cases}
\]
where $\delta_y$ denotes the Dirac measure at $y$.
By \cref{eq:bob-tail-gap}, these are the only two cases.
Since $f_i(x_i)\in[0,1]$ and $y_i\in[\varepsilon,1-\varepsilon]$, the denominators in the corresponding cases are positive, and the coefficients are nonnegative and sum to one. So $Q_i$ is a well-defined probability distribution.
In either case,
$
\int z\,dQ_i(z)=y_i.
$
Writing
\(
\lambda_i:=Q_i\prn{\crl{f_i(x_i)}},
\)
we have $\lambda_i\ge\varepsilon$.

We let the family of distributions be
\(
\cP:=\{P_\infty\}\cup\{P_{j,\ell}:j,\ell\in\NN\},
\)
where $P_\infty$ and $P_{j,\ell}$ are defined as follows.
Define the distribution $P_\infty$ by
\[
\PP_{(X,Y)\sim P_\infty}\prn{X=x_i}=p_i
\quad  \text{and} \quad
Y\mid X=x_i\sim Q_i.
\]
For $j,\ell\in\NN$, define the distribution $P_{j,\ell}$ by
\[
\PP_{(X,Y)\sim P_{j,\ell}}\prn{X=x_i}=\frac{p_i}{Z_{j,\ell}},
\quad \text{and} \quad
Y\mid X=x_i\sim Q_i,
\qquad \text{if } i<j,
\]
with the remaining mass distributed as
\[
\PP_{(X,Y)\sim P_{j,\ell}}\prn{X=x_j,Y=f_j(x_j)}=\frac{p_j}{Z_{j,\ell}} \quad \text{and} \quad \PP_{(X,Y)\sim P_{j,\ell}}\prn{X=z_{j,\ell},Y=f_j(z_{j,\ell})}
=
\frac{\omega_\ell}{Z_{j,\ell}}.
\]

\paragraph{Bayes rules.}
Under $P_\infty$, the conditional mean at $x_i$ is, by construction of $Q_i$,
$
\EE_{P_\infty}\brk{Y\mid X=x_i}=y_i=f_\infty(x_i),
$
so $f_\infty$ is Bayes optimal for squared loss.
The bound
$0\le \cR_{P_\infty}(f_k)-\cR_{P_\infty}(f_\infty)
=\sum_{r\ge k}2^{-r}\bigl(f_k(x_r)-y_r\bigr)^2
\le 2^{1-k}$
implies that $\inf_{f\in\cF}\cR_{P_\infty}(f)
=\cR_{P_\infty}(f_\infty).$
Now fix $j,\ell\in\NN$.
Under $P_{j,\ell}$, if $i<j$, then
$
\EE_{P_{j,\ell}}\brk{Y\mid X=x_i}=y_i=f_j(x_i)
$
by the prefix condition \cref{eq:bob-tail-prefix}.
At $x_j$ and $z_{j,\ell}$,
$
\EE_{P_{j,\ell}}\brk{Y\mid X=x_j}=f_j(x_j)$ and $\EE_{P_{j,\ell}}\brk{Y\mid X=z_{j,\ell}}=f_j(z_{j,\ell}).
$
Thus $f_j$ agrees with the conditional mean on the support of $P_{j,\ell}$, so $f_j$ is Bayes optimal for $P_{j,\ell}$.

\subsection{Learnability in Both Worlds}
\label{subsec:learnability-both-worlds-counter}

\paragraph{Algorithm $\Aexp$.}
Let $\Aexp$ output $f_j$ for the largest $j$ such that some sample point has $X=z_{j,\ell}$ for some $\ell\in\NN$, and otherwise output $f_\infty$.
Under $P_\infty$, no point $z_{j,\ell}$ appears, so $\Aexp$ always outputs the Bayes rule $f_\infty$.
Now fix $j,\ell\in\NN$.
Under $P_{j,\ell}$, whenever a sample point has $X=z_{j,\ell}$, the algorithm outputs the Bayes rule $f_j$.
Therefore its failure probability satisfies
\begin{align*}
\PP_{\rS\sim P_{j,\ell}^n}\prn{
    \cR_{P_{j,\ell}}\prn{\Aexp(\rS)}>\inf_{f\in\cF}\cR_{P_{j,\ell}}\prn{f}
}
&\le
\PP_{\rS\sim P_{j,\ell}^n}\prn{\text{no sample point has }X=z_{j,\ell}}\\
&=
\prn{1-\frac{\omega_\ell}{Z_{j,\ell}}}^n\\
&\le
\exp\prn{-\frac{\omega_\ell n}{Z_{j,\ell}}}.
\end{align*}
This establishes the exponential universal rate with constant $ \omega_{\ell}/Z_{j,\ell} $, which depends only on the fixed distribution.

\paragraph{Algorithm $\Amini$.}
Let $\Amini$ output $f_j$ for the largest $j$ such that some sample point has $X=z_{j,\ell}$ for some $\ell\in\NN$, and otherwise output $f_{I_{\max}}$, where $I_{\max}$ is the largest index such that some sample point has $X=x_{I_{\max}}$ (if no such point exists, let $ I_{\max}=1 $).
Fix $n\in\NN$ and $\delta\in(0,1)$.
If $n<4\log(1/\delta)$, the desired upper bound follows for $C_{\mathrm{mini}}\ge 4$, because the excess risk is at most $1$.
We choose $ C_{\mathrm{mini}} $ larger than $ 4 $, so this case is complete.
Now assume that $n\ge 4\log(1/\delta)$ and define
$
A:=\frac{n}{2\log(1/\delta)},
$ and $
m:=\floor{\log_2(A)}\ge 1.
$
These definitions imply
\begin{equation}
\exp\prn{-n2^{-m-1}}\le \delta,
\qquad
2^{-m}<4\frac{\log(1/\delta)}{n}.
\label{eq:bob-general-m}
\end{equation}

Under $P_\infty$, no $z$-point appears.
On the event $\{I_{\max}\ge m\}$, the prefix condition \cref{eq:bob-tail-prefix} implies that $f_{I_{\max}}$ and $f_\infty$ agree at $x_1,\ldots,x_{I_{\max}-1}$.
Since squared-loss excess at a support point is at most $1$,
\[
\cR_{P_\infty}\prn{f_{I_{\max}}}-\cR_{P_\infty}\prn{f_\infty}
\le
\sum_{r\ge I_{\max}}p_r
=2^{1-I_{\max}}
\le 2^{1-m}
\le 8\frac{\log(1/\delta)}{n}.
\]
The complementary event $\{I_{\max}<m\}$ has probability
\[
\PP_{\rS\sim P_\infty^n}\prn{I_{\max}<m}
=
\prn{1-\sum_{r\ge m}p_r}^n
\le
\exp\prn{-n2^{1-m}}
\le \delta.
\]

Now fix $j,\ell\in\NN$ and put $I:=\min\{j,m\}$.
First suppose that $\omega_\ell\le 2^{-m}$.
Let
\[
G_{j,\ell}:=
\crl{\text{some sample point has }X=z_{j,\ell}}
\cup
\crl{I_{\max}\ge I}.
\]
If $I=j$, then on $G_{j,\ell}$ either the certificate, $ z_{j,\ell} $,  is observed and $\Amini$ outputs $f_j$, or else $I_{\max}=j$ and again $\Amini$ outputs $f_j$.
Thus the excess risk is zero on $G_{j,\ell}$.
If $I=m<j$, the only nontrivial case on $G_{j,\ell}$ is $m\le I_{\max}<j$, in which the functions $f_{I_{\max}}$ and $f_j$ can differ only at $x_{I_{\max}},\ldots,x_j$ and at $z_{j,\ell}$.
In this case, the excess risk is bounded by
\begin{align*}
\cR_{P_{j,\ell}}\prn{f_{I_{\max}}}-\cR_{P_{j,\ell}}\prn{f_j}
&\le
\frac{1}{Z_{j,\ell}}\prn{\sum_{r=I_{\max}}^j p_r+\omega_\ell}\\
&\le
2\prn{\sum_{r=m}^\infty 2^{-r}+2^{-m}} \tag{by $ Z_{j,\ell}\geq 1/2 $, $ p_{r}=2^{-r} $, the assumption $ \omega_{\ell}\leq 2^{-m} $, and $ m \le I_{\max} $}\\
&=
6\cdot 2^{-m}
\le
24\frac{\log(1/\delta)}{n} \tag{by \cref{eq:bob-general-m}}.
\end{align*}
The complement of $G_{j,\ell}$ has probability
\begin{align*}
\PP_{\rS\sim P_{j,\ell}^n}\prn{G_{j,\ell}^c}
&=
\prn{1-\frac{\omega_\ell+\sum_{r=I}^j p_r}{Z_{j,\ell}}}^n\\
&\le
\exp\prn{-\frac{n}{Z_{j,\ell}}\sum_{r=I}^j p_r}\\
&\le
\exp\prn{-\frac{n2^{-I}}{2}}
\le
\exp\prn{-\frac{n2^{-m}}{2}}
\le \delta,
\end{align*}
where we used $Z_{j,\ell}<2$ and \cref{eq:bob-general-m}.

It remains to consider $\omega_\ell>2^{-m}$.
If the certificate $z_{j,\ell}$ appears, then $\Amini$ outputs $f_j$.
By \cref{eq:bob-general-m}, the probability that this does not happen is at most
\[
\prn{1-\frac{\omega_\ell}{Z_{j,\ell}}}^n
\le
\exp\prn{-\frac{n\omega_\ell}{Z_{j,\ell}}}
\le
\exp\prn{-\frac{n2^{-m}}{2}}
\le \delta.
\]
These bounds establish the uniform minimax rate with, for example,
$
24\log(1/\delta)/n.
$

\paragraph{Optimality of the exponential scale.}
The exponential probability scale above is unavoidable for zero excess risk.
Fix any learning algorithm $\cA$ and let
\[
E_n^{\mathrm{exp}}:=
\crl{\rS:\text{ every sample point has }X=x_1}.
\]
Given $E_n^{\mathrm{exp}}$, the conditional law of the sample is the same under $P_\infty$ and under $P_{2,1}$, namely $Q_1^n$ on the labels.
Under $P_\infty$,
$
\PP_{\rS\sim P_\infty^n}\prn{E_n^{\mathrm{exp}}}=2^{-n}.
$
Under $P_{2,1}$,
\[
\PP_{\rS\sim P_{2,1}^n}\prn{E_n^{\mathrm{exp}}}
=
\prn{\frac{p_1}{Z_{2,1}}}^n
=
\prn{\frac{1/2}{1/2+1/4+1/200}}^n
=
\prn{\frac{100}{151}}^n
\ge 2^{-n}.
\]
Let $\theta_n$ be the common conditional probability, given $E_n^{\mathrm{exp}}$, of
\[
\abs{\cA(\rS)(x_2)-y_2}<\frac{\gamma}{2}.
\]
If $\theta_n<1/2$, then under $P_\infty$ the learner's prediction at $x_2$ is at distance at least $\gamma/2$ from the Bayes value $y_2$ with conditional probability at least $1/2$, so the zero-excess failure probability under $P_\infty$ is at least $\frac12 2^{-n}$.
If instead $\theta_n\ge 1/2$, then under $P_{2,1}$ the learner's prediction at $x_2$ is within
$\gamma/2$ of $y_2$ with conditional probability at least $1/2$.
The inequality $|f_2(x_2)-y_2|\ge\gamma$ places the learner's prediction at distance at least
$\gamma/2$ from the Bayes value $f_2(x_2)$. That is, for infinitely many $n$, the learner fails to achieve zero excess risk with probability of exponential order under either $P_\infty$ or $P_{2,1}$.
Thus, no learner can improve the exponential scale uniformly over $\cP$.

\paragraph{Optimality of the minimax scale.}
We next show that the rate $\log(1/\delta)/n$ is unavoidable up to constants.
Fix a learning algorithm $\cA$, let $\delta\in(0,1/2)$, put $L:=\log(1/(2\delta))$, and assume $n\ge L$.
Choose
$
t:=\ceil{\log_2\prn{\frac{12n}{L}}}.
$
The definition of $t$ gives
\begin{equation}
\frac{L}{24n}<p_t=2^{-t}\le \frac{L}{12n}.
\label{eq:bob-general-minimax-critical-t}
\end{equation}
Choose $\ell\in\NN$ such that $\omega_\ell\le p_t/2$.
Let
\[
G_t^{\operatorname{adj}}:=
\bigcup_{i=1}^{t-1}\crl{x_i}\times\operatorname{supp}(Q_i),
\qquad
E_{t,n}^{\operatorname{adj}}
:=
\crl{\rS:\text{ every sample point belongs to }G_t^{\operatorname{adj}}}.
\]
The conditional law of $\rS$ given $E_{t,n}^{\operatorname{adj}}$ is the same under $P_{t,\ell}$ and $P_{t+1,\ell}$.
Under $P_{t,\ell}$, the event excludes $x_t$ and $z_{t,\ell}$.
Under $P_{t+1,\ell}$, it excludes $x_t$, $x_{t+1}$, and $z_{t+1,\ell}$.
Since $n\ge L$, we have $t\ge4$, and hence $p_t\le 1/16$.
Furthermore, $ Z_{t,\ell},Z_{t+1,\ell}\geq 1/2 $. Thus, the two excluded one-sample masses are at most
\[
\frac{p_t+\omega_\ell}{Z_{t,\ell}}
\le
3p_t
\le \frac{3}{16},
\qquad
\frac{p_t+p_{t+1}+\omega_\ell}{Z_{t+1,\ell}}
\le
4p_t
\le \frac{1}{4}.
\]
The inequality $\log(1-u)\ge -2u$, which holds for $ u\in(0,1/2] $, can be applied to both masses.By \cref{eq:bob-general-minimax-critical-t},
\[
2n\frac{p_t+\omega_\ell}{Z_{t,\ell}}
\le
6np_t
\le
\frac{L}{2},
\qquad
2n\frac{p_t+p_{t+1}+\omega_\ell}{Z_{t+1,\ell}}
\le
8np_t
\le
\frac{2L}{3}.
\]
Consequently, we obtain
$
\PP_{\rS\sim P_{t,\ell}^n}\prn{E_{t,n}^{\operatorname{adj}}}
\ge
e^{-L/2}
\ge 2\delta,
$ and $
\PP_{\rS\sim P_{t+1,\ell}^n}\prn{E_{t,n}^{\operatorname{adj}}}
\ge
e^{-2L/3}
\ge 2\delta.
$
Let $\theta_{t,n}$ be the common conditional probability, given $E_{t,n}^{\operatorname{adj}}$, of
\[
\abs{\cA(\rS)(x_t)-y_t}<\frac{\gamma}{2}.
\]
If $\theta_{t,n}\ge 1/2$, then under $P_{t,\ell}$ the Bayes value at $ x_{t} $ is $ f_t(x_t) $, and with probability at least $\delta$,
\[
\cR_{P_{t,\ell}}\prn{\cA(\rS)}
-\inf_{f\in\cF}\cR_{P_{t,\ell}}\prn{f}
\ge
\frac{p_t}{Z_{t,\ell}}\frac{\gamma^2}{4}
\ge
\frac{\gamma^2}{192}\frac{L}{n}.
\]
If $\theta_{t,n}<1/2$, then under $P_{t+1,\ell}$ the Bayes value at $x_t$ is $y_t$, and with probability at least $\delta$,
\[
\cR_{P_{t+1,\ell}}\prn{\cA(\rS)}
-\inf_{f\in\cF}\cR_{P_{t+1,\ell}}\prn{f}
\ge
\frac{p_t}{Z_{t+1,\ell}}\frac{\gamma^2}{4}
\ge
\frac{\gamma^2}{192}\frac{L}{n}.
\]
Thus, the minimax rate $\log(1/\delta)/n$ cannot be improved, up to constants.

\subsection{Lower Bound for Any Learner}
\label{subsec:lower-any-counter}
Fix any learning algorithm $\cA$.
For $t\in\NN$, put
\[
G_t:=
\prn{\bigcup_{i=1}^{t-1}\crl{x_i}\times\operatorname{supp}(Q_i)}
\cup
\crl{(x_t,f_t(x_t))},
\]
and define
\[
E_{t,n}:=
\crl{
S=\prn{(X_r,Y_r)}_{r=1}^n:
\forall r\in\crl{1,\ldots,n},\
(X_r,Y_r)\in G_t
}.
\]
Let $\bar P_t$ be the distribution of one draw from $P_{t,\ell}$ conditional on belonging to $G_t$; it does not depend on $\ell$.
Under $\bar P_t$, the point $x_i$ has probability $p_i/s_t$ for $i\le t$, where
$
s_t:=\sum_{i=1}^t p_i=1-2^{-t},
$
the label law is $Q_i$ for $i<t$, and the label at $x_t$ is deterministically $f_t(x_t)$.
Let
$
N_t(S):=\#\{r:S_r=(x_t,f_t(x_t))\}.
$
Under $\bar P_t^{\,n}$, $N_t$ is binomial with parameters $n$ and $p_t/s_t$.
Recall that $\lambda_t=Q_t(\{f_t(x_t)\})\ge\varepsilon$.
For $S\in E_{t,n}$, the likelihood comparison with $P_\infty$ is
$
\PP_{P_\infty^n}\prn{\rS=S}
=
s_t^n\lambda_t^{N_t(S)}
\bar P_t^{\,n}(\{S\}).
$
This holds since $P_\infty((x_i,y))=s_t\bar P_t((x_i,y))$ for $i<t$ and $y\in\operatorname{supp}(Q_i)$, while
$
P_\infty((x_t,f_t(x_t)))
=p_t\lambda_t
=s_t\lambda_t\,\bar P_t((x_t,f_t(x_t))).
$

For the lower bound, choose
$
t_n:=\ceil{\log_2(\varphi(n))},
$ and
$
\ell_n:=\max\crl{1,\ceil{\log_2 n}}.
$
These choices ensure $t_n\ge 2$, $2^{t_n}\ge\varphi(n)$, $2^{t_n}\le 2\varphi(n)$, and $n\omega_{\ell_n}\le 1/100$.
For $P_{t_n,\ell_n}$, the event $E_{t_n,n}$ fails only if the certificate point $z_{t_n,\ell_n}$ appears, giving
\[
\PP_{P_{t_n,\ell_n}^{n}}\prn{E_{t_n,n}}
=
\prn{1-\frac{\omega_{\ell_n}}{Z_{t_n,\ell_n}}}^n
\ge
\exp\prn{-\frac{2n\omega_{\ell_n}}{Z_{t_n,\ell_n}}}
\ge
e^{-1/25}
\ge
\frac{9}{10},
\]
where we used $ Z_{t_{n},\ell_{n}} \ge 1/2$ and $ w_{\ell_{n}}/Z_{t_{n},\ell_{n}}\leq 1/50  $, so $ \log{(1-u)}\geq -2u $, which holds for $ u\in(0,1/2] $, is applicable.
Define
$
\theta_n:=
\bar P_{t_n}^{\,n}\prn{
\abs{\cA(\rS)(x_{t_n})-y_{t_n}}<\frac{\gamma}{2}
}.
$
At least one of the inequalities $\theta_n<1/2$ or $\theta_n\ge 1/2$ holds for infinitely many values of $n$.

\paragraph{Case $\theta_n<1/2$ infinitely often.}
Choose a subsequence $(n_k)_{k=1}^\infty$ such that $\theta_{n_k}<1/2$.
Let
$
C_k:=\{S:\abs{\cA(S)(x_{t_{n_k}})-y_{t_{n_k}}}\ge \frac{\gamma}{2}\}.
$
By case analysis and construction, $\bar P_{t_{n_k}}^{\,n_k}(C_k)>1/2$.
Define the sets
$
H_k:=\{S:N_{t_{n_k}}(S)\le 8n_kp_{t_{n_k}}\}.
$
Since $\EE_{\bar P_{t_{n_k}}^{\,n_k}}[N_{t_{n_k}}]=n_kp_{t_{n_k}}/s_{t_{n_{k}}}\le 2n_kp_{t_{n_k}}$, Markov's inequality gives $\bar P_{t_{n_k}}^{\,n_k}(H_k)\ge 3/4$, so
$
\bar P_{t_{n_k}}^{\,n_k}(C_k\cap H_k)\ge \frac14.
$
The Bayes value at $x_{t_{n_{k}}}$ under $ P_\infty $  is $ y_{t_{n_{k}}} $, so for $ S\in C_k $ the excess risk under $P_\infty$ of the algorithm is at least
\[
p_{t_{n_k}}\frac{\gamma^2}{4}=2^{-t_{n_{k}}}\frac{\gamma^2}{4}
\ge
\frac{\gamma^2}{8\varphi(n_k)}
\ge
\frac{\gamma^2}{16\varphi(n_k)}.
\]
Using the likelihood comparison, the bound $s_{t_{n_k}}^{n_k}\ge\exp(-2n_kp_{t_{n_k}})$, which follows from $ s_{t_{n_{k}}}=1-p_{t_{n_{k}}}$, $ 1/2 \ge p_{t_{n_{k}}} $, and $ \log{(1-u)}\geq -2u $ for $ u\in(0,1/2] $, the bound $ \lambda_{t_{n_{k}}} \ge
 \eps $  and the inequality $N_{t_{n_k}}(S)\le 8n_kp_{t_{n_k}}$ on $H_k$, we obtain
\begin{align*}
\PP_{P_\infty^{n_k}}\prn{
\excessRiskPar{P_\infty}{\cF}(\cA(\rS))
\ge
\frac{\gamma^2}{16\varphi(n_k)}
}
&\ge
\PP_{P_\infty^{n_k}}\prn{C_k\cap H_k\cap E_{t_{n_k},n_k}}\\
&\ge
\frac14
\exp\prn{-2n_kp_{t_{n_k}}}
\varepsilon^{8n_kp_{t_{n_k}}}\\
&=
\frac14
\exp\prn{-\prn{2+8\log\prn{\frac{1}{\varepsilon}}}n_kp_{t_{n_k}}}\\
&\ge
\frac14
\exp\prn{-c_\varepsilon\frac{n_k}{\varphi(n_k)}}.
\end{align*}
This is the first alternative, with $P=P_\infty$.

\paragraph{Case $\theta_n\ge 1/2$ infinitely often.}
Choose a subsequence $(n_k)_{k=1}^\infty$ such that $\theta_{n_k}\ge 1/2$.
Let
$
B_k:=\crl{S:\abs{\cA(S)(x_{t_{n_k}})-y_{t_{n_k}}}<\frac{\gamma}{2}}.
$
Since $P_{t_{n_k},\ell_{n_{k}}}^{n_k}(\cdot\mid E_{t_{n_k},n_k})=\bar P_{t_{n_k}}^{\,n_k}$,
\[
\PP_{P_{t_{n_k},\ell_{n_{k}}}^{n_k}}\prn{B_k}
\ge
\PP_{P_{t_{n_k},\ell_{n_{k}}}^{n_k}}\prn{B_k\mid E_{t_{n_k},n_k}}
\PP_{P_{t_{n_k},\ell_{n_{k}}}^{n_k}}\prn{E_{t_{n_k},n_k}}
\ge
\frac12\cdot\frac{9}{10}
=
\frac{9}{20}.
\]
On $B_k$, \cref{eq:bob-tail-gap} gives
$
\abs{\cA(S)(x_{t_{n_k}})-f_{t_{n_k}}(x_{t_{n_k}})}
\ge
\frac{\gamma}{2}.
$
Under $P_{t_{n_k},\ell_{n_{k}}}$, the Bayes value at $x_{t_{n_k}}$ is $f_{t_{n_k}}(x_{t_{n_k}})$, so
\begin{align*}
\excessRiskPar{P_{t_{n_k},\ell_{n_{k}}}}{\cF}(\cA(S))
\ge
\frac{p_{t_{n_k}}}{Z_{t_{n_k},\ell_{n_{k}}}}\frac{\gamma^2}{4}
\ge
\frac{\gamma^2}{16\varphi(n_k)},
\end{align*}
where we used $Z_{t_{n_k},\ell_{n_{k}}}<2$ and $p_{t_{n_k}}\ge 1/(2\varphi(n_k))$.
The second alternative holds with $P_k:=P_{t_{n_k},\ell_{n_{k}}}$.

\section{Proof of \cref{thm:QBOB}}
\label{proof:QBOB}

We begin by restating the exact result by \cite{lecue2014optimal} on the minimax optimality of $Q$-aggregation, as we will use this later in the proof of \cref{thm:QBOB}.
In particular, \cite[Theorem A]{lecue2014optimal} show that the estimator of \cref{eqn:Q-aggregation-definition} has the following guarantee.
\begin{theorem}[Theorem A in \cite{lecue2014optimal} for squared loss]\label{thm:Qaggregationrestatementlecue2014optimal}
Let $\rhohatQ$ denote any minimizer of \cref{eqn:Q-aggregation-definition} with $\phi\equiv 1$, and ties resolved arbitrarily.
There exists a universal constant $ c > 0 $ such that, for any function class $ \cF\subseteq \cM $ of cardinality $ M $, and any distribution $P$ over $ \cX\times [0,1] $, any prior distribution $\pi$ over $ [M] $ with $\pi_i>0$ for every $i\in[M]$, any $ \beta\geq c $, and any $\delta\in(0,1)$, it holds that with probability at least $ 1-\delta $ over the draw of a sample $ \rS\sim P^{n} $,
\begin{align*}
    \cR_P(f_{\rhohatQ}) \leq \min_{i\in [M]}\left\{ \cR_P(f_{i}) + \frac{\beta\log\prn{1/\pi_i}}{n} + \frac{2\beta\log\prn{1/\delta}}{n} \right\}.
\end{align*}
\end{theorem}

 For brevity, throughout the proof of \cref{thm:QBOB}, we denote $\cR_P^\star = \inf_{f\in\cF}\cR_P(f)$ but do not require the infimum to be attained.
We first state some preliminary facts.
\paragraph{Preliminaries.}
For the proof, it is convenient to work with the corresponding unconstrained regularized problem; that is, we allow the variable $\rho$ to range over $\mathbb{R}$. Write
\begin{align*}
\fexp := \Aexp(\rS'),
\qquad
\fmini := \Amini(\rS'),
\end{align*}
so these functions are \emph{random} with respect to $\rS'$.

For a realization of $\rS'$, define the conditional population quantities
\begin{align*}
\Delta' &= \cR_P(\fmini)-\cR_P(\fexp), \\
d'^2 &= \EE_{X\sim P_X}\!\left[\left(\fmini(X)-\fexp(X)\right)^2\right].
\end{align*}
Define the empirical quantities computed on $\rS$ (which also depend on $\rS'$) by
\begin{align*}
\Deltahat' &= \Rhat_\rS(\fmini) - \Rhat_\rS(\fexp), \\
\dhat'^2 &= \frac{1}{n}\sum_{i=1}^n \left(\fmini(X_i)-\fexp(X_i)\right)^2.
\end{align*}
The objective in \cref{alg:QBOB} can be rewritten as
\begin{align}
    \Psihat_{Q}(\rho) =& \frac{1}{2} \Rhat_\rS \!\left((1-\rho) \fexp + \rho \fmini\right) + \frac{1}{2}\left( (1-\rho)\Rhat_\rS \!\left( \fexp\right) +  \rho \Rhat_\rS \!\left( \fmini\right)\right) \nonumber
    \\&+ \frac{\beta(1-\rho)}{n}\log\prn{\tfrac{1}{\piexp}} + \frac{\beta\rho}{n} \log\prn{\tfrac{1}{\pimini}} \tag{Original objective} \\
    =& \frac{1}{2} \left( (1-\rho)\Rhat_\rS(\fexp) + \rho \Rhat_\rS(\fmini) - \rho(1-\rho)\dhat'^2 \right) + \frac{1}{2} (1-\rho)\Rhat_\rS(\fexp) + \frac{1}{2} \rho \Rhat_\rS(\fmini) \nonumber \\
    &+ \frac{\beta(1-\rho)}{n}\log\prn{\tfrac{1}{\piexp}} + \frac{\beta\rho}{n} \log\prn{\tfrac{1}{\pimini}} \tag{Expand squared risk} \\
    =& (1-\rho) \Rhat_\rS(\fexp) + \rho \Rhat_\rS(\fmini) - \frac{1}{2}\rho(1-\rho)\dhat'^2 + \frac{\beta(1-\rho)}{n}\log\prn{\tfrac{1}{\piexp}}
    + \frac{\beta\rho}{n} \log\prn{\tfrac{1}{\pimini}} \tag{Combine risk terms} \\
    =& \Rhat_\rS(\fexp) + \rho \Deltahat' - \frac{1}{2}(\rho - \rho^2)\dhat'^2+ \frac{\beta}{n}\log\prn{\tfrac{1}{\piexp}} + \rho \frac{\beta}{n} \log\prn{\tfrac{\piexp}{\pimini}} \tag{Substitute $\Deltahat'$ and rearrange} \\
    =& \frac{\dhat'^2}{2} \rho^2 + \left( \Deltahat' - \frac{\dhat'^2}{2} + \frac{\beta}{n}\log\prn{\tfrac{\piexp}{\pimini}} \right) \rho + \Rhat_\rS(\fexp) + \frac{\beta}{n}\log\prn{\tfrac{1}{\piexp}}. \label{eq:qbob-objective}
\end{align}
If $\dhat'^2 > 0$, the objective is strictly convex and its unique unconstrained minimizer over $\RR$ is
\begin{align}
\rhohatQun
=
\frac{1}{2} - \frac{\Deltahat' + \frac{\beta}{n}\log\prn{\frac{\piexp}{\pimini}}}{\dhat'^2}.
\label{eq:qbob-rhoun}
\end{align}
Since \cref{eq:qbob-objective} is a convex quadratic in $\rho$ with positive leading coefficient, the constrained minimizer is the projection of $\rhohatQun$ onto $[0,1]$:
\begin{align}
\rhohatQ
=
\Pi_{[0,1]}(\rhohatQun)
=
\min\{\max\{\rhohatQun,0\},1\}.
\label{eq:qbob-projection}
\end{align}
If $\dhat'^2 = 0$, then $\Deltahat' = 0$ and \cref{eq:qbob-objective} is affine with nonnegative slope $\frac{\beta}{n}\log\prn{\frac{\piexp}{\pimini}}$, so it keeps decreasing as $\rho\to -\infty$, and we interpret $ \rhohatQun=-\infty $. If $\pimini=\piexp$ the algorithm breaks ties in favor of the smaller $\rho$, and we again obtain $\rhohat_Q = 0$. Thus, \cref{eq:qbob-projection} remains valid in all cases.

We start with the universal rates guarantee of $ f_{\rhohatQ} $.

\subsection{Universal Exponential-Rate Guarantee for QBOB}
We will show that for any realization $ \rS' $ such that $ \cR_P(\Aexp(\rS'))\leq \cR_P^\star $ we have that
\begin{align*}
\PP_{\rS\sim P^n}\prn{\cR_P(f_{\rhohatQ}) > \cR_P^\star\mid \rS'} \leq \exp\left( - c_1n \max{\left\{\Delta', d'^2, \frac{\beta}{n}\log\prn{\frac{\piexp}{\pimini}} \right\}} \right).
\end{align*}
Define the event $ E= \{\cR_P(\Aexp(\rS'))\leq \cR_P^\star\} $. Having the above conditional bound then implies:
\begin{align*}
 &\PP_{(\rS',\rS)\sim P^{2n}}\prn{\cR_P(f_{\rhohatQ}) > \cR_P^\star} \\
 &= \EE_{\rS'\sim P^{n}}\brk{\PP_{\rS\sim P^{n}}\prn{\cR_P(f_{\rhohatQ}) > \cR_P^\star\mid \rS'}} \\
 &= \EE_{\rS'\sim P^{n}}\brk{\PP_{\rS\sim P^{n}}\prn{\cR_P(f_{\rhohatQ}) > \cR_P^\star\mid \rS'} \one_{E}}
 +
 \EE_{\rS'\sim P^{n}}\brk{\PP_{\rS\sim P^{n}}\prn{\cR_P(f_{\rhohatQ}) > \cR_P^\star\mid \rS'} \one_{E^{c}}} \\
 &\leq \EE_{\rS'\sim P^{n}}\left[\exp\left( - c_1n \max{\left\{\Delta', d'^2, \frac{\beta}{n}\log\prn{\frac{\piexp}{\pimini}} \right\}} \right)\right] + \PP_{\rS'\sim P^{n}}[E^{c}] \\
 &\leq \EE_{\rS'\sim P^{n}}\left[\exp\left( - c_1n \max{\left\{\Delta', d'^2, \frac{\beta}{n}\log\prn{\frac{\piexp}{\pimini}} \right\}} \right)\right] + C\exp{\prn{-c n}},
\end{align*}
which is the claimed bound in \cref{thm:QBOB} as the maximum is at least $\frac{\beta}{n}\log\prn{\frac{\piexp}{\pimini}}$.
Accordingly, assume for now that we have a realization of $ \rS' $ such that $ E $  holds.

\paragraph{Case $\cR_P(\fmini)\leq \cR_P^\star$:}
By convexity we have that $\cR_P(f_{\rhohatQ})\leq (1-\rhohatQ) \cR_P(\fexp) + \rhohatQ \cR_P(\fmini) \leq \cR_P^\star$, since the event $E$ implies that $ \cR_P(\fexp) \leq \cR_P^\star $ and the case considered implies $ \cR_P(\fmini)\leq \cR_P^\star $, completing the proof in this case.

\paragraph{Case $\cR_P(\fmini) >   \cR_P^\star$:}
This case, in combination with the event $E$, implies that $\cR_P(\fexp)\leq \cR_P^\star <\cR_P(\fmini)$. By continuity, there must exist a $\rho\in[0,1]$ such that $\cR_P(f_\rho)= \cR_P^\star$. Hence, we may solve for $\cR_P(f_{\rho})=\cR_P^\star$. For $f_\rho = (1-\rho) \fexp + \rho \fmini$, expanding the squared-loss gives:
\begin{align*}
    \cR_P(f_\rho) = (1-\rho) \cR_P(\fexp) + \rho \cR_P(\fmini) - \rho(1-\rho) d'^2.
\end{align*}
Here $ d'^{2}\neq 0 $, as $ d'^{2}=0 $ would imply $ \fexp=\fmini $ ($P$-almost surely) and  $ \cR_P(\fexp) = \cR_P(\fmini) $, contradicting the case assumption that $ \cR_P(\fmini) > \cR_P^\star \geq \cR_P(\fexp) $.
Solving $\cR_P(f_\rho) = \cR_P^\star$ gives:
\begin{align*}
    (1-\rho) \cR_P(\fexp) + \rho \cR_P(\fmini) - \rho(1-\rho) d'^2 = \cR_P^\star.
\end{align*}
Rearranging and grouping by powers of $\rho$ yields a quadratic equation:
\begin{align*}
    d'^2 \rho^2 + (\Delta' - d'^2) \rho + (\cR_P(\fexp) - \cR_P^\star) = 0,
\end{align*}
where we have used $\Delta' = \cR_P(\fmini) - \cR_P(\fexp)$.

Solving this quadratic equation for $\rho$ yields two solutions:
\begin{align*}
    \rho = \frac{d'^2 - \Delta' \pm \sqrt{(\Delta' - d'^2)^2 - 4 d'^2 (\cR_P(\fexp) - \cR_P^\star)}}{2d'^2}.
\end{align*}
We are interested in the rightmost solution (the $+$ branch), which can be simplified to:
\begin{align*}
    \rho_+ &= \frac{1}{2} - \frac{\Delta'}{2d'^2} + \sqrt{ \left( \frac{\Delta' - d'^2}{2d'^2} \right)^2 - \frac{\cR_P(\fexp) - \cR_P^\star}{d'^2} } \\
           &= \frac{1}{2} - \frac{\Delta'}{2d'^2} + \sqrt{ \left( \frac{1}{2} - \frac{\Delta'}{2d'^2} \right)^2 - \frac{\cR_P(\fexp) - \cR_P^\star}{d'^2} }.
\end{align*}
Since $\cR_P(\fexp) \leq \cR_P^\star$ by the definition of the event $E$, the term under the square root is guaranteed to be non-negative, ensuring that $0\leq \rho_+$ is a real number. As $ \cR_P^\star<\cR_P(\fmini)=\cR_P(f_{1}) $, we have $ \rho_+<1 $.
Furthermore, since $\cR_P(f_0)=\cR_P(\fexp) \leq \cR_P^\star=\cR_P(f_{\rho_{+}})$ and $ \cR_P(f_\rho) $ is a quadratic with positive leading coefficient, we have that for any $ \rho \in [0, \rho_{+}] $ it holds that $ \cR_P(f_\rho) \leq \cR_P^\star $.
We also know that $ \rhohatQ = \min\{\max\{\rhohatQun,0\},1\} $. It therefore suffices to prove that $\rhohatQun\leq\rho_+$, as it implies $\rhohatQ\in[0,\rho_{+}]$.
Indeed, if $ \rhohatQun\leq 0$, then $ \rhohatQ=0\leq \rho_{+} $ as $ \rho_{+}\geq0 $ , and if $ \rhohatQun> 0$, $\rhohatQ\leq \rhohatQun\leq \rho_{+} $.
We will show that this condition is satisfied with probability at least $ 1 - \exp( - c_1n \max{\{\Delta', d'^2, \frac{\beta}{n}\log\prn{\piexp/\pimini} \}} ) $.

To this end, we will show that the following condition implies $\rhohatQun \leq \rho_{+}$:
\begin{align*}
 \Deltahat' - \dhat'^2 \left( \frac{1}{2} - \rho_{+} \right) + \frac{\beta}{n}\log\prn{\frac{\piexp}{\pimini}} &\geq 0,
\end{align*}
and that it holds with probability at least $ 1 - \exp\prn{-nc_1\max{\left\{\Delta', \frac{\beta}{n}\log\prn{\frac{\piexp}{\pimini}}, d'^2 \right\}} } $.

To see that this event implies $\rhohatQun \leq \rho_{+}$, we recall that when $ \dhat'^{2}\neq 0 $, we have $ \rhohatQun = \frac{1}{2} - \prn{\Deltahat' + \frac{\beta}{n}\log\prn{\frac{\piexp}{\pimini}}}/\dhat'^2 $. Rearranging the above, we can derive the condition $\rhohatQun \leq \rho_{+}$ as follows:
\begin{align*}
    \Deltahat' - \dhat'^2 \left( \frac{1}{2} - \rho_{+} \right) + \frac{\beta}{n}\log\prn{\frac{\piexp}{\pimini}} \geq 0
    &\implies \dhat'^2 \left( \frac{1}{2} - \rho_{+} \right) \leq \Deltahat' + \frac{\beta}{n}\log\prn{\frac{\piexp}{\pimini}} \\
    &\implies \frac{1}{2} - \rho_{+} \leq \frac{\Deltahat' + \frac{\beta}{n}\log\prn{\frac{\piexp}{\pimini}}}{\dhat'^2} \\
    &\implies \rhohatQun=\frac{1}{2} - \frac{\Deltahat' + \frac{\beta}{n}\log\prn{\frac{\piexp}{\pimini}}}{\dhat'^2} \leq \rho_{+}.
\end{align*}
Recall that when $ \dhat'^{2}=0 $, we have $ \rhohatQun = -\infty $ if $\piexp > \pimini$,  $\rhohatQun = \infty$ if $\piexp < \pimini$, and any value if $\piexp = \pimini$; in the latter case we choose $\rhohatQun = -\infty$. Since we assume that the prior is such that $ \piexp\geq \pimini $, we have $ \rhohatQun \leq 0 $ when $ \dhat'^{2}=0 $, implying that $ \rhohatQ = 0 \leq \rho_{+} $, as concluded earlier. From these observations, we conclude that the condition $ \Deltahat' - \dhat'^2 \prn{ \frac{1}{2} - \rho_{+} } + \frac{\beta}{n}\log\prn{\frac{\piexp}{\pimini}} \geq 0 $ implies that $ \rhohatQ \leq \rho_{+} $, which in turn implies that $ \cR_P(f_{\rhohatQ}) \leq \cR_P^\star $, so it suffices to show that the above condition holds with probability at least $ 1 - \exp\prn{ - c_1n \max{\crl{\Delta', d'^2, \frac{\beta}{n}\log\prn{\frac{\piexp}{\pimini}} }} } $, which is what we do next.

By substituting the explicit definition of $\rho_{+}$, the subtracted term simplifies by canceling the leading $\frac{1}{2}$:
\begin{align*}
    \gamma := \frac{1}{2} - \rho_{+} = \frac{\Delta'}{2d'^2} - \sqrt{ \left( \frac{1}{2} - \frac{\Delta'}{2d'^2} \right)^2 - \frac{\cR_P(\fexp) - \cR_P^\star}{d'^2} }.
\end{align*}
Inserting this back into the inequality, the fully expanded empirical event becomes:
\begin{align*}
    \Deltahat' - \dhat'^2 \gamma + \frac{\beta}{n}\log\prn{\frac{\piexp}{\pimini}} &\geq 0.
\end{align*}
We recall that $ \rho_{+} $ in the case we consider is in $ [0,1] $, which implies that $ \gamma \in [-1/2, 1/2] $.
Define the combined random variable $W = Z - \gamma U$, where:
\begin{align*}
    Z(X, Y) &= (\fmini(X) - Y)^2 - (\fexp(X) - Y)^2, \\
    U(X) &= (\fmini(X) - \fexp(X))^2.
\end{align*}
Respectively, $\Deltahat'$ and $\dhat'^2$ are empirical averages of these variables. Hence, the empirical average of $ W $ is $ \Deltahat' - \gamma\dhat'^2$. We bound the range and variance of $W$. By expansion:
\begin{align*}
    W(X, Y) &= Z(X, Y) - \gamma U(X) \\
    &= \left( (\fmini(X) - Y)^2 - (\fexp(X) - Y)^2 \right) - \gamma(\fmini(X) - \fexp(X))^2 \\
    &= (\fmini(X) - Y - (\fexp(X) - Y))(\fmini(X) - Y + \fexp(X) - Y) - \gamma(\fmini(X) - \fexp(X))^2 \\
    &= (\fmini(X) - \fexp(X))\brk{\fmini(X) + \fexp(X) - 2Y} - \gamma(\fmini(X) - \fexp(X))^2 \\
    &= (\fmini(X) - \fexp(X)) \left[ (\fmini(X) + \fexp(X) - 2Y) - \gamma(\fmini(X) - \fexp(X)) \right] \\
    &= (\fmini(X) - \fexp(X)) \left[ (1-\gamma)\fmini(X) + (1+\gamma)\fexp(X) - 2Y \right].
\end{align*}
Because $\fmini, \fexp, Y \in [0,1]$ and $-1/2\leq \gamma \leq 1/2$, the magnitude of the second factor is bounded by 2, so
\begin{align*}
    |W| &\leq 2 |\fmini(X) - \fexp(X)| \\
    W^2 &\leq 4 (\fmini(X) - \fexp(X))^2 = 4 U(X).
\end{align*}
Therefore, the variance and the sup norm satisfy
\begin{align*}
    \operatorname{Var}(W) &\leq \EE[W^2] \leq 4\EE[U] = 4d'^2 \\
   \|W\|_{\infty} &\leq 2.
\end{align*}

We want to bound the failure probability $\PP\left( \Deltahat' - \gamma\dhat'^2 < -\frac{\beta}{n}\log\prn{\frac{\piexp}{\pimini}}\right)$. By subtracting the empirical mean $  \Deltahat' - \gamma\dhat'^2 $  from the true mean $ \EE[W]= \Delta' - d'^2\gamma  $ , we cast this into the one-sided bound form:
\begin{align*}
    \PP\left( \Deltahat' - \gamma\dhat'^2 < -\frac{\beta}{n}\log\prn{\frac{\piexp}{\pimini}} \right) &= \PP\left(\Delta' - d'^2\gamma   - (\Deltahat' - \gamma\dhat'^2) > \Delta' - d'^2\gamma  + \frac{\beta}{n}\log\prn{\frac{\piexp}{\pimini}} \right).
\end{align*}
The following calculation shows that we may apply Bernstein's inequality:
\begin{align}
 &\Delta' - d'^2\gamma  + \frac{\beta}{n}\log\prn{\frac{\piexp}{\pimini}} \nonumber
 \\
 &= \Delta' - d'^2 \left( \frac{\Delta'}{2d'^2} - \sqrt{ \left( \frac{1}{2} - \frac{\Delta'}{2d'^2} \right)^2 - \frac{\cR_P(\fexp) - \cR_P^\star}{d'^2} } \right) + \frac{\beta}{n}\log\prn{\frac{\piexp}{\pimini}} \nonumber \\
 &= \Delta' - \frac{\Delta'}{2} + \sqrt{ d'^4 \left( \frac{(d'^2 - \Delta')^2}{4d'^4} - \frac{\cR_P(\fexp) - \cR_P^\star}{d'^2} \right) } + \frac{\beta}{n}\log\prn{\frac{\piexp}{\pimini}} \nonumber \\
 &= \frac{\Delta'}{2} + \frac{1}{2}\sqrt{ (d'^2 - \Delta')^2 - 4d'^2 (\cR_P(\fexp) - \cR_P^\star) } + \frac{\beta}{n}\log\prn{\frac{\piexp}{\pimini}}> 0.\label{eq:QBOB1}
\end{align}
The last step holds because $\Delta'>0$ in the case we consider.
Applying Bernstein's inequality for a zero mean random variable $X$, yields
$$
\PP(X > t) \leq \exp\left(-\frac{t^2}{2\operatorname{Var}(X) + \frac{2}{3}\|X\|_{\infty}t}\right),
$$
implying that
\begin{align*}
      \PP\left( \Deltahat' - \gamma\dhat'^2 < -\frac{\beta}{n}\log\prn{\frac{\piexp}{\pimini}} \right) &= \PP\left(\Delta' - d'^2\gamma   - (\Deltahat' - \gamma\dhat'^2) > \Delta' - d'^2\gamma  + \frac{\beta}{n}\log\prn{\frac{\piexp}{\pimini}} \right)\\
    &\leq \exp\left( - n\frac{ \left(\Delta' - d'^2\gamma  + \frac{\beta}{n}\log\prn{\frac{\piexp}{\pimini}} \right)^2}{8d'^2 + \frac{8}{3}\left( \Delta' - d'^2\gamma  + \frac{\beta}{n}\log\prn{\frac{\piexp}{\pimini}} \right)} \right). \tag{By $\operatorname{Var}(W-\EE[W])=\operatorname{Var}(W)$ and $\|W-\EE[W]\|_{\infty}\leq 2\|W\|_{\infty}$}
\end{align*}
We now claim that regardless of the relative values of $\Delta'$ and $d'^2$, $  \left( \Delta' - d'^2\gamma  + \frac{\beta}{n}\log\prn{\frac{\piexp}{\pimini}} \right) $  and $ 8d'^2 + \frac{8}{3}\left( \Delta' - d'^2\gamma  + \frac{\beta}{n}\log\prn{\frac{\piexp}{\pimini}} \right) $ are comparable up to a universal constant, so the term inside the exponential can be chosen as the latter multiplied by a constant and $-n$. Since  $ 8d'^2 + \frac{8}{3}\left( \Delta' - d'^2\gamma  + \frac{\beta}{n}\log\prn{\frac{\piexp}{\pimini}} \right)\geq \max\{\Delta', \frac{\beta}{n}\log\prn{\frac{\piexp}{\pimini}}, d'^{2} \} $, we can conclude that the probability of the above event for a sufficiently small $ c_1 $ is at most $ \exp\prn{-nc_1\max{\left\{\Delta', \frac{\beta}{n}\log\prn{\frac{\piexp}{\pimini}}, d'^2 \right\}} } $, as claimed and conclude the proof.

We consider two cases.
In the case that $d'^2 \leq 2 \Delta'$, we have that
\begin{align*}
 \Delta' - d'^2\gamma  + \frac{\beta}{n}\log\prn{\frac{\piexp}{\pimini}}
 &\leq 8d'^2 +  \frac{8}{3}\left(\Delta' - d'^2\gamma  + \frac{\beta}{n}\log\prn{\frac{\piexp}{\pimini}}\right)
 \\
 &\leq \frac{8+4\cdot8\cdot 3}{3}\left(\Delta' - d'^2\gamma  + \frac{\beta}{n}\log\prn{\frac{\piexp}{\pimini}}\right)
 \\
 &= \frac{104}{3}\left( \Delta' - d'^2\gamma  + \frac{\beta}{n}\log\prn{\frac{\piexp}{\pimini}} \right).
\end{align*}
In the case that $d'^2 > 2 \Delta'$, we have that
\begin{align*}
 d'^2-\Delta' \geq d'^2/2
& \implies (d'^2 - \Delta')^2 \geq d'^4/4
 \\
 &\implies \sqrt{ (d'^2 - \Delta')^2 - 4d'^2 (\cR_P(\fexp) - \cR_P^\star) } \geq d'^2/2,
\end{align*}
which implies the required comparability of the numerator and denominator:
\begin{align*}
 &\Delta' - d'^2\gamma  + \frac{\beta}{n}\log\prn{\frac{\piexp}{\pimini}}\\
 &\leq 8d'^2 +  \frac{8}{3}\left(\Delta' - d'^2\gamma  + \frac{\beta}{n}\log\prn{\frac{\piexp}{\pimini}}\right)
 \\
 &\leq 8d'^2 + \frac{8}{3}\left( \frac{\Delta'}{2} + \frac{1}{2}\sqrt{ (d'^2 - \Delta')^2 - 4d'^2 (\cR_P(\fexp) - \cR_P^\star) } + \frac{\beta}{n}\log\prn{\frac{\piexp}{\pimini}} \right) \tag{by \cref{eq:QBOB1}}
 \\
 &\leq \frac{8+4\cdot 8\cdot 3}{3}\left( \frac{\Delta'}{2} + \frac{1}{2}\sqrt{ (d'^2 - \Delta')^2 - 4d'^2 (\cR_P(\fexp) - \cR_P^\star) } + \frac{\beta}{n}\log\prn{\frac{\piexp}{\pimini}} \right)
 \\
&=\frac{104}{3}\left( \Delta' - d'^2\gamma  + \frac{\beta}{n}\log\prn{\frac{\piexp}{\pimini}}  \right) \tag{by \cref{eq:QBOB1}}.
\end{align*}
These two cases prove the claimed constant-factor comparison between the denominator and numerator, and concludes the proof of the exponential rate guarantee of $ f_{\rhohatQ} $.

\subsection{Minimax Guarantee for QBOB}
We know that for $\delta\in(0,c)$ with probability at least $1-\delta/2$ over $ \rS' $ we have that
\begin{align*}
    \cR_P(\fmini) \leq \cR_P^\star + r_{n}(\delta/2).
\end{align*}
Furthermore, for any realization $ S' $ of $ \rS' $, we have that $ f_{\rhohatQ} $ is the output of the $ Q $-aggregation procedure with sample $ \rS $ over the hypothesis class $\{\fexp,\fmini\}$ and $ \beta \geq c_2 $. Thus \cref{thm:Qaggregationrestatementlecue2014optimal} implies that with probability at least $ 1-\delta/2 $ over $ \rS\sim P^{n} $ it holds that
\begin{align*}
    \cR_P(f_{\rhohatQ}) &\leq \min_{\vartheta\in\{\exp,\operatorname{mini}\}}\left\{ \cR_P(f_\vartheta) + \frac{\beta\log\prn{1/\pi_{\vartheta}}}{n} + \frac{2\beta\log\prn{2/\delta}}{n} \right\}
    \\
    &\leq  \cR_P(\fmini) + \frac{\beta\log\prn{1/\pimini}}{n} + \frac{2\beta\log\prn{2/\delta}}{n}.
\end{align*}
By independence of $ \rS' $ and $ \rS $, we can apply a union bound to conclude that with probability at least $ 1-\delta $ over the draw of both $ \rS' $ and $ \rS $ both of the above events occur. On their intersection, we have that
\begin{align*}
    \cR_P(f_{\rhohatQ}) \leq \cR_P^\star + r_{n}(\delta/2) + \frac{\beta\log\prn{1/\pimini}}{n} + \frac{2\beta\log\prn{2/\delta}}{n}.
\end{align*}
This completes the proof of the minimax guarantee for $ f_{\rhohatQ}$.

\stopcontents[appendix]

\end{document}